\documentclass[11pt, one side,emlines,reqno]{amsart}
\usepackage[a4paper,margin=1in]{geometry}
\usepackage{amssymb,latexsym,xy,eucal,mathrsfs,tikz}
\usetikzlibrary{patterns,snakes}

\textwidth=18cm \textheight=27cm \theoremstyle{plain}
\newtheorem{theorem}{Theorem}[section]
\newtheorem{thm}[theorem]{Theorem}
\newtheorem{lemma}[theorem]{Lemma}
\newtheorem{result}[theorem]{Main Theorem}

\newtheorem{corollary}[theorem]{Corollary}
\newtheorem{definition}[theorem]{Definition}

\newtheorem{Rm}[theorem]{Remark}
\newtheorem{Not}[theorem]{Notation}

\newtheorem{Con}[theorem]{Conjecture}
\numberwithin{equation}{section} 

\begin{document}
\baselineskip 18 truept
\title{Decompositions of $n$-Cube into $2^mn$-Cycles}
\date{}
\author{S. A. Tapadia, B. N. Waphare and Y. M. Borse }
\address{\rm Department of Mathematics, Savitribai Phule Pune University, Pune-411007, India.}
\email{\emph{tapadiasandhya@gmail.com,
waphare@yahoo.com, ymborse11@gmail.com}}
 
 \maketitle

 \noindent {\bf \small Abstract:} {\small It is known that the $n$-dimensional hypercube $Q_n,$ for $n$ even,  has a decomposition into $k$-cycles  for  $k=n, 2n,$ $2^l$ with $2 \leq l \leq n.$ In this paper, we prove that $Q_n$ has a decomposition into $2^mn$-cycles for $n \geq 2^m.$ As an immediate consequence of this result, we get path decompositions of $Q_n$ as well. This gives a partial solution to a conjecture posed by Ramras and also, it solves some special cases of a conjecture due to Erde.}

\section{Introduction}
 Let $Q_n$ denote the graph of the $n$-cube, so that $V(Q_n)$ is the set of $2^n$ binary $n$-tuples, and $E(Q_n)$ consists of those pairs of vertices which differ in exactly one co-ordinate. Therefore $Q_n$ is an $n$-regular graph with  $2^n$ vertices and $n2^{n-1}$ edges.

A \textit{decomposition} of a graph $G$ is a collection of edge-disjoint
subgraphs $H_1, H_2, \ldots, H_r$ of $G$ such that every edge of $G$ belongs
to exactly one $H_i.$ If all the subgraphs in the decomposition of $G$ are isomorphic to a graph $H$, we say that $G$ can be \textit{decomposed} into $H$ or $H$ \textit{decomposes} $G.$ By  \textit{$k$-cycle (respectively $k$-path)}, we mean the cycle (respectively path) with $k$ edges.

Clearly, to get a decomposition of $Q_n$ into cycles it is necessary that $n$ must be  even.  In what follows we assume that $n$ is an even integer.

\vskip.2cm
Due to various applications in parallel processing and interconnection networks, decompositions of the hypercube $Q_n$ into cycles have been studied extensively in the literature; see [1,2,4,6,7,9].  Alspach et al. \cite{Als} proved that $Q_n$ has a decomposition into hamiltonian cycles; whereas Song \cite{Son}  attempted to give an explicit construction of hamiltonian cycles in such a decomposition. Horak et al. \cite{Hor} proved that $Q_n$ can be decomposed into any graph $H$ of size $n,$ each of whose blocks is either an even cycle or an edge. 
\vskip.2cm
Following \cite{Ram}, a subset $F$ of $E(Q_n)$ is a \emph{fundamental set} for $Q_n$ with respect to a subgroup $\mathcal G$ of the automorphism group $Aut(Q_n)$ of $Q_n,$ if  $\{g(F)\colon g \in \mathcal G\}$ forms a decomposition of $Q_n.$

Ramras \cite {Mar} proved that the edge set of certain $n$-cycles of $Q_n$ are fundamental sets; while Mollard  and Ramras \cite{Mol} obtained a particular $2n$-cycle whose edge set forms a fundamental set for $Q_n.$ They further asked a question for which $k$ dividing $|E(Q_n)|=n2^{n-1},$ $Q_n$ has a decomposition into $k$-cycles or into $k$-paths. El-Zanati and Eynden \cite{Elz} proved that  $Q_n$ can be decomposed into $k$-cycles if $ k = 2^l$ with $2 \leq l \leq n.$

In this paper, we answer the question for $k=2^mn$ with $n \geq 2^m$ and $ m \geq 1.$  We obtain the following result.

\begin{theorem}
For $n \geq 2^m$  with $ m \geq 1,$ $n$ even, a $2^mn$-cycle decomposes the hypercube $Q_n.$ 
\end{theorem}
 Following result is an obvious corollary of the Theorem.
 \begin{corollary}
 Suppose $n \geq 2$ is even and $m$ is the largest integer such that $n\geq 2^m.$ Then $Q_n$ can be decomposed into $2^rn$-cycles, $r=1,2,\ldots,m.$
 \end{corollary}
 
Since a $k$-cycle is decomposable into paths of
length $r,$ where $r < k$ and $r$ divides $k,$ we get the following consequence for path decompositions of $Q_n.$

\begin{corollary}
For $n \geq 2^m$ with $ m \geq 1,$ $n$ even, a path of length $r$ decomposes $Q_n$ for $ r = n, 2n, \ldots, 2^{m-1}n.$
\end{corollary}

\vskip.2cm \indent This solves the cases $r = 2n,$ $ r = 4n,$ \ldots, $r = 2^{m-1}n$  of the following conjecture due to Erde \cite{Erd}, for $n \geq 2^m.$
\begin{Con}
 For $n$ even, $ r $ dividing $n2^{n-1}$ and $ r < 2^n,$ a path of length $r$ decomposes $Q_n.$
\end{Con}
 Also, the case $r=2n$ of Corollary 1.2 gives decomposition of $Q_n$ into paths of length $2n.$ So, the result solves partially the following conjecture by Ramras\cite{Ram}. 
\begin{Con}
For an even $n \geq 4,$ the edge set of a path of length $2n$ is a fundamental set for $Q_n$.
\end{Con}

We actually prove the following stronger result than Theorem 1.1. 
\begin{result}
Let $n \geq 2^m$ be an even integer with $ m\geq 1.$  Then the hypercube $Q_n$ can be decomposed into $2^mn$-cycles, (say) $C_1, C_2, \ldots, C_r,$ where $r = 2^{n-1-m}$ such that
\begin{itemize}
\item [(I)] every $C_i$ contains $2^m$ edges $e_{i1}, e_{i2}, \ldots, e_{i2^m}$ such that $C_i - \{e_{i1}, e_{i2}, \ldots, e_{i2^m}\}$ has $2^m$ components each of which is a path of length $n-1$;
\item[(II)] $M=\{e_{ij}\colon j = 1, 2, \ldots, 2^m; i = 1, 2, \ldots, r\}$ forms a perfect matching in $Q_n.$
\end{itemize}
\end{result}

We prove this theorem by induction on $n.$ The proof of the induction step is straight-forward. The major part of this paper is devoted to proving the basis step $n=2^m$. For the proof of the basis step, we create a lot of algebraic machinery. As $Q_n = Q_{n/2} \Box Q_{n/2}$, for $n=2^m$ we first find a decomposition of $Q_{n/2}$ into $n$-cycles as given by Mollard and Ramras in \cite{Mol}. We need to partition this collection of $n$-cycles into sub-collections consisting of vertex-disjoint cycles such that union of cycles in each sub-collection forms a spanning subgraph of $Q_{n/2}.$ For this purpose, we construct a particular subgroup of $Aut(Q_n)$ in Section 2. Then, by taking the Cartesian product of these $n$-cycles in the decomposition of $Q_{n/2}$ in a particular manner, we construct the $2^mn$-cycles in $Q_n$. This construction is given in Section 3. We  give a collection of $n$ edges from every $2^mn$-cycle in the decomposition of $Q_n$ satisfying condition (I). The proof  of the fact  that these edges also satisfy condition (II), is given in Section 5.  In Section 4, we give the proof of the induction step.
We illustrate construction of $64$-cycles in the decomposition of $Q_8$ and selection of edges from these cycles to form a perfect matching of $Q_8$ in the Appendix provided at the end.  

\section{Construction of the subgroup }

 For a positive integer $n,$ we use the notation $[n]$ for the set $\{1,2,\ldots,n\}$ and consider $V(Q_n),$ the vertex set of the hypercube $Q_n,$ as the power set $ \mathcal{P}([n])$. Two vertices $A$ and $B$ of $Q_n$ are adjacent if and only if $|A \Delta B| = 1,$ where $\Delta$ denotes the symmetric difference of the sets $A$ and $B.$  The  \emph{direction} of an edge $e = (A,B)$  of $Q_n$ is $i$ if $A\Delta B =\{ i\}.$ A cycle $C=U_1-U_2-U_3-\ldots- U_k - U_1$ in $Q_n$ can be expressed in terms of a vertex $U_1$ of the cycle and edge-direction sequence $S=(i_1, i_2, \ldots, i_k)$ in the sense that the cycle $C=C(U,S),$ where $U=U_1,$ 
 $U_j=U_1\Delta\{i_1\},\Delta\{i_2\}\Delta\ldots\Delta\{i_{j-1}\}$ for $j=2,3,\ldots,k,$ and $U_k\Delta\{i_k\}=U_1.$ For example, the cycle $ C = \emptyset-\{1\}-\{1,2\}- \{2\}-\emptyset$ in $Q_4$ can be written as $C(\emptyset, S)$  with $ S = (1,2,1,2).$ 
 
 For $A \subset [n],$ the map $\sigma_A: \mathcal{P}([n]) \rightarrow \mathcal{P}([n])$ defined by $\sigma_A(B)=A\Delta B$ is an automorphism of $Q_n.$  For a  subgraph $W$ of $Q_n,$ let  $\sigma_A(W)$ denote  the  image of $W$ under the map $\sigma_A.$


We state an obvious lemma regarding the action of an automorphism $\sigma_A$ on a cycle in $Q_n.$
\begin{lemma}
If $A \subseteq [n],$ and  $C = C(U,S)$ is a cycle in $Q_n,$ then $\sigma_A(C)$ is the cycle $ C(A\Delta U, S)$ in $Q_n.$
\end{lemma}



Let $ G=\{A \subseteq [n-1]\colon |A|~even\}$ and let $\mathcal G=\{\sigma_A\colon A \in G\}.$
Then $G$ is a subgroup of the group $(\mathcal {P}([n]);\Delta)$ and $\mathcal G$ is a subgroup of Aut$(Q_n).$ Let
$C=C(\emptyset,S)$ be the  $2n$-cycle in $Q_n$ with the initial vertex $\emptyset$ and the edge direction sequence $S=(1,2,\ldots,n,1,2,\ldots,n).$ Then by Lemma 2.1, the collection $\mathcal C =\{\sigma_A(C)\colon \sigma_A \in \mathcal G\}=\{\sigma_A(C)\colon A \in G\} = \{C(A, S) \colon A \in G\}$ is a collection of $2n$-cycles.

\begin{thm} [\cite{Mol}]
	The members of the above collection $\mathcal C$ decompose the hypercube $Q_n.$ 
	\end{thm}

Suppose $n$ is a power of $2.$ We will partition the collection $\mathcal C$ into $n/2$ sub-collections (say), $\mathcal C_1,\mathcal C_2,\ldots,\mathcal C_{n/2},$  in such a way that the $2n$-cycles belonging to each sub-collection are mutually vertex-disjoint and the union of these cycles in each $\mathcal C_i$ forms a spanning subgraph of $Q_n.$ To get this partition, we construct a particular subgroup $H$ of $G.$ The cosets of $H$ partition $G,$ which in turn gives the desired partition of $\mathcal C.$

\textbf{Construction of the subgroup $\bf H:$} Let $n=2^m,$ with $m \geq 2$  and $G=\{A \subseteq [n-1]\colon |A|~ even\}.$ Then $(G; \Delta)$ is a group with $2^{n-2}$ elements. Let $H$ be the subgroup of $(G;\Delta)$ which is generated by the collection of $2$-element subsets given by
 $$\mathsf K =\{\{j,2^{i-1}+j\}\colon j=1,2, \ldots, 2^{i-1}-1,~i=2,3,\ldots,m\}.$$
  For convenience, we write $K_1=\emptyset$ and for $i=2,3,\ldots,m,$ 
	$$K_i = \{\{j,2^{i-1}+j\}\colon j=1,2,\ldots,2^{i-1}-1\}.$$ Note that $K_r \cap K_s = \emptyset$ whenever $r \neq s.$
	


In particular, for $n=2^3=8,$ $K_2=\{\{1,3\}\},$  $K_3=\{\{1,5\},\{2,6\},\{3,7\}\}$ and so $K=\{\{1,3\},\{1,5\},\{2,6\},\{3,7\}\}.$ Therefore  $H=<\{1,3\},\{1,5\},\{2,6\},\{3,7\}>=\{\emptyset,\{1,3\},\{1,5\},\linebreak\{1,7\}, \{2,6\},\{3,5\}, \{3,7\},  \{5,7\}, \{1,2,3,6\}, \{1,2,5,6\},\{1,2,6,7\}, \{1,3,5,7\},\{2,3,5,6\},\linebreak\{2,3,6,7\}, \{2,5,6,7\}, \{1,2,3,5,6,7\}\}.$ So $|H|=16.$ 
\vskip .5cm
We prove the following lemma to get cosets of $H$ in $G.$ We call a set of consecutive integers a \textit{consecutive string}.
\begin{lemma}
Let $n=2^m$ with $ m\geq 2 $ and let $G, H, K_i,~ 2\leq i \leq m$  be as above. Then 
\begin{enumerate}
 \item No element in $H$ forms a consecutive string in $[n]$  and $|H| = 2^{n-m-1}.$
 \item $ H_1 = H,$ $H_2=\{1,2\} \Delta H,$ $H_3=\{1,2,3,4\}\Delta H,$ \ldots, $H_{n/2}=\{1,2,\ldots,n-2\}\Delta H$ are precisely the cosets of $H$ in $G.$ 
 \end{enumerate}
\end{lemma}

\begin{proof}

(1). We prove the result by using induction on $m.$ If $m=2,$ then $G = \{\emptyset, \{1,2\},\{1,3\},\{2,3\}\}$ and $H = \{\emptyset, \{1,3\}\}.$ Hence the result is true for $ m = 2.$ Suppose $m > 2.$ Assume that the result is true for $m-1.$ 
We prove the result for $m.$ Let $A \in H.$ Then $A = A_1\Delta A_2 \Delta \ldots \Delta A_r,$ for some $A_1, A_2, \ldots, A_r \in \mathsf K.$  If  each  $A_i \subset [2^{m-1}  -1], $ then, by induction, $A$ is not a consecutive string.  Suppose $A_i \in K_m$ for some $i.$ We may assume that $A_i \in K_m$  for $1\leq i\leq l$ and $ A_j \notin K_m$  for $ l+1 \leq j \leq r.$   Then  $A_i = \{x_i,2^{m-1}+x_i\}$ for some $x_i \in [2^{m-1}-1].$ Since the members of $K_m$ are mutually disjoint, $A_1\Delta A_2 \Delta\ldots\Delta A_l = \displaystyle \bigcup_{i=1}^{l} A_i = \displaystyle \bigcup_{i=1}^{l} \{x_i,2^{m-1}+x_i\}=\{x_1,x_2,\ldots,x_l\}\cup\{x_1+2^{m-1},x_2+2^{m-1},\ldots,x_l+2^{m-1}\}.$ Let  $B = A_{l+1}\Delta A_{l+2} \Delta \ldots \Delta A_{r}.$  By induction, $B$ is not a consecutive string. As  $ B \subset [2^{m-1} - 1],$ $ 2^{m-1}+x_i \notin B$ for $1\leq i \leq l.$ Therefore  $A = A_1\Delta A_2 \Delta\ldots\Delta A_l\Delta B= (\{x_1,x_2,\ldots,x_l\}  B)  \cup \{2^{m-1}+x_1,2^{m-1}+x_2,\ldots,2^{m-1}+x_l\}.$
If  $\{x_1,x_2,\ldots,x_l\}\Delta B = \emptyset,$  then $B = \{x_1,x_2,\ldots,x_l\}$ and hence  $A = \{2^{m-1}+x_1,2^{m-1}+x_2,\ldots,2^{m-1}+x_l\}$ is not a consecutive string as $B$ is not a consecutive string. If  $y \in \{x_1,x_2,\ldots,x_l\}\Delta B,$ then $y < 2^{m-1} <  2^{m-1} + x_i $ for $ 1\leq i \leq l.$ Hence $A$ is not a consecutive string.

 Since  $H$ is a finite abelian group generated by ${n-m-1}$ pairwise disjoint sets each of which has cardinality  two, $|H| = 2^{n-m-1}.$

(2). As $ G = \{ A\colon A \subset [n-1], |A|~even\}$ and by (2) no member of the subgroup $H$ is a consecutive string, $ \{1, 2, \ldots, 2i-2\} \in G - H$ for $ i = 2, \ldots, n/2.$ Hence  $ H_1 = H,$ $H_2=\{1,2\}\Delta H,$ $H_3=\{1,2,3,4\}\Delta  H,$ \ldots, $H_{n/2}=\{1,2,\ldots, n-3,n-2\}\Delta H$ are cosets of $H$ in $G.$ If  $i < j,$ then $H_i = H_j$ implies that $\{2i-1,2i,\ldots,2j-2\} \in H,$ which is a contradiction by (2). Hence  all these $n/2$  cosets are distinct. Further, by (2), $|H|\times n/2  =  2^{n-m-1} \times 2^{m-1}=2^{n-2} = |G|  = |H|\times |G/H|.$ This gives $ |G/H| = n/2.$ Thus $ H_1, H_2, \dots, H_{n/2}$ are the only cosets of $H$ in $G.$ 
\end{proof}
\begin{lemma}
Let $n=2^m$ with $ m\geq 2 $ and  $G, H_i ,$ $1\leq i \leq n/2$ be as in Lemma 2.3  and $ C = C(\emptyset, S)$ be the $2n$-cycle in $Q_n$ with edge direction sequence $ S = (1, 2, \ldots, n, 1, 2, \dots, n).$  Then, for each $ i, $  the members of the collection $\{ C(A, S) \colon A\in H_i\}$ are mutually vertex-disjoint $2n$-cycles and their union is a spanning 2-regular subgraph of $Q_n.$  
\end{lemma}
\begin{proof}
Let us recall that $G = \{A \subset [n-1]\colon |A|~{\rm even}\}$ is a group with respect to symmetric difference $\Delta$ and $H$ is the subgroup of $G$ generated by the collection of $2$-element sets given by \linebreak
 ${ \mathsf K =\{\{j,2^{i-1}+j\}\colon j=1,2, \ldots, 2^{i-1}-1,~i=2,3,\ldots,m\}}.$ By Lemma 2.3 (2), $H$ does not contain a consecutive string and further, $H_1, H_2, \ldots, H_{n/2}$ are cosets of $H$ in $G,$ where $H_1=H$ and $H_i = [2i-2]\Delta H$ for $ 2\leq i \leq n/2.$  
By Theorem 2.2,  ${\mathcal C} = \{\sigma_A(C)\colon A \in G\}=\{C(A, S)\colon A \in G\}$ is the collection of edge-disjoint $2n$-cycles that decompose $Q_n.$  Let 
${\mathcal C}_i = \{\sigma_A(C)\colon A \in H_i\}=\{ C(A, S)\colon A\in H_i\}$ for $i = 1, 2, \ldots, n/2.$ Then ${\mathcal C}_1, {\mathcal C}_2, \ldots, {\mathcal C}_{n/2}$ partition ${\mathcal C}.$

We prove that  any two members of ${\mathcal C}_1$ vertex-disjoint. Consider two distinct members $C(A, S)$ and $C(B, S)$ of ${\mathcal C}_1.$ Then $ A \neq B$ and $ A, B \in H_1=H .$  Therefore $A\Delta B \in H $ and $ A\Delta B \neq \emptyset.$ Since the edge direction sequence of these cycles is  $S = (1, 2, \ldots, n, 1, 2, \dots, n),$  the vertices of the cycle  $C(A, S)$ are $A_0 = A, A_1, A_2, \ldots,  A_{2n-1}, A_{2n} =A$ such that  $ A_{j-1}\Delta A_j = \{j\}$ if $1\leq j \leq n $ and  $ A_{j-1}\Delta A_j = \{j-n\}$ if $n+1\leq j \leq 2n .$  By using an induction argument on $j,$ it follows that $ A_j = A \Delta [j]$ if $1\leq j \leq n $ and  $ A_j = A\Delta [n]\Delta[j-n]$ if $n+1\leq j \leq 2n .$  Hence $A = A_j\Delta[j]$ or $A = A_j\Delta [n]\Delta[j-n].$ Similarly, the vertices of $C(B,  S)$ are  $B_0 = B, B_1, B_2, \ldots, B_{2n-1}, B_{2n} =B,$ where  $B_j = B \Delta [j]$ if $1\leq j \leq n $ and  $  B_j = B \Delta  [n] \Delta [j-n]$ if $n+1\leq j \leq 2n .$ 

Suppose the cycles $C(A, S)$ and $C(B, S)$ share  a vertex. Then $ A_p = B_q$ for some $ 1\leq p, q \leq 2n.$ We may assume that $ p \leq q.$  If $ p =q,$ then $ A\Delta B = \emptyset,$ a contradiction. Hence $ p < q.$ If $ q\leq n,$   then $A\Delta B = \{p+1, p+2, \ldots, q\}$ is a consecutive string, a contradiction. Similarly, $n \leq p $  gives a contradiction. If $ p<n < q,$ then $n \in A\Delta B,$ a contradiction. 

Thus the members of ${\mathcal C}_1$ are mutually vertex-disjoint $2n$-cycles of $Q_n.$  Since $ |C_1| \times 2n = |H| \times 2n = 2^{n-m-1}\times 2n = 2^n,$ the union of the  of these $2n$-cycles is a 2-regular spanning subgraph of  $Q_n.$ 

Let $ 1< i \leq n/2.$ Consider ${\mathcal C}_i = \{C(A, S)\colon A\in H_i\} = \{C(A, S)\colon A\in [2i-2]H\} = \{ C([2i-2]B,S)\colon B \in H\} = \{\sigma_{[2i-2]} C(B, S)\colon B\in H\} = \{\sigma_{[2i-2]} (D)\colon D\in {\mathcal C}_1\}.$ Note that $\sigma_{[2i-2]}$ is an automorphism of $Q_n.$ Therefore like ${\mathcal C}_1,$ the members of ${\mathcal C}_i$ are  mutually vertex-disjoint $2n$-cycles and their union is a spanning 2-regular subgraph of $Q_n.$ 
\end{proof}
\vskip .2cm
\noindent
\textbf{More Properties of $H:$}
 \vskip .2cm

We prove some more properties of $H$ in the three lemmas given below. One can skip them till Section 5, as they are used there only.

\begin{lemma}
 Let $n=2^m$ and $H$ be the group as in Lemma 2.3. Then
\begin{enumerate} 
 \item No element of $H$ contains $2^m$ and $2^{m-1}.$
\item Suppose $A\in H$ contains $2^k,$ then for some $t$  it must contain $2^k+\sum_{i=1}^{t} 2^{k_i'},$ where $k_i'>k.$ 
 \item $\{1,2k+1\} \in H$, $1 \leq k \leq n/2-1.$
\end{enumerate}
\end{lemma}

\begin{proof}
Clearly, (1) and (2) follows from the definition of $H.$ Since $2k$ is even for all $k\geq 1,$ it can be written as the sum of distinct powers of $2.$ Let $2k = \sum_{i=1}^{t}2^{k_i}.$ Then $\{1,2k+1\}=\{1,2^{k_1}+1\}\Delta\{2^{k_1}+1,2^{k_1}+2^{k_2}+1\}\Delta\{2^{k_1}+2^{k_2}+1,2^{k_1}+2^{k_2}+2^{k_3}+1\}\Delta\ldots \Delta\{\sum_{i=1}^{t-1}2^{k_i}+1, \sum_{i=1}^{t}2^{k_i}+1\} \in H,$ as each pair involved in the symmetric difference on right hand side belongs to $H.$ This proves (3).

 \end{proof}

By Lemma 2.3(2), no member of $H$ is a consecutive string. However, a member of $H$ may be the union of two disjoint consecutive strings. In the following lemma, we prove some properties of such members of $H.$

\begin{Not} Let $A$ and $B$ be sets of integers. Then we define $A+2^k=\{a+2^k \colon a \in A\}.$  
We write $A < 2^k,$ if $a<2^k$ for all $a \in A,$ and similarly $A > 2^k,$ if $a > 2^k$ for all $a \in A.$ 
Also, by $A<B$ we mean that the largest element of $A$ is smaller that the smallest element of $B.$
\end{Not}
\begin{lemma}
Let $G$ and $H$ be the groups as above and let $A = A_1 \cup A_2 \in G,$ where $A_1$ and $A_2$ are disjoint consecutive strings. Then
 \begin{enumerate}
  \item if $A \in H,$ then there exists $B \in H$ such that $A\Delta B = S_1 \cup S_2,$ where $S_1$ and $S_2$ are consecutive strings such that $S_1 + 2^k = S_2$ for some $k$;
	\item if $A \in H$ and $A_1 = [r],$ then there exists a positive integer $k$ such that $A_1 < 2^k < A_2 < 2^{k+1}.$ 
 \item if $A_1 = [r]$ and $A_2 =\{ t+1, t+2, \dots, t+s \}$ for some $ s, t >  r \geq 1,$ then $A\notin {H};$ 
\item  if $A_1 = [2^k]$ and $A_2 =\{t +1, t +2, \dots, t +s \},$ where $t> 2^{k+1},$ $s < 2^k,$ then $ A\notin {H};$
\item  if $A_1 = [r]$ and $A_2 =\{ \sum_{i=1}^{t}2^{k_i} +1, \sum_{i=1}^{t}2^{k_i} +2, \dots, \sum_{i=1}^{t}2^{k_i} +s \},$ where $k_i < k_{i+1},$ $1 \leq s < r < 2^{k_1},$ then $ A \notin {H}.$  
  
 \end{enumerate}
\end{lemma}

\begin{proof}
 
(1). As every member of the group $G$ is a subset of $[2^m -1],$ $ A < 2^m.$ Let $k$ be the smallest integer such that $A < 2^k.$ Since $A\neq \emptyset$ and $|A|$ is even,  $ k\geq 2.$  We prove the result by induction on $k.$  Suppose $ A \in H.$ Suppose $k=2,$ then $A=\{1,3\}=\{1\}\cup\{3\} = A_1 \cup A_2.$ Let $ S_1 = A_1 = \{1\}$ and $ S_2 = A_2 = \{3\}$ and $ B = \emptyset.$ Then $S_1$ and $S_2$ are consecutive strings and $ S_1 + 2^1 = S_2$ and further,  $A\Delta B = A = S_1 \cup S_2.$ Therefore the result holds for $k=2.$ Suppose $k > 2.$ Assume that the result holds for all sets $ D = D_1 \cup D_2$ such that $k'$ is the smallest integer with $ D < 2^{k'} $ and $k'< k.$  $A = A_1 \cup A_2 < 2^k,$ $A_1$ and $A_2$ are consecutive strings and by the choice of $k,$ either $2^{k-1} < A < 2^k$ or $A_1 < 2^{k-1} < A_2$  by Lemma 2.5(1).  

Suppose $2^{k-1} < A < 2^k.$  Let $A' =  A-2^{k-1} .$ Then $  A'\in H$  is the union of two disjoint consecutive strings $(A_1-2^{k-1})$ and $(A_2-2^{k-1}).$ Further, the smallest $k'$ for which $A' < 2^{k'}$ is $\leq k-1.$ Therefore, by induction, there exists $B' \in H$ such that $A'\Delta B' = S_1 \cup S_2,$ where $S_1$  and $S_2$ are consecutive strings and  $S_2 = S_1 + 2^{k''}$ for some $ k''.$  Let $B = A\Delta A'\Delta B'.$ Then $ B \in H$ and $A\Delta B =A\Delta(A\Delta A'\Delta B') = S_1 \cup S_2$ as $A\Delta A = \emptyset$ by the definition of symmetric difference.

Suppose $A_1 < 2^{k-1} < A_2$.  Then $A_2 \cup (A_2-2^{k-1})\in H.$  Let  $A' = A\Delta (A_2 \cup (A_2-2^{k-1}))=A_1\Delta (A_2-2^{k-1}).$ Then $ A' \in H.$Suppose $A'=\emptyset,$ then $A_1=(A_2-2^{k-1}).$ By setting $ S_1 = A_1,$ $S_2 = A_2$ and $B= \emptyset,$ we have $A\Delta B = A = S_1 \cup S_2$  with $ S_2 = S_1 + 2^{k-1}.$ Suppose $A' \neq \emptyset.$ Being the symmetric difference of two consecutive strings, $A' $ is union of two consecutive strings as $A'$ cannot be a consecutive string by Lemma 2.3(2). Let $ A' = A_3 \cup A_4,$ where $A_3$ and $A_4$ are disjoint consecutive strings.  Therefore the smallest $k'$ for which $A' < 2^{k'}$ is $\leq k-1.$ By induction hypothesis, there is  $B' \in H$ such that $A'\Delta B' = S_1 \cup S_2,$ where $S_1$ and $S_2$ are consecutive strings and $S_2=S_1+2^{k''}$ for some $k''.$ If $B =  A\Delta B'\Delta A',$ then $ B \in H$ and $ A\Delta B =  A'\Delta B' = S_1 \cup S_2.$ This proves (1).
\vskip.2cm
\noindent
(2). As $A \in H,$ $A \subset [2^m].$ By Lemma 2.5(1), $2^{m-1}$ does not belong to $A$ and hence, it does not belong to both $A_1$ as well as $A_2.$ Moreover, $A_1$ is a consecutive string which starts with $1.$ So $A_1 < 2^{m-1}.$ If $A_2 > 2^{m-1},$ then by Lemma 2.5(1), $2^m$ does not belong to $A_2.$ Hence we are through. If $A_2 < 2^{m-1},$ then $A \subset [2^{m-1}].$ Again by Lemma 2.5(1) $2^{m-2} \notin A.$ Therefore  $A_2 < 2^{m-2}.$ Again if $A_2 > 2^{m-2},$ then we are through. If not, continuing the above arguments, there exists a positive integer $k$ such $A_1 < 2^k < A_2 < 2^{k+1},$ as $n$ is finite. 
\vskip.2cm
\noindent
(3). Here $n=2^m,$ $ |A_1| = r,  |A_2| = s $ and so $|A_2 | - |A_1|  = s -r > 0.$  Since $ r \leq t -1 $ and $ r\leq s-1,$ $ 2r \leq t  + s -2< s + t \leq n -1,$  and so $ r < n/2 = 2^{m-1}.$  If $ n/2 \in A,$ then, by  Lemma 2.5(1), $ A \notin H.$ Therefore assume that $ n/2 \notin A.$  Hence   $ A_1 \subset [n/2 - 1]$ and  $ A_2 \subset [n/2 - 1]$ or $ A_2 \subset [n - 1] - [n/2].$ 

We prove the result by induction on $n = 2^m.$    If $n = 2$ or $ n = 2^2,$ then such $A$ does not exist in $G.$  Suppose $ n = 2^3.$  Then  the only element of $G$ that satisfy the hypothesis is $\{1, 3,4,5\}.$ However, this does not belong to $H.$ Therefore the result holds for $ n = 2^3.$ Suppose $ m \geq 4.$  Assume that the result holds for $ 2^{m-1} \geq 8.$ We prove the result for $ n = 2^m.$ 

 Let  $  G' = \{ B \colon B\subset [2^{m-1} - 1]~{\rm  and }~|A| ~{\rm  is~ even} \}.$ Then $G'$  is a subgroup of $G.$ Let  $\mathsf K' =\{\{j,2^{i-1}+j\}\colon j=1,2,\ldots,2^{i-1}-1,i=2,3,\ldots,m-1\}.$  Let  $H'$ be the subgroup of $G'$ generated by $\mathsf K'.$  Then $ H' \subset H.$ Suppose $A_2 \subset [n/2-1].$ Then $A \in G'.$ By induction, $A \notin H'.$ Therefore $A \notin H.$   Suppose  $A_2  $ is not a subset of $[n/2-1].$ Then $ A_2 \subset [n - 1] - [n/2].$ Let $ A_2 = \{ y_1, y_2, \dots, y_s\},$ with $y_{i+1}=y_i+1$  Then  $ y_i = n/2 + x_i$ for some $x_i \in [n/2 -1] $ for $ i=1, 2, \dots, s.$ Since $y_i$'s are consecutive integers, $x_i$'s are also consecutive.  Therefore $ \{x_i,  y_i\} \in H - H'.$  Let $Q = \cup_{i=1}^{s} \{x_i,  y_i\}= A_2' \cup A_2,$ where $A_2' = \{x_1, x_2, \dots, x_s\}.$ Then $ Q \in H - H'.$  Assume that $ A\in H.$  To get the elements of $A_2$ in $H,$  it follows that $ A = ( P_1\Delta P_2\Delta \dots\Delta P_j )\Delta Q $ for some $P_1, P_2, \dots, P_j  \in H'.$ Hence $ A\Delta Q =  P_1 \Delta P_2 \Delta \dots \Delta P_j.$ The element on right side of this equation is in $H'$ and hence in $H.$ Therefore $ A \Delta Q = (A_1 \cup A_2)\Delta ( A_2' \cup A_2) = A_1 \Delta A_2'  \in H'.$ We obtain a contradiction by showing that $A_1\Delta A_2'  \notin H'.$ 

By Lemma 2.3(2),  $A_1\Delta A_2' = \{1, 2, \dots, r\}\Delta \{x_1, x_2, \dots, x_s\}$ is not  a set of consecutive integers. Therefore $x_1 \neq 1$ and also $x_1 \neq r+1.$  Suppose $ x_1 > r+1.$ Then $A_1$ and $A_2'$ are disjoint and hence $ A_1\Delta A_2' = A_1 \cup A_2'.$
Since $A_1 \cup A_2 \in G,$  either both $|A_1|$ and $|A_2|$  are even or both are odd. Hence either both $|A_1|$ and $|A_2'|$  are even or both are odd as $|A_2'| = |A_2|.$ Therefore $A_1 \cup A_2' \in G'.$ Also,  $ |A_2'| - |A_1| = s - r.$ Therefore,  by induction, $ A_1 \cup A_2' \notin H',$ a contradiction. Hence $ 1< x_1 \leq r,$ that is, $A_1$ intersects $A_2'.$ Then $A_1 \cap A_2'$ is the set of consecutive integers from $x_1$ to $r.$  Consequently, $A_1 \Delta A_2' = \tilde{A_1} \cup \tilde{A_2'},$ where $\tilde{A_1} = \{1, 2, \dots, x_1 - 1\}$ and $\tilde{A_2} = \{ r + 1, r+2, \dots, x_s-1, x_s \}.$  The sets  $\tilde{A_1}$ and $\tilde{A_2}$ are non-empty consecutive strings and  disjoint with each other. Therefore $\tilde{A_1} \cup \tilde{A_2} \in G'.$ By induction, $\tilde{A_1} \cup \tilde{A_2} \notin H',$ a contradiction.
 \vskip.2cm
\noindent
(4). Assume that $A  \in H.$ By definition of $H,$ as $2^k \in A_1, \alpha = 2^k+2^{k'} \in A_2$ for some $k'>k.$ Therefore $\alpha = t+i$ for some $i \in \{1,2,\ldots,s\}.$ Consequently, $A_2 = \{\alpha-i+1, \alpha-i+ 2,\ldots,\alpha, \alpha+1,\alpha+2,\ldots,\alpha+(s-i)\}.$  Obviously, $ \{j, \alpha+j \} = \{ j, 2^k + j\}\Delta\{ 2^k + j, \alpha + j\} \in H$ and also $\{2^k - j, \alpha - j \} \in H$ for each $ 1\leq j< 2^k.$ Therefore $B = \{1,2,\ldots,s-i\} \cup \{\alpha+1,\alpha+2,\ldots,\alpha+s-i\} \in H$ and    $C=  \{ 2^k - i + 1, \dots, 2^k \}  \cup \{\alpha-i+1, \alpha-1,\ldots,\alpha \} \in H .$ Therefore $ A\Delta B\Delta C \in H.$ But $A\Delta B\Delta C =  \{s-i+1,s-i+2,\ldots,2^k-i\} $ is a consecutive string, a contradiction by Lemma 2.3(2).  
\vskip.2cm
\noindent
(5). Suppose $ A  \in H.$ By definition of $H,$ $A_2 \cup (A_2-2^{k_t}) \in H.$ Similarly, $(A_2-2^{k_t}) \cup (A_2-2^{k_t}-2^{k_{t-1}}),$ $(A_2-2^{k_t}-2^{k_{t-1}}) \cup (A_2-2^{k_t}-2^{k_{t-1}}-2^{k_{t-2}}),$ \ldots $(A_2- \sum_{i=3}^{t}2^{k_i}) \cup (A_2- \sum_{i=2}^{t}2^{k_i}) \in H.$ Therefore $(A_1 \cup A_2)\Delta (A_2\cup (A_2-2^{k_t}))\Delta ((A_2-2^{k_t}) \cup (A_2-2^{k_t}-2^{k_{t-1}}))\Delta \ldots \Delta((A_2-\sum_{i=3}^{t}2^{k_i}) \cup (A_2- \sum_{i=2}^{t}2^{k_i})) = A_1 \cup (A_2 -\sum_{i=2}^{t}2^{k_i}) = \{1,2,\ldots,r\} \cup \{2^{k_1}+1,2^{k_1}+2,\ldots,2^{k_1}+s\} \in H.$  Again, by definition of $H,$ $\{j,2^{k_1}+j\} \in H$ for all $1\leq j < 2^{k_1}$. Therefore  $\{1,2,\ldots,s\}\cup\{2^{k_1}+1,2^{k_1}+2,\ldots,2^{k_1}+s\} \in H.$ By taking its symmetric difference with $A_1 \cup (A_2 -\sum_{i=2}^{t}2^{k_i}),$ we get $\{s+1,s+2,\ldots,r\}\in H$, a contradiction by Lemma 2.3(2). 
\end{proof}

The {\it Cartesian product} $G \Box H$ of two graphs $G$ and $H$ is the graph whose vertex set is $V(G) \times V(H)$; and any two vertices $(U,U')$ and $(V,V')$ in it are adjacent if and only if either $U = V$ and $U'$ is adjacent to $V'$ in $H,$ or $U' = V'$ and $U$ is adjacent to $V$ in $G$.

The proof of the following lemma is trivial.

\begin{lemma}\cite{Elz}	
\begin{enumerate}
\item Let a graph $G_1$ be decomposed into spanning subgraphs $H_1, H_2, \ldots, H_r$ and let a graph $G_2$ be  decomposed
into spanning subgraphs $F_1, F_2, \ldots,F_r.$  Then the graph
$G_1 \Box ~ G_2$ can be decomposed into spanning subgraphs $H_1
\Box ~ F_1, H_2 \Box ~ F_2, \ldots, H_r \Box ~ F_r.$
\item Let $G_1$ be a graph with components $H_1, H_2, \ldots, H_r$ and let $G_2$ be a graph with components $F_1, F_2, \ldots, F_s.$ Then the components of  $ G_1 \Box ~ G_2$ are $ H_i \Box ~ F_j$ with $ i= 1, 2, \ldots, r$ and $ j = 1, 2, \ldots, s.$
\end{enumerate}
\end{lemma}
 \indent The following is an older result due to Kotzig \cite{Kot}.
\begin{lemma} [\cite{Kot}] If $G$ is the Cartesian product of two cycles, then $G$ can be
decomposed into two hamiltonian cycles.
\end{lemma}

We provide the detailed description of Kotzig's construction of Hamiltonian cycles in the following remark as given in \cite{Mol}.
For that we define a bijective map $\theta$ for the vertices of  $Q_{2n}.$ 

\begin{definition}
Let $\theta~:[2n]\rightarrow [2n]$ be defined as

\begin{tabular}{cclc}
$\theta(i)$ &$=$& $i+n ~ (mod ~2n),$ & if $i \neq n$\\
 &$=$& $2n,$& if $i = n.$
\end{tabular}

\end{definition}

\begin{Rm}
Let $C_1=C(U_1,S)$ and $C_2=C(U_2,S)$ be two $k$-cycles in $Q_n$ with initial vertices $U_1$ and $U_2,$ respectively and the edge direction sequence $S = (s_1,s_2,\ldots,s_k).$ Then $U_1 \subseteq[n],$ $U_2 \subseteq [n]$ and $s_i \in [n]$ for all $i=1,2,\ldots,k.$ By Lemma 2.9, $C_1\Box C_2 $ is decomposable into two $k^2$-cycles (say) $\Phi$ and $\Gamma$  in $ Q_{2n} = Q_n \Box Q_n.$ We assume that both $\Phi$ and $\Gamma$ start with the same initial vertex $(U_1,U_2).$ By considering the adjacency preserving bijective correspondence from $V(Q_n \Box Q_n)$ to $V(Q_{2n})$ given by $(U_1,U_2) \mapsto U_1 \cup \theta(U_2),$ we can write $U_1 \cup \theta(U_2) $ for $(U_1, U_2).$ The edge-direction sequences for the cycles $\Phi$ and $\Gamma$ are $\mathcal S$ and $\theta(\mathcal S)$ respectively, where $\mathcal S$ is given by 
\begin{center}
\begin{tabular}{lll}
$\mathcal S$&$=$&$(s_1,s_2,s_3,\ldots,s_{k-1},\theta(s_1),$\\
&&$s_k, s_1,s_2,\ldots,s_{k-2}, \theta(s_2),$\\
&&$s_{k-1}, s_k,s_1,\ldots,s_{k-3},\theta(s_3),$\\
&&\hspace{2cm}$\vdots$\\
&&$s_2,s_3,s_4,\ldots, s_k,\theta(s_k)).$
\end{tabular}
\end{center}
\noindent By $\theta(\mathcal S),$ we mean the sequence obtained by applying $\theta$ to each term of $\mathcal S.$ Note that $\theta(\theta(s_i))=s_i,$ by definition. Since  $U_1 \subseteq \{1, 2, \ldots, n\}$ and $\theta (U_2) \subseteq \{n+1, n+2, \ldots, 2n\},$ $U_1 \cup \theta(U_2)=U_1\Delta~\theta(U_2)=U_1\theta(U_2).$ Thus  $C_1(U_1,S) \Box C_2(U_2,S)=\Phi(U_1\theta(U_2),\mathcal S) \sqcup \Gamma(U_1\theta(U_2),\theta(\mathcal S))~~{(see~Figure~ 1)}.$

\begin{figure}[h]
    \begin{tikzpicture}[scale=0.6] 
    
 \draw[fill=black](-2.5,0) circle(.07); 
 \draw[fill=black](-2.5,1) circle(.07); 
 \draw[fill=black](-2.5,2) circle(.07); 
 \draw[fill=black](-2.5,3) circle(.07); 
 \draw[fill=black](-2.5,4) circle(.07); 
 \draw[fill=black](-2.5,5) circle(.07); 
 \draw[fill=black](-2.5,6) circle(.07); 
 \draw[fill=black](-2.5,7) circle(.07);    

\draw[fill=black](0,0) circle(.07); 
\draw[fill=black](0,1) circle(.07); 
\draw[fill=black](0,2) circle(.07); 
\draw[fill=black](0,3) circle(.07); 
\draw[fill=black](0,4) circle(.07); 
\draw[fill=black](0,5) circle(.07); 
\draw[fill=black](0,6) circle(.07); 
\draw[fill=black](0,7) circle(.07); 

\draw[fill=black](-1,0) circle(.07); 
\draw[fill=black](-1,1) circle(.07); 
\draw[fill=black](-1,2) circle(.07); 
\draw[fill=black](-1,3) circle(.07); 
\draw[fill=black](-1,4) circle(.07); 
\draw[fill=black](-1,5) circle(.07); 
\draw[fill=black](-1,6) circle(.07); 
\draw[fill=black](-1,7) circle(.07); 

\draw[fill=black](2,0) circle(.07); 
\draw[fill=black](2,1) circle(.07); 
\draw[fill=black](2,2) circle(.07); 
\draw[fill=black](2,3) circle(.07); 
\draw[fill=black](2,4) circle(.07); 
\draw[fill=black](2,5) circle(.07); 
\draw[fill=black](2,6) circle(.07); 
\draw[fill=black](2,7) circle(.07); 

\draw[fill=black](3,0) circle(.07); 
\draw[fill=black](3,1) circle(.07); 
\draw[fill=black](3,2) circle(.07); 
\draw[fill=black](3,3) circle(.07); 
\draw[fill=black](3,4) circle(.07); 
\draw[fill=black](3,5) circle(.07); 
\draw[fill=black](3,6) circle(.07); 
\draw[fill=black](3,7) circle(.07); 

\draw[fill=black](4,0) circle(.07); 
\draw[fill=black](4,1) circle(.07); 
\draw[fill=black](4,2) circle(.07); 
\draw[fill=black](4,3) circle(.07); 
\draw[fill=black](4,4) circle(.07); 
\draw[fill=black](4,5) circle(.07); 
\draw[fill=black](4,6) circle(.07); 
\draw[fill=black](4,7) circle(.07);
 
\draw[fill=black](5,0) circle(.07); 
\draw[fill=black](5,1) circle(.07); 
\draw[fill=black](5,2) circle(.07); 
\draw[fill=black](5,3) circle(.07); 
\draw[fill=black](5,4) circle(.07); 
\draw[fill=black](5,5) circle(.07); 
\draw[fill=black](5,6) circle(.07); 
\draw[fill=black](5,7) circle(.07); 

\draw[fill=black](6,0) circle(.07); 
\draw[fill=black](6,1) circle(.07); 
\draw[fill=black](6,2) circle(.07); 
\draw[fill=black](6,3) circle(.07); 
\draw[fill=black](6,4) circle(.07); 
\draw[fill=black](6,5) circle(.07); 
\draw[fill=black](6,6) circle(.07); 
\draw[fill=black](6,7) circle(.07); 

\draw[fill=black](1,0) circle(.07); 
\draw[fill=black](1,1) circle(.07); 
\draw[fill=black](1,2) circle(.07); 
\draw[fill=black](1,3) circle(.07); 
\draw[fill=black](1,4) circle(.07); 
\draw[fill=black](1,5) circle(.07); 
\draw[fill=black](1,6) circle(.07); 
\draw[fill=black](1,7) circle(.07);

\draw[fill=black](9,0) circle(.07); 
\draw[fill=black](9,1) circle(.07); 
\draw[fill=black](9,2) circle(.07); 
\draw[fill=black](9,3) circle(.07); 
\draw[fill=black](9,4) circle(.07); 
\draw[fill=black](9,5) circle(.07); 
\draw[fill=black](9,6) circle(.07); 
\draw[fill=black](9,7) circle(.07); 

\draw[fill=black](10,0) circle(.07); 
\draw[fill=black](10,1) circle(.07); 
\draw[fill=black](10,2) circle(.07); 
\draw[fill=black](10,3) circle(.07); 
\draw[fill=black](10,4) circle(.07); 
\draw[fill=black](10,5) circle(.07); 
\draw[fill=black](10,6) circle(.07); 
\draw[fill=black](10,7) circle(.07); 

\draw[fill=black](11,0) circle(.07); 
\draw[fill=black](11,1) circle(.07); 
\draw[fill=black](11,2) circle(.07); 
\draw[fill=black](11,3) circle(.07); 
\draw[fill=black](11,4) circle(.07); 
\draw[fill=black](11,5) circle(.07); 
\draw[fill=black](11,6) circle(.07); 
\draw[fill=black](11,7) circle(.07); 

\draw[fill=black](12,0) circle(.07); 
\draw[fill=black](12,1) circle(.07); 
\draw[fill=black](12,2) circle(.07); 
\draw[fill=black](12,3) circle(.07); 
\draw[fill=black](12,4) circle(.07); 
\draw[fill=black](12,5) circle(.07); 
\draw[fill=black](12,6) circle(.07); 
\draw[fill=black](12,7) circle(.07); 

\draw[fill=black](13,0) circle(.07); 
\draw[fill=black](13,1) circle(.07); 
\draw[fill=black](13,2) circle(.07); 
\draw[fill=black](13,3) circle(.07); 
\draw[fill=black](13,4) circle(.07); 
\draw[fill=black](13,5) circle(.07); 
\draw[fill=black](13,6) circle(.07); 
\draw[fill=black](13,7) circle(.07);
 
\draw[fill=black](14,0) circle(.07); 
\draw[fill=black](14,1) circle(.07); 
\draw[fill=black](14,2) circle(.07); 
\draw[fill=black](14,3) circle(.07); 
\draw[fill=black](14,4) circle(.07); 
\draw[fill=black](14,5) circle(.07); 
\draw[fill=black](14,6) circle(.07); 
\draw[fill=black](14,7) circle(.07); 

\draw[fill=black](15,0) circle(.07); 
\draw[fill=black](15,1) circle(.07); 
\draw[fill=black](15,2) circle(.07); 
\draw[fill=black](15,3) circle(.07); 
\draw[fill=black](15,4) circle(.07); 
\draw[fill=black](15,5) circle(.07); 
\draw[fill=black](15,6) circle(.07); 
\draw[fill=black](15,7) circle(.07); 

\draw[fill=black](8,0) circle(.07); 
\draw[fill=black](8,1) circle(.07); 
\draw[fill=black](8,2) circle(.07); 
\draw[fill=black](8,3) circle(.07); 
\draw[fill=black](8,4) circle(.07); 
\draw[fill=black](8,5) circle(.07); 
\draw[fill=black](8,6) circle(.07); 
\draw[fill=black](8,7) circle(.07); 

\draw[fill=black](-1,8.5) circle(.07); 
\draw[fill=black](0,8.5) circle(.07); 
\draw[fill=black](1,8.5) circle(.07); 
\draw[fill=black](2,8.5) circle(.07); 
\draw[fill=black](3,8.5) circle(.07); 
\draw[fill=black](4,8.5) circle(.07); 
\draw[fill=black](5,8.5) circle(.07); 
\draw[fill=black](6,8.5) circle(.07);

\draw (-2.5,0)--(-2.5,1)--(-2.5,2)--(-2.5,3)--(-2.5,4)--(-2.5,5)--(-2.5,6)--(-2.5,7);

\draw (-1,0)--(-1,1)--(-1,2)--(-1,3)--(-1,4)--(-1,5)--(-1,6)--(-1,7);
\draw (0,1)--(0,2)--(0,3)--(0,4)--(0,5)--(0,6)--(0,7);
\draw (1,0)--(1,1);
\draw (1,2)--(1,3)--(1,4)--(1,5)--(1,6)--(1,7);
\draw (2,0)--(2,1)--(2,2);
\draw (2,3)--(2,4)--(2,5)--(2,6)--(2,7);
\draw (3,0)--(3,1)--(3,2)--(3,3);
\draw (3,4)--(3,5)--(3,6)--(3,7);
\draw (4,0)--(4,1)--(4,2)--(4,3)--(4,4);
\draw (4,5)--(4,6)--(4,7);
\draw (5,0)--(5,1)--(5,2)--(5,3)--(5,4)--(5,5);
\draw (5,6)--(5,7);
\draw (6,0)--(6,1)--(6,2)--(6,3)--(6,4)--(6,5)--(6,6);

\draw (-1,0)--(0,0);
\draw (0,1)--(1,1);
\draw (1,2)--(2,2);
\draw (2,3)--(3,3);
\draw (3,4)--(4,4);
\draw (4,5)--(5,5);
\draw (5,6)--(6,6);

\draw (0,0) to[out=100, in=260] (0,7);
\draw (1,0) to[out=100, in=260] (1,7);
\draw (2,0) to[out=100, in=260] (2,7);
\draw (3,0) to[out=100, in=260] (3,7);
\draw (4,0) to[out=100, in=260] (4,7);
\draw (5,0) to[out=100, in=260] (5,7);
\draw (6,0) to[out=100, in=260] (6,7);

\draw (-1,7) to[out=10, in=170] (6,7);

\draw (-1,8.5)--(0,8.5)--(1,8.5)--(2,8.5)--(3,8.5)--(4,8.5)--(5,8.5)--(6,8.5);

\draw  (8,1)--(9,1)--(9,0)--(10,0)--(11,0)--(12,0)--(13,0)--(14,0)--(15,0);
\draw  (8,2)--(9,2)--(10,2)--(10,1)--(11,1)--(12,1)--(13,1)--(14,1)--(15,1);
\draw  (8,3)--(9,3)--(10,3)--(11,3)--(11,2)--(12,2)--(13,2)--(14,2)--(15,2);
\draw  (8,4)--(9,4)--(10,4)--(11,4)--(12,4)--(12,3)--(13,3)--(14,3)--(15,3);
\draw  (8,5)--(9,5)--(10,5)--(11,5)--(12,5)--(13,5)--(13,4)--(14,4)--(15,4);
\draw  (8,6)--(9,6)--(10,6)--(11,6)--(12,6)--(13,6)--(14,6)--(14,5)--(15,5);
\draw  (8,7)--(9,7)--(10,7)--(11,7)--(12,7)--(13,7)--(14,7)--(15,7);
\draw  (15,6)--(15,7);
\draw  (8,0) to[out=100, in=260] (8,7);
\draw  (8,0) to[out=10, in=170] (15,0);
\draw  (8,1) to[out=10, in=170] (15,1);
\draw  (8,2) to[out=10, in=170] (15,2);
\draw  (8,3) to[out=10, in=170] (15,3);
\draw  (8,4) to[out=10, in=170] (15,4);
\draw  (8,5) to[out=10, in=170] (15,5);
\draw  (8,6) to[out=10, in=170] (15,6);

\draw (-2.5,0) to[out=100, in=260] (-2.5,7);

\draw (-1,8.5) to[out=10, in=170] (6,8.5);

\node [below] at (2.5,-0.25){$\Phi$};
\node [below] at (11.5,-0.25){$\Gamma$};

\node [above] at (-2.5,7){\small{$U_1$}};
\node [above] at (8,6.9){\small{$U_1\theta(U_2)$}};
\node [above] at (-.75,6.9){\small{$U_1\theta(U_2)$}};
\node [above] at (2.5,8.8){\small{$\theta(s_k)$}};

\node [below] at (-0.75,8.5){\small{$\theta(s_1)$}};
\node [below] at (0.5,8.5){\small{$\theta(s_2)$}};
\node [below] at (3,8.2){$\ldots$};
\node [below] at (5.75,8.5){\small{$\theta(s_{k-1})$}};
\node [left] at (-1,8.5){\small{$\theta(U_2)$}};

\node [left] at (-2.7,3.5){\small{$s_k$}};

\node [right] at (-2.5,6.5){\small{$s_1$}};
\node [right] at (-2.5,5.5){\small{$s_2$}};
\node [right] at (-2.3,3){$\vdots$};
\node [right] at (-2.5,0.5){\small{$s_{k-1}$}};

\end{tikzpicture}
\caption{Decomposition of $C_1\Box C_2$ into cycles $\Phi$ and $\Gamma$}\label{f3}
\end{figure}
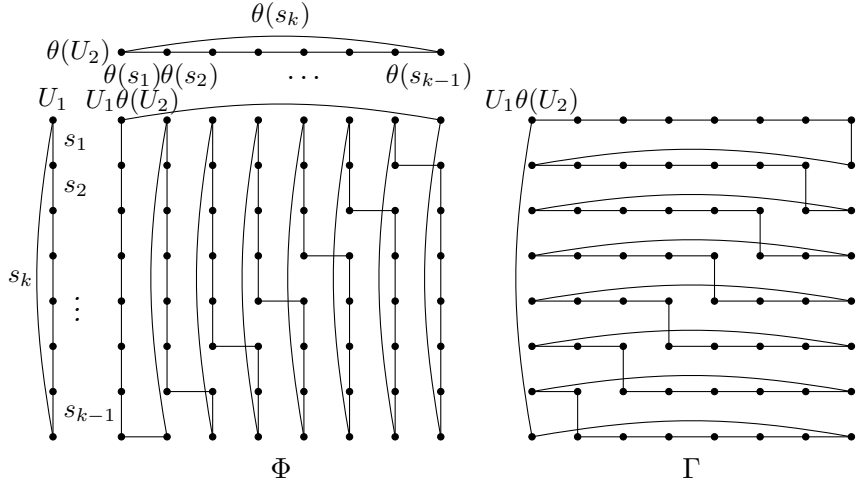
\end{Rm}

\begin{Rm}
Let $G,H$ be as in Lemma 2.3. If $H_i$ is a coset of  $H$ in the group $G,$ then $H_i \times H_i$ is a coset of $H\times H$ in $G \times G.$ By Lemma 2.3(3), $H_i=[2i-2]H.$ Hence $H_i \times H_i=[2i-2]\theta([2i-2])(H_1 \times H_1).$
\end{Rm}

\section{Proof of Basis Step}
In Theorem 3.1 of this section, we state the Basis Step for the proof of Main Theorem 1.6. The proof of Theorem 3.1 is crucial and has two parts. In this section, we prove the first part of the Theorem and the second part is proved in Section 5. 

We introduce the following notation.

\noindent \textbf{Notation:} We use the notation $\sqcup$ for the union of edge-disjoint cycles and $\uplus$ for the union of vertex-disjoint cycles.

\begin{theorem}{\textbf{Basis Step:}}
For $n = 2^m, m \geq 1, $ $Q_n$ can be decomposed into $2^mn$-cycles, (say) $C_1, C_2, \ldots, C_r,$ where $r = 2^{n-1-m}$ such that
\begin{itemize}
\item [(I)] every $C_i$ contains $2^m$ edges $e_{i_1}, e_{i_2}, \ldots, e_{i_{2^m}}$ such that  $C_i - \{e_{i_1}, e_{i_2}, \ldots, e_{i_{2^m}}\}$ has $2^m$ components each of which is a path of length $n-1$;
\item[(II)] $M=\{e_{i_j}\colon j = 1, 2, \ldots, 2^m; i = 1, 2, \ldots, r\}$ forms a perfect matching in $Q_n.$
\end{itemize}
\end{theorem}

\begin{proof}
If $ m = 1,$ then $Q_n$ is a just a 4-cycle and the result holds obviously. Suppose $ m \geq 2.$  We prove the theorem in two steps. In the first step, we give a construction of the $2^mn$-cycles which decompose $Q_n.$ In the second step,  we give a method to select edges from these cycles to get a perfect matching in $Q_n$ satisfying (I) and (II).\\~\\
\textbf{Step 1. Construction of $2^mn$-cycles in $Q_n$:}

Consider $Q_n = Q_{n/2} \Box Q_{n/2}.$ Then $n/2=2^{m-1},$ $m-1 \geq 1.$ 

Let $G = \{A\colon A \subset [n/2-1],~|A|~even\}.$ Then $G$ is a group with respect to $\Delta$ with order $2^{n/2-2}=2^{2^{m-1}-2}.$  Let $H= <\mathsf K>,$ where $\mathsf K =\{\{j,2^{i-1}+j\}\colon j=1,2,\ldots,2^{i-1}-1, i=2,3,\ldots,m-1\}.$ Then by Lemma 2.3(2), $|H|=2^{n/2-m}$ and $H$ does not contain a consecutive string. By Lemma 2.3(3), $H=H_1$ and its cosets $H_i = [2i-2]H, i=2,3,\ldots,n/4$ partition the group $G.$

Consider the $n$-cycle $C=C(\emptyset, S)$ in $Q_{n/2},$ having initial vertex $\emptyset$ and edge direction sequence $S=(1,2,\ldots,n/2,1,2,\ldots,n/2).$ We can write $S=T^2,$ the concatenation of $T$ with itself, with $T=(1,2,\ldots,n/2).$ Recall the map $\sigma_A:V(Q_{n/2})\rightarrow V(Q_{n/2})$ defined by $\sigma_A(B)=A\Delta B=AB.$ Then, by Theorem 2.2,  $\mathcal C =\{\sigma_A(C)\colon A \in G\} = \{C(A,S) \colon A \in G\}$ is a collection of edge-disjoint $n$-cycles decomposing $Q_{n/2}.$ Hence $Q_{n/2}=\displaystyle \bigsqcup_{A \in G} C(A,S).$  Let  $\mathcal C_i=\{C(A,S)\colon A \in H_i\}$ for $ i=1,2,\ldots,n/4.$  Then $\mathcal C_1, \mathcal C_2, \ldots, \mathcal C_{n/4}$  partition $\mathcal C.$ By Lemma 2.4, the cycles belonging to each $\mathcal C_i$ are vertex-disjoint and their union $W_i=\displaystyle \biguplus_{A \in H_i} C(A,S) $ is  a $2$-regular spanning subgraph of $Q_{n/2}.$  Thus, \\
\begin{tabular}{lll}
$Q_{n/2}=\displaystyle \bigsqcup_{A \in G} C(A,S)$&$=$&$[\displaystyle \biguplus_{A \in H_1} C(A,S) ]\sqcup [\displaystyle \biguplus_{A \in H_2} C(A,S) ]\sqcup\ldots\sqcup[\displaystyle \biguplus_{A \in H_{n/4}} C(A,S)]$\\
&$=$&$W_1\sqcup W_2 \sqcup \ldots \sqcup W_{n/4}$
\end{tabular}\\
Now $$Q_n = Q_{n/2} \Box Q_{n/2} = (W_1 \sqcup W_2 \sqcup \ldots \sqcup W_{n/4}) \Box (W_1 \sqcup W_2 \sqcup \ldots \sqcup W_{n/4})$$ $$=\textrm{(by Lemma 2.8(1))} ~~(W_1 \Box W_1) \sqcup (W_2 \Box W_2) \sqcup \ldots \sqcup (W_{n/4} \Box W_{n/4})$$\\
By Lemma 2.8(2), for $ i = 1, 2, \ldots, n/2,$  $$W_i \Box W_i = (\displaystyle \biguplus_{A \in H_i} C(A,S))  \Box (\displaystyle \biguplus_{B \in H_i} C(B,S)) = \displaystyle \biguplus_{A,B \in H_i} (C(A,S) \Box C(B,S)).$$ By Lemma 2.9 and Remark 2.11, $C(A,S) \Box C(B,S),$ where $A,B \in H_i,$ is decomposed into two cycles $\Phi^i_{AB}=\Phi^i(A\Delta~\theta(B),\mathcal S)$ and $\Gamma^i_{AB} =\Gamma^i(A\Delta~\theta(B), \theta(\mathcal S))$ of $Q_n,$ each of length $n^2.$ Moreover,  $\Phi^i_{AB}$ and $\Gamma^i_{AB}$ have the same initial vertex $A\Delta~\theta(B).$ The edge-direction sequences of $\Phi^i_{AB}$ and $\Gamma^i_{AB}$ are $\mathcal S$ and $\theta(\mathcal S)$ respectively, where $\mathcal S = \tau^2,$ the concatenation of $\tau$ with itself, with 
\begin{center}
\begin{tabular}{rcl}
$\tau$&$=$&$(1, 2, 3, 4, \ldots, n/2, 1, 2, \ldots, n/2-1, \theta(1),$\\
&&$n/2, 1, 2,3, \ldots, n/2, 1, 2, \ldots, n/2-2, \theta(2),$\\
&&$n/2-1, n/2, 1, 2, \ldots, n/2, 1, 2, \ldots, n/2-3, \theta(3),$\\
&&\hspace{3cm}\vdots\\
&&$2,3,4,5, \ldots, n/2, 1, 2, \ldots, n/2, \theta(n/2)).$\\

\end{tabular}
\end{center}
Hence
\begin{center}
\begin{tabular}{ccl}
$W_i\Box W_i$ & $=$ & $\displaystyle \biguplus_{A,B \in H_i} (~~C(A,S)~ \Box~ C(B,S)~~)$\\ \\
&$=$&$\displaystyle \biguplus_{A,B \in H_i} (~~\Phi^i(A\Delta~\theta(B),\mathcal S) ~\sqcup ~\Gamma^i(A\Delta~\theta(B), \theta(\mathcal S)~~)$ \\ \\
&$=$&$[\displaystyle \biguplus_{A,B \in H_i} \Phi^i_{AB}~~]~\sqcup~ [\displaystyle \biguplus_{A,B \in H_i} \Gamma^i_{AB}~~]$ \\ \\
&$=$&$F_i \sqcup F_i',$
\end{tabular}
\end{center} 
where $F_i=\displaystyle \biguplus_{A,B \in H_i} \Phi^i_{AB}$ and $F_i'=\displaystyle \biguplus_{A,B \in H_i} \Gamma^i_{AB}.$ Therefore $F_i$ and $F_i'$ are spanning, $2$-regular subgraphs of $Q_n.$
Now  $|H_i|=2^{n/2-m}.$ Hence the number of cycles in each $F_i$ and $F_i'$ is $2^{n/2-m}\times 2^{n/2-m}=2^{n-2m}.$ This implies that both $2$-regular subgraphs $F_i$ and $F_i'$ are spanning subgraphs of $Q_n$ as $n^2\times 2^{n-2m}=(2^m)^2\times 2^{n-2m}=2^n=|V(Q_n)|$.\\
 So $$Q_n = Q_{n/2} \Box Q_{n/2} = (W_1 \Box W_1) \sqcup (W_2 \Box W_2) \sqcup \ldots \sqcup (W_{n/4} \Box W_{n/4})$$ $$= \{\Phi^i_{AB}\colon A,~B ~\in~ H_i;~ 1 \leq i \leq n/4\}~\sqcup~\{\Gamma^i_{AB}\colon A,~B ~\in~ H_i;~ 1 \leq i \leq n/4\}.$$
Therefore $Q_n$ is decomposed into $2^mn$-cycles, when $n=2^m.$ (See Appendix for illustration of $Q_8.$)
\vskip .2cm  

\textbf{Step 2. Selection of the edges from the cycles to satisfy (I) and (II): }

Note that for fixed $i,$ the $2^mn$ cycles in $F_i$ are all vertex-disjoint. So, whatever edges we select from one cycle of $F_i,$ the corresponding edges can be selected from the other cycles of $F_i.$ Also, the $2^m$ edges to be selected from each of the cycles should be equidistant in it by condition (I). So, it is enough to select a single edge from any one cycle of $F_i,$ and the remaining edges will be selected automatically. The case for $F_i'$ is similar. Also note that all the cycles in $F_i$ and $F_i'$ have edge-direction sequence $\mathcal S$ and $\theta(\mathcal S),$ respectively. Therefore while selecting the edges from $F_i,$ $\mathcal S$ will be under consideration. Similarly, while selecting the edges from $F_i',$ $\theta(\mathcal S)$ will be considered.

We will select the $n$ edges from each cycle of $F_i$ satisfying condition (I) and denote the collection of these $(2^n / n^2) \times n = 2^n/n$ edges by $M_i$. Similarly, the collection of edges to be selected from the cycles in $F_i'$ will be denoted by $M_i'.$ We begin with $F_1.$

\textbf{Selection of the edges from the cycles of $F_1$:} Consider the cycle $\Phi^1_{\emptyset\emptyset}=\Phi(\emptyset,\mathcal S)$ in $F_1.$ We need to select $2^m=n$ edges from this cycle to get the matching $M_1.$  We start with the first edge viz., $e^1_1= (\emptyset,\{1\}).$ By condition (I), the next edge of $M_1$ should be the $(n+1)^{st}$ edge of the cycle. This edge  is given by $e^1_2=(\{n/2,\theta(1)\},\{\theta(1)\}).$ (Recall that $\theta(i)=n/2+i,$ for $i=1,2,\ldots,n/2.$ ) The third edge contributing to $M_1$ is the $(2n+1)^{st}$ edge of the cycle given by $e^1_3=(\{n/2-1,n/2,\theta(1),\theta(2)\}, \{n/2,\theta(1),\theta(2)\}).$ The remaining edges of $M_1$ are selected similarly from the cycle.  Therefore, the $n$ edges contributing to $M_1$ are as follows.\\ \\
\begin{tabular}{lll}
$e^1_1$&$=$&$(\emptyset,\{1\})$\\
$e^1_2$&$ =$&$ (\{n/2,\theta(1)\},\{\theta(1)\}),$\\
$e^1_3$&$ =$&$ (\{n/2-1,n/2,\theta(1),\theta(2)\}, \{n/2,\theta(1),\theta(2)\}),$ \\
$e^1_4$&$ = $&$(\{n/2-2,n/2-1,n/2,\theta(1),\theta(2),\theta(3)\}, \{n/2-1,n/2,\theta(1),\theta(2),\theta(3)\}),$\\
 &\vdots&\\
$e^1_{n/2-1}$&$=$&$(\{3,4,\ldots,n/2,\theta(1),\theta(2),\ldots,\theta(n/2-2)\},\{4,5,\ldots,n/2,\theta(1),\theta(2),\ldots,\theta(n/2-2)$, \\
$e^1_{n/2}$&$ =$&$ (\{2,3,\ldots,n/2,\theta(1),\theta(2),\ldots,\theta(n/2-1)\},\{3,4,\ldots,n/2,\theta(1),\theta(2),\ldots,\theta(n/2-1)\})$,\\ 
$e^1_{n/2+1}$&$ =$&$ (\{1,2,\ldots,n/2,\theta(1),\theta(2),\ldots,\theta(n/2)\},\{2,\ldots,n/2,\theta(1),\theta(2),\ldots,\theta(n/2)\}),$\\
 $e^1_{n/2+2}$&$ =$&$ (\{1,2,\ldots,n/2-1,\theta(2),\ldots,\theta(n/2)\},\{1,2,\ldots,n/2,\theta(2),\ldots,\theta(n/2)\}),$\\
& \vdots&\\
$e^1_{n-1}$&$ =$&$ (\{1,2,\theta(n/2-1),\theta(n/2)\},\{1,2,3,\theta(n/2-1),\theta(n/2)\}),$\\
 $e^1_n$&$ =$&$ (\{1,\theta(n/2)\},\{1,2,\theta(n/2)\}).$
\end{tabular}\\ \\
Observe that the end vertices of the above edges are essentially represented with the help of symbols in $[n/2]=\{1,2,\ldots,n/2\}$ only.
 In fact, one can write the vertices in a particular pattern, which is simpler. For that purpose, we modify the usual notation $[i]$ to $\langle i \rangle,$ and define it as follows.
\begin{Not}
\begin{enumerate}
\item [(i)] $\langle 0\rangle=\emptyset.$
\item [(ii)]For $1\leq i \leq n/2,$  $\langle i\rangle =\{1,2,\ldots,i\}=[i]$ (in usual sense)
\item [(iii)]For $1\leq i \leq n/2,$  $\langle n/2+i\rangle =\{i+1,i+2,\ldots,n/2\}=\langle n/2 \rangle\Delta\langle i \rangle.$
\item [(iv)]For $X \subseteq \{1,2,\ldots,n/2\},$ $\overline{X}=\langle n/2 \rangle\Delta X,$ the \emph{complement} of $X$ in $\langle n/2 \rangle.$ 
\item [(v)]For $A=\{a_1,a_2,\ldots,a_r\}\subseteq \langle n/2 \rangle,$ $\theta (A)=\{\theta(a_1),\theta(a_2),\ldots,\theta(a_r)\},$\\ and $\theta( \emptyset)= \theta(\langle 0\rangle )=  \emptyset.$ Also, $\overline{\theta(A)}=\theta(\overline{A}).$
\end{enumerate} 
\end{Not}
Observe that, $\langle n/2+i\rangle=\overline{\langle i\rangle}$ and $\overline{\langle n/2+i\rangle}=\overline{\overline{\langle i\rangle}}=\langle i\rangle,$ for $1\leq i \leq n/2.$ 
\vskip .2 cm 
\noindent
\textbf{From here on, we use the above notations throughout.}\\ \\
By this notation, we can write the above-mentioned edges in a compact form as follows. \\
\small{
\begin{tabular}{l|l}
$e^1_1=(~\overline{\langle n/2\rangle}\Delta~\theta(~\langle 0\rangle ~),\overline{\langle n/2+1\rangle}\Delta~\theta(~\langle 0\rangle ~)~),$ & $e^1_{n/2+1}=(~\langle n/2\rangle \Delta~\theta(~\overline{\langle 0\rangle }~),\langle n/2+1\rangle \Delta~\theta(~\overline{\langle 0\rangle}~)~),$\\
$e^1_2=(~\overline{\langle n/2-1\rangle}\Delta~\theta(~\langle 1\rangle ~),\overline{\langle n/2\rangle}\Delta~\theta(~\langle 1\rangle ~)~),$ & $e^1_{n/2+2}=(~\langle n/2-1\rangle\Delta~\theta(~{\overline{\langle 1\rangle}}~),\langle n/2\rangle \Delta~\theta(~\overline{\langle 1\rangle}~)~),$\\
$e^1_3=(~\overline{\langle n/2-2\rangle}\Delta~\theta(~\langle 2\rangle ~),\overline{\langle n/2-1\rangle}\Delta~\theta(~\langle 2\rangle ~)~),$ & $e^1_{n/2+3}=(~\langle n/2-2\rangle\Delta~\theta(~\overline{\langle 2\rangle }~),\langle n/2-1\rangle \Delta~\theta(~\overline{\langle 2\rangle}~)~),$\\
\hskip 2cm \vdots & \hskip 2cm \vdots\\
$e^1_{n/2-1}=(~\overline{\langle 2\rangle}\Delta~\theta(~\langle n/2-2\rangle ~),{\overline{\langle 3\rangle}}\Delta~\theta(~\langle n/2-2\rangle ~)~),$ & $e^1_{n-1}=(~\langle 2\rangle\Delta~\theta(~\overline{\langle n/2-2\rangle }~),\langle 3\rangle \Delta~\theta(~\overline{\langle n/2-2\rangle}~)~),$\\
$e^1_{n/2}=(~\overline{\langle 1\rangle}\Delta~\theta(~\langle n/2-1\rangle ~),\overline{\langle 2\rangle}\Delta~\theta(~\langle n/2-1\rangle ~)~),$& $e^1_{n}=(~\langle 1\rangle\Delta~\theta(~\overline{\langle n/2-1\rangle }~),\langle 2\rangle \Delta~\theta(~\overline{\langle n/2-1\rangle}~)~),$
\end{tabular}}
\vskip 0.2cm
\normalsize
Note that,  end vertices of the edges in the left column are complements of end vertices of the edges in right column. In fact,these edges can be rewritten in more compact form as follows.  \\
\begin{tabular}{lll}
$e^1_r$&$=$&$ (~\overline{\langle n/2-r+1\rangle} \Delta~\theta(~\langle r-1\rangle ~), \overline{\langle n/2-r+2\rangle }\Delta~\theta(~\langle r-1\rangle ~)~), $\\
$e^1_{n/2+r}$&$=$&$ (~\langle n/2-r+1\rangle\Delta~\theta(~\overline{\langle r-1\rangle }~), \langle n/2-r+2\rangle \Delta~\theta(~\overline{\langle r-1\rangle }~)~),$ where $r=1,2,\ldots,n/2.$
\end{tabular}
\vskip 0.2cm
Note that these are the edges from the cycle $\Phi(~\emptyset,\mathcal S~),$ contributing to the matching  in $F_1,$\\  

For any cycle $\Phi(~A\Delta~\theta(~B~),\mathcal S~),$ different from $\Phi(~\emptyset,\mathcal S~)$ in $F_1,$ where $A,B\in H,$ we select the $n$ edges corresponding to the $n$ edges of $\Phi(~\emptyset,\mathcal S~)$ as follows.\\
\begin{tabular}{lll}
$e^1_{r_{A\theta(~B~)}}$&$=$&$ (~A\Delta\overline{\langle n/2-r+1\rangle }\Delta~\theta(~\langle r-1\rangle\Delta B~), A\Delta\overline{\langle n/2-r+2\rangle }\Delta~\theta(~\langle r-1\rangle \Delta B~)~), $\\
$e^1_{n/2+r_{A\theta(~B~)}}$&$=$&$ (~A\Delta\langle n/2-r+1\rangle \Delta~\theta(~\overline{\langle r-1\rangle}\Delta B~), A\Delta\langle n/2-r+2\rangle \Delta~\theta(~\overline{\langle r-1\rangle}\Delta B~)~),$ \\
&&where $r=1,2,\ldots,n/2.$
\end{tabular}
\vskip 0.2cm
Thus we have selected the edges from all the cycles in $F_1$ and they are vertex-disjoint, as the cycles are vertex-disjoint.\\ \\
\textbf{Selection of the edges from the cycles in $F_i$:}
While selecting the edges from the cycles in $F_1,$ we had started with the first edge of every cycle. For $F_2,$ we started the selection with the third edge and in general for $F_i,$ we start with $(~2i-1~)^{st}$ edge of every cycle in $F_i.$ Therefore the edges selected from the first cycle $\Phi(~\langle 2i-2\rangle \Delta~\theta\langle 2i-2\rangle ,\mathcal S~)$ of $F_i$ are as follows.\\ \\
\begin{tabular}{lll}
$e^i_r$&$=$&$ (~\langle 2i-2\rangle \Delta\overline{\langle n/2-r+2i-1\rangle}\Delta~\theta(~\langle r-1\rangle \Delta \langle 2i-2\rangle ~),$\\ &&$\langle 2i-2\rangle \Delta\overline{\langle n/2-r+2i\rangle }\Delta~\theta(~\langle r-1\rangle \Delta\langle 2i-2\rangle ~)~), $\\
$e^i_{n/2+r}$&$=$&$ (~\langle 2i-2\rangle\Delta \langle n/2-r+2i-1\rangle \Delta~\theta(~\overline{\langle r-1\rangle }\Delta\langle 2i-2\rangle ~),$\\ &&$ \langle 2i-2\rangle\Delta\langle n/2-r+2i\rangle \Delta~\theta(~\overline{\langle r-1\rangle}\Delta\langle 2i-2\rangle ~)~),$
\end{tabular}\\ where $r=2i-1,2i,\ldots,n/2+2i-2.$\\ \\
The corresponding edges from the arbitrary cycle $\Phi(~A\Delta\langle 2i-2\rangle \Delta~\theta\langle 2i-2\rangle \Delta~\theta(~B~)~)$ of $F_i$ are as follows.\\ \\
\begin{tabular}{lll}
$e^i_{r_{A\theta(~B~)}}$&$=$&$ (~A\Delta\langle 2i-2\rangle\Delta \overline{\langle n/2-r+2i-1\rangle}\Delta~\theta(~\langle r-1\rangle\Delta \langle 2i-2\rangle\Delta B~),$\\
&&$A\Delta\langle 2i-2\rangle\Delta \overline{\langle n/2-r+2i\rangle}\Delta~\theta(~\langle r-1\rangle\Delta \langle 2i-2\rangle\Delta B~)~), $\\
$e^i_{n/2+r_{A\theta(~B~)}}$&$=$&$(~ A\Delta\langle 2i-2\rangle \Delta\langle n/2-r+2i-1\rangle \Delta~\theta(~\overline{\langle r-1\rangle }\Delta\langle 2i-2\rangle\Delta B~),$\\
&&$ A\Delta\langle 2i-2\rangle \Delta\langle n/2-r+2i\rangle \Delta~\theta(~\overline{\langle r-1\rangle }\Delta\langle 2i-2\rangle \Delta B~)~),$
\end{tabular}\\ where $r=2i-1,2i,\ldots,n/2+2i-2, A,B \in H.$
\vskip 0.2cm
Let $M_i=\{e^i_{r_{A\theta(~B~)}}, e^i_{n/2+r_{A\theta(~B~)}}\colon A,B \in H, r=2i-1,2i,\ldots,n/2+2i-2\}$ for $i=1,2,\ldots,n/4.$  Then each $M_i$ is a matching, as cycles in $F_i$ are vertex-disjoint. We take the union of all these matchings, and call it $M.$ Therefore $ M=M_1\cup M_2\cup\ldots \cup M_{n/4}.$ 
\vskip 0.2cm
\textbf{Selection of the edges from the cycles in $F_i'$:} From any cycle of $F_i,$ we had selected $(~2i-1~)^{st},$ $(~n+2i-1~)^{st},$ $(~2n+2i-1~)^{st},$\ldots,$(~n^2-n+2i-1~)^{st}$ edges, contributing to the matching $M_i.$ Now from any cycle of $F_i',$ we select $(~n/2+2i-1~)^{st}$, $(~n/2+n+2i-1~)^{st}$,$(~n/2+2n+2i-1~)^{st}$,\ldots,$(~n/2+n^2-n+2i-1~)^{st}$ edges to contribute in the matching $M_i',$ satisfying the condition (~I~). Therefore the edges selected from the cycle $\Gamma^i_{A\theta(~B~)}$ in $F_i'$ are as follows.\\
\begin{tabular}{lll}
$f^i_{r_{A\theta(~B~)}}$&$=$&$ (~A\Delta\langle 2i-2\rangle \Delta\langle r-1\rangle \Delta~\theta(~\langle 2i-2\rangle\Delta \langle n/2-r+2i-1\rangle\Delta B~),$\\&&$A\Delta\langle 2i-2\rangle \Delta\langle r-1\rangle\Delta~\theta(~\langle 2i-2\rangle\Delta \langle n/2-r+2i\rangle \Delta B~)~), $\\
$f^i_{n/2+r_{A\theta(~B~)}}$&$=$&$ (~A\Delta\langle 2i-2\rangle \Delta\overline{\langle r-1\rangle }\Delta~\theta(~\langle 2i-2\rangle \Delta\overline{\langle n/2-r+2i-1\rangle }\Delta B~),$\\&&$A\Delta\langle 2i-2\rangle \Delta\overline{\langle r-1\rangle }\Delta~\theta(~\langle 2i-2\rangle\Delta \overline{\langle n/2-r+2i\rangle }\Delta B~)~),$\\&&where $r=2i-1,2i,\ldots,n/2+2i-2.$
\end{tabular}\\
Let $M_i'=\{f^i_{r_{A\theta(~B~)}}, f^i_{n/2+r_{A\theta(~B~)}}\colon A,B \in H, r=2i-1,2i,\ldots,n/2+2i-2\}$ for $i=1,2,\ldots,n/4.$ Then $M_i'$ is a matching in $F_i'.$ Again, we take the union of all these matchings, and call it $M'.$ Therefore $ M'=M_1'\cup M_2'\cup\ldots \cup M_{n/4}'.$ 
\vskip 0.2cm
Note that both $M$ and $M'$ satisfy condition (~I~) by the selection of the edges. Let $\mathcal M=M \cup  M'.$ Then $\mathcal M$ is the collection of the edges selected from $n^2$-cycles in the decomposition of $Q_n,$ satisfying (~I~). It remains to prove that the edges in $\mathcal M$  satisfy also condition (~II~) of the Theorem. \textbf{We prove this in Section 5.} (See Appendix for Illustration.)
\end{proof}

\section{Proof of Main Theorem}
The next theorem serves as the induction step for the proof of Main Theorem 1.6
\begin{theorem}
Let $n \geq 6$ be even.  Suppose $Q_{n-2}$ can be
decomposed into
 $r(~n-2)$-cycles $C_1,C_2,\ldots,C_s$ where $s = 2^{n-m-3}$ such that each  $C_i$ contains $r$ edges
$e_{i1},e_{i2}, \ldots e_{ir}$ whose removal from $C_i$ leaves $r$
vertex-disjoint paths of length $n-3$ and further, the set
$\{e_{ij} \colon  j = 1,2,\ldots,r; ~ i = 1,2, \ldots, s\}$ forms a
perfect matching in $Q_{n-2}.$ Then $Q_n$ can be decomposed into
$rn$-cycles $Z_1,Z_2,\ldots,Z_{4s}$ such that each $Z_i$ contains
$r$ edges  $f_{i1},f_{i2}, \ldots f_{ir}$ whose removal from
$Z_i$ leaves $r$ vertex-disjoint paths of length $n-1$ and further,
the set $\{f_{ij} \colon  j = 1,2,\ldots,r; ~ i = 1,2, \ldots, 4s\}$
forms a perfect matching in $Q_{n}.$
\end{theorem}
\begin{proof}
 Write $Q_n$ as $Q_n=Q_{n-2} \Box ~ Q_{2}.$ One can say that $Q_n$ is obtained from four vertex-disjoint copies $Q_{n-2}^{(~1)}, Q_{n-2}^{(~2)}, Q_{n-2}^{(~3)}, Q_{n-2}^{(~4)}$ of $Q_{n-2}$ such that $Q_{n-2}^p$ is connected to  $Q_{n-2}^{p+1(~mod~ 4)}$ by a matching between their corresponding vertices for $p = 1, 2, 3, 4$ (~See Figure 2.). For $p\in \{1, 2, 3, 4\},$ let $\{C^p_1, C^p_2, \ldots, C^p_s \}$
 be a decomposition of $Q_{n-2}^p$ into $r(~n-2)$-cycles such that
 every $ C_i^p$ contains $r$ edges $ e_{i1}^p, e_{i2}^p, \ldots,
 e_{ir}^p$ whose removal from
  $ C_i^p$ leaves $r$ paths of length $n-3$ each and further, the set $\{ e_{i1}^p, e_{i2}^p, \ldots, e_{ir}^p \colon i=1, 2, \ldots, s\} $ forms
  a perfect matching in $Q_{n-2}^p.$ Then  $s = 2^{n-3}/r.$  
 \begin{figure}[h]
     \begin{tikzpicture} [scale=0.8]
     \draw[fill=black](2,2) circle(.05); 
     \draw[fill=black](2,2.5) circle(.05);
     \draw[fill=black](2,3) circle(.05);
     \draw[fill=black](2,4) circle(.05);
     \draw[fill=black](2,4.5) circle(.05);
     \draw[fill=black](2,5) circle(.05);
     \draw[fill=black](2,5.5) circle(.05);
     \draw[fill=black](2,6.5) circle(.05);
     \draw[fill=black](2,8) circle(.05);
     \draw[fill=black](2,8.5) circle(.05);
     \draw[fill=black](2,9) circle(.05);
      \draw[fill=black](2,9.5) circle(.05);
     \draw[fill=black](2,10.5) circle(.05);
 
 	\draw[fill=black](6,2) circle(.05); 
     \draw[fill=black](6,2.5) circle(.05);
     \draw[fill=black](6,3) circle(.05);
     \draw[fill=black](6,4) circle(.05);
     \draw[fill=black](6,4.5) circle(.05);
     \draw[fill=black](6,5) circle(.05);
     \draw[fill=black](6,5.5) circle(.05);
     \draw[fill=black](6,6.5) circle(.05);
     \draw[fill=black](6,8) circle(.05);
     \draw[fill=black](6,8.5) circle(.05);
     \draw[fill=black](6,9) circle(.05);
      \draw[fill=black](6,9.5) circle(.05);
     \draw[fill=black](6,10.5) circle(.05);
     
    		 \draw[fill=black](10,2) circle(.05); 
         \draw[fill=black](10,2.5) circle(.05);
         \draw[fill=black](10,3) circle(.05);
         \draw[fill=black](10,4) circle(.05);
         \draw[fill=black](10,4.5) circle(.05);
         \draw[fill=black](10,5) circle(.05);
         \draw[fill=black](10,5.5) circle(.05);
         \draw[fill=black](10,6.5) circle(.05);
         \draw[fill=black](10,8) circle(.05);
         \draw[fill=black](10,8.5) circle(.05);
         \draw[fill=black](10,9) circle(.05);
          \draw[fill=black](10,9.5) circle(.05);
         \draw[fill=black](10,10.5) circle(.05);
         
         \draw[fill=black](14,2) circle(.05); 
            \draw[fill=black](14,2.5) circle(.05);
            \draw[fill=black](14,3) circle(.05);
             \draw[fill=black](14,4) circle(.05);
             \draw[fill=black](14,4.5) circle(.05);
             \draw[fill=black](14,5) circle(.05);
             \draw[fill=black](14,5.5) circle(.05);
             \draw[fill=black](14,6.5) circle(.05);
             \draw[fill=black](14,8) circle(.05);
             \draw[fill=black](14,8.5) circle(.05);
             \draw[fill=black](14,9) circle(.05);
              \draw[fill=black](14,9.5) circle(.05);
             \draw[fill=black](14,10.5) circle(.05);

 \draw[loosely dotted,  thin] (2,2)--(6,2)--(10,2)--(14,2);
 \draw[loosely dotted,  thin] (2,2.5)--(6,2.5)--(10,2.5)--(14,2.5);
 \draw[loosely dotted,  thin] (2,3)--(6,3)--(10,3)--(14,3);
 \draw[loosely dotted,  thin] (2,4)--(6,4)--(10,4)--(14,4);
 \draw[loosely dotted,  thin] (2,4.5)--(6,4.5)--(10,4.5)--(14,4.5);
 \draw[loosely dotted,  thin] (2,5)--(6,5)--(10,5)--(14,5);
 \draw[loosely dotted,  thin] (2,5.5)--(6,5.5)--(10,5.5)--(14,5.5);
 \draw[loosely dotted,  thin] (2,6.5)--(6,6.5)--(10,6.5)--(14,6.5);
 \draw[loosely dotted,  thin] (2,8)--(6,8)--(10,8)--(14,8);
 \draw[loosely dotted,  thin] (2,8.5)--(6,8.5)--(10,8.5)--(14,8.5);
 \draw[loosely dotted,  thin] (2,9)--(6,9)--(10,9)--(14,9);
 \draw[loosely dotted,  thin] (2,9.5)--(6,9.5)--(10,9.5)--(14,9.5);
 \draw[loosely dotted,  thin] (2,10.5)--(6,10.5)--(10,10.5)--(14,10.5);
 
 \draw[loosely dotted,  thin] (2,2) to[out=4, in=176] (14,2);
 \draw[loosely dotted,  thin] (2,2.5) to[out=4, in=176] (14,2.5);
 \draw[loosely dotted,  thin] (2,3) to[out=4, in=176] (14,3);
 \draw[loosely dotted,  thin] (2,4) to[out=4, in=176] (14,4);
 \draw[loosely dotted,  thin] (2,4.5) to[out=4, in=176] (14,4.5);
 \draw[loosely dotted,  thin] (2,5) to[out=4, in=176] (14,5);
 \draw[loosely dotted,  thin] (2,5.5) to[out=4, in=176] (14,5.5);
 \draw[loosely dotted,  thin] (2,6.5) to[out=4, in=176] (14,6.5);
 \draw[loosely dotted,  thin] (2,8) to[out=4, in=176] (14,8);
 \draw[loosely dotted,  thin] (2,8.5) to[out=4, in=176] (14,8.5);
 \draw[loosely dotted,  thin] (2,9) to[out=4, in=176] (14,9);
 \draw[loosely dotted,  thin] (2,9.5) to[out=4, in=176] (14,9.5);
 \draw[loosely dotted,  thin] (2,10.5) to[out=4, in=176] (14,10.5);
 
 \draw [thin] (2,2)--(2,2.5)--(2,3);
 \draw [loosely dashed] (2,3)--(2,4);
 \draw [thin] (2,4)--(2,4.5)--(2,5)--(2,5.5);
 \draw [loosely dashed] (2,5.5)--(2,6.5)--(2,8);
 \draw [thin] (2,8)--(2,8.5)--(2,9)--(2,9.5);
 \draw [loosely dashed] (2,9.5)--(2,10.5);
 \draw [thin] (2,2) to[out=100, in=260] (2,10.5);
 
 \draw [thin] (6,2)--(6,2.5)--(6,3);
 \draw [loosely dashed] (6,3)--(6,4);
 \draw [thin] (6,4)--(6,4.5)--(6,5)--(6,5.5);
 \draw [loosely dashed] (6,5.5)--(6,6.5)--(6,8);
 \draw [thin] (6,8)--(6,8.5)--(6,9)--(6,9.5);
 \draw [loosely dashed] (6,9.5)--(6,10.5);
 \draw [thin] (6,2) to[out=100, in=260] (6,10.5);
 
 \draw [thin] (10,2)--(10,2.5)--(10,3);
 \draw [loosely dashed] (10,3)--(10,4);
 \draw [thin] (10,4)--(10,4.5)--(10,5)--(10,5.5);
 \draw [loosely dashed] (10,5.5)--(10,6.5)--(10,8);
 \draw [thin] (10,8)--(10,8.5)--(10,9)--(10,9.5);
 \draw [loosely dashed] (10,9.5)--(10,10.5);
 \draw [thin] (10,2) to[out=100, in=260] (10,10.5);
 
 \draw [thin] (14,2)--(14,2.5)--(14,3);
 \draw [loosely dashed] (14,3)--(14,4);
 \draw [thin] (14,4)--(14,4.5)--(14,5)--(14,5.5);
 \draw [loosely dashed] (14,5.5)--(14,6.5)--(14,8);
 \draw [thin] (14,8)--(14,8.5)--(14,9)--(14,9.5);
 \draw [loosely dashed] (14,9.5)--(14,10.5);
 \draw [thin] (14,2) to[out=100, in=260] (14,10.5);
 
 \draw [densely dotted] (2,6.25) ellipse (1.2 and 5.5);
 \draw [densely dotted] (6,6.25) ellipse (1.2 and 5.5);
 \draw [densely dotted] (10,6.25) ellipse (1.2 and 5.5);
 \draw [densely dotted] (14,6.25) ellipse (1.2 and 5.5);

 \node [right] at (2,1.9){\small{$ v_{i1}^1$}};
 \node [right] at (2,2.6){\small{$ v_{i2}^1$}};
 \node [right] at (2,4.4){\small{$ v_{i(n-1)}^1$}};
 \node [right] at (2,5.1){\small{$ v_{in}^1$}};
 \node [right] at (2,7.9){\small{$ v_{i(r-1)(n-2)}^1$}};
 \node [right] at (2,10.6){\small{$ v_{ir(n-2)}^1$}};
 
 \node [right] at (6,1.9){\small{$ v_{i1}^2$}};
 \node [right] at (6,2.6){\small{$ v_{i2}^2$}};
 \node [right] at (6,4.4){\small{$ v_{i(n-1)}^2$}};
 \node [right] at (6,5.1){\small{$ v_{in}^2$}};
 \node [right] at (6,7.9){\small{$ v_{i(r-1)(n-2)}^2$}};
 \node [right] at (6,10.6){\small{$ v_{ir(n-2)}^2$}};
 
 \node [right] at (10,1.9){\small{$ v_{i1}^3$}};
 \node [right] at (10,2.6){\small{$ v_{i2}^3$}};
 \node [right] at (10,4.4){\small{$ v_{i(n-1)}^3$}};
 \node [right] at (10,5.1){\small{$ v_{in}^3$}};
 \node [right] at (10,7.9){\small{$ v_{i(r-1)(n-2)}^3$}};
 \node [right] at (10,10.6){\small{$ v_{ir(n-2)}^3$}};
 
 \node [right] at (14,1.9){\small{$ v_{i1}^4$}};
 \node [right] at (14,2.6){\small{$ v_{i2}^4$}};
 \node [right] at (14,4.4){\small{$ v_{i(n-1)}^4$}};
 \node [right] at (14,5.1){\small{$ v_{in}^4$}};
 \node [right] at (14,7.9){\small{$ v_{i(r-1)(n-2)}^4$}};
 \node [right] at (14,10.4){\small{$ v_{ir(n-2)}^4$}};
 
 \node [left] at (2.1,2.25){\small{$ e_{i1}^1 \bf \{$}};
 \node [left] at (2.1,4.75){\small{$ e_{i2}^1 \bf \{$}};
 \node [left] at (2.1,8.75){\small{$ e_{ir}^1 \bf \{$}};
 
 \node [left] at (6.1,2.25){\small{$ e_{i1}^2 \bf \{$}};
  \node [left] at (6.1,4.75){\small{$ e_{i2}^2 \bf \{$}};
  \node [left] at (6.1,8.75){\small{$ e_{ir}^2 \bf \{$}};
 
 \node [left] at (10.1,2.25){\small{$ e_{i1}^3 \bf \{$}};
  \node [left] at (10.1,4.75){\small{$ e_{i2}^3 \bf \{$}};
  \node [left] at (10.1,8.75){\small{$ e_{ir}^3 \bf \{$}};
 
  \node [left] at (14.1,2.25){\small{$ e_{i1}^4 \bf \{$}};
   \node [left] at (14.1,4.75){\small{$ e_{i2}^4 \bf \{$}};
   \node [left] at (14.1,8.75){\small{$ e_{ir}^4 \bf \{$}};

 
 \node [above] at (2,10.8){$ C_i^1$};
 \node [above] at (6,10.8){$ C_i^2$};
 \node [above] at (10,10.8){$ C_i^3$};
 \node [above] at (14,10.8){$ C_i^4$};
 
 \node [above] at (2,12) {$ Q_{n-2}^1$} ;
 \node [above] at (6,12) {$ Q_{n-2}^2$} ;
 \node [above] at (10,12) {$ Q_{n-2}^3$};
 \node [above] at (14,12) {$ Q_{n-2}^4$} ;

     \end{tikzpicture}
     \caption{Corresponding $r(n-2)$-cycles in the four copies of $Q_{n-2}$ in $Q_n.$}\label{f3}
     \end{figure}
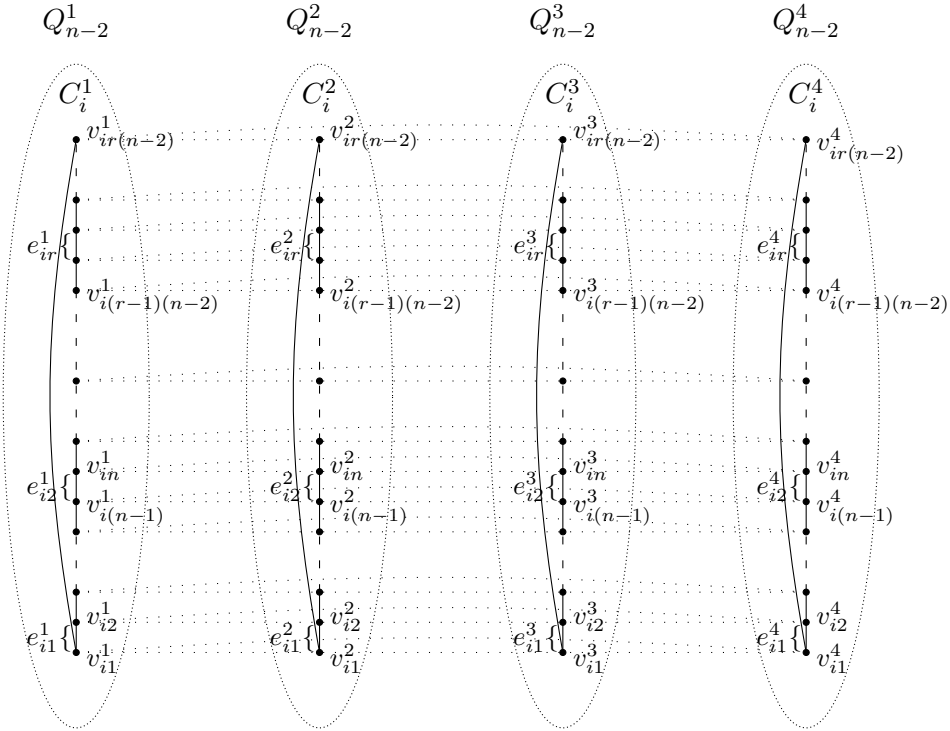
  Since the four
   copies $Q_{n-2}^{(1)}, Q_{n-2}^{(2)}, Q_{n-2}^{(3)}, Q_{n-2}^{(4)}$ are
   vertex-disjoint and together they have $2^n$ vertices,
   $\{ e_{i1}^p, e_{i2}^p,$ $ \ldots, e_{ir}^p \colon i=1, 2, \ldots s ~{\rm and}~p =1, 2, 3, 4\} $
   forms a perfect matching in $Q_n.$

 Let $ q \equiv ~p+ 1~(\rm mod~4).$  We now construct a $rn$-cycle
  $C_{ip}$ by deleting the $r$ edges $ e_{ij}^p$ from $C_i^p,$  and adding the $r$  edges $ e_{ij}^q$ of $C_i^q$
 along with the cross edges between the end vertices of these $2r$ edges as follows.

 Define $ C_{ip} = (C_i^p - \{ e_{ij}^p\colon j=1, ,2 \ldots, r\}) \cup
    \{ e_{ij}^q\colon j=1, ,2 \ldots, r\} \cup \{(v_{ij'}^p, v_{ij'}^q), (v_{ij''}^p, v_{ij''}^q)  \colon j=1, 2 \ldots, r\},$ where $v_{ij'}^p$ and $v_{ij''}^p$ are the end vertices of the edge $e_{ij}^p. $ Thus we get  $rn$-cycles  $C_{i1}, C_{i2}, C_{i3}, C_{i4}$ in
$Q_n$ from the cycles $C_i^1, C_i^2, C_i^3, C_i^4$ and the cross
edges  between  the end  vertices of the edges  $\{ e_{i1}^p,
e_{i2}^p, \ldots, e_{ir}^p \colon p=1, 2, 3, 4\} .$  Also, $C_{i1}
\cup C_{i2} \cup C_{i3} \cup C_{i4} = C_i^1 \cup C_i^2 \cup C_i^3
\cup C_i^4 \cup \{(v_{ij'}^p, v_{ij'}^q), (v_{ij''}^p, v_{ij''}^q)
\colon  j=1, 2 \ldots, r,~ {\rm and}~ p= 1,2,3,4\} $ with $ q \equiv ~p+ 1~(
\rm mod~4).$  Since $C_i^1 ,C_i^2 ,C_i^3, C_i^4$ are mutually
edge-disjoint, $C_{i1}, C_{i2} , C_{i3}, C_{i4}$ are also
edge-disjoint with each other. Further, for each $p \in
\{1,2,3,4\} ,$ $C_{1p} ,C_{2p} ,\ldots,  C_{rp}$ are edge-disjoint
with each other because $C_1^p ,C_2^p , \ldots,  C_r^p$  are
mutually edge-disjoint. Moreover, each $ C_{ip}$ contains the $r$
edges $e_{ij}^q$ for $j =1, 2, \ldots, r.$ Now, $C_i^p - \{
e_{ij}^p\colon j=1, 2 \ldots, r\}$ is union of $r$ vertex-disjoint
paths of length $n-3$ each. If $T$ is one such path with end
vertices $ x^p, y^p,$ then  $ T \cup \{ (x^p, x^q), (y^p, y^q)\}$
is a path of length $ n-1$ in $C_{ip} - \{ e_{ij}^q\colon j=1, 2
\ldots, r\}.$  Thus we get $r$ vertex-disjoint paths of length
$n-1$ each in $C_{ip} - \{ e_{ij}^q\colon j=1, 2 \ldots, r\}$ from $r$
paths of  lengths $n-3$ in $C_i^p - \{ e_{ij}^p\colon j=1, 2 \ldots,
r\}.$

These $4s$ edge-disjoint $rn$-cycles  $C_{ip}$'s altogether have the $
4rns = 4 rn 2^{n-3}/r = n 2^{n-1} $ edges which exhaust the entire
edge set of $Q_n$ and thus decomposes $Q_n.$  Therefore $Q_n$ has
a decomposition into  $rn$-cycles as desired. In the following figure we have 
shown the construction of $rn$-cycles \emph{viz.,} $C_{i1}$ and $C_{i2}$ in $Q_n,$ 
with the help of $r(n-2)$-cycles $C_i^1, C_i^2$ and $C_i^2, C_i^3$ in $Q_{n-2}$ respectively. 
(The cycles $C_{i1}$ and $C_{i2}$ are represented by the solid lines.)\\
    
\begin{figure}[h] 
	\begin{tikzpicture} [scale=1.23]
	\draw[fill=black](2,2) circle(.05); 
	\draw[fill=black](2,2.5) circle(.05);
	\draw[fill=black](2,3) circle(.05);
	\draw[fill=black](2,4) circle(.05);
	\draw[fill=black](2,4.5) circle(.05);
	\draw[fill=black](2,5) circle(.05);
	\draw[fill=black](2,5.5) circle(.05);
	\draw[fill=black](2,6.5) circle(.05);
	\draw[fill=black](2,8) circle(.05);
	\draw[fill=black](2,8.5) circle(.05);
	\draw[fill=black](2,9) circle(.05);
	\draw[fill=black](2,9.5) circle(.05);
	\draw[fill=black](2,10.5) circle(.05);
	
	\draw[fill=black](5,2) circle(.05); 
	\draw[fill=black](5,2.5) circle(.05);
	\draw[fill=black](5,3) circle(.05);
	\draw[fill=black](5,4) circle(.05);
	\draw[fill=black](5,4.5) circle(.05);
	\draw[fill=black](5,5) circle(.05);
	\draw[fill=black](5,5.5) circle(.05);
	\draw[fill=black](5,6.5) circle(.05);
	\draw[fill=black](5,8) circle(.05);
	\draw[fill=black](5,8.5) circle(.05);
	\draw[fill=black](5,9) circle(.05);
	\draw[fill=black](5,9.5) circle(.05);
	\draw[fill=black](5,10.5) circle(.05);
	
	\draw[fill=black](8,2) circle(.05); 
	\draw[fill=black](8,2.5) circle(.05);
	\draw[fill=black](8,3) circle(.05);
	\draw[fill=black](8,4) circle(.05);
	\draw[fill=black](8,4.5) circle(.05);
	\draw[fill=black](8,5) circle(.05);
	\draw[fill=black](8,5.5) circle(.05);
	\draw[fill=black](8,6.5) circle(.05);
	\draw[fill=black](8,8) circle(.05);
	\draw[fill=black](8,8.5) circle(.05);
	\draw[fill=black](8,9) circle(.05);
	\draw[fill=black](8,9.5) circle(.05);
	\draw[fill=black](8,10.5) circle(.05);
	
	\draw[fill=black](11,2) circle(.05); 
	\draw[fill=black](11,2.5) circle(.05);
	\draw[fill=black](11,3) circle(.05);
	\draw[fill=black](11,4) circle(.05);
	\draw[fill=black](11,4.5) circle(.05);
	\draw[fill=black](11,5) circle(.05);
	\draw[fill=black](11,5.5) circle(.05);
	\draw[fill=black](11,6.5) circle(.05);
	\draw[fill=black](11,8) circle(.05);
	\draw[fill=black](11,8.5) circle(.05);
	\draw[fill=black](11,9) circle(.05);
	\draw[fill=black](11,9.5) circle(.05);
	\draw[fill=black](11,10.5) circle(.05);
	
	\draw [dotted] (2,2)--(2,2.5) ;
	\draw [loosely dashed] (2,3)--(2,4) ;
	\draw [dotted] (2,4.5)--(2,5) ;
	\draw [loosely dashed] (2,5.5)--(2,6.5);
	\draw [loosely dashed] (2,7)--(2,7.6) ;
	\draw [dotted] (2,8.5)--(2,9) ;
	\draw [loosely dashed] (2,9.5)--(2,10.5) ;
	
	\draw [dotted] (8,2)--(8,2.5) ;
	\draw [loosely dashed] (8,3)--(8,4) ;
	\draw [dotted] (8,4.5)--(8,5) ;
	\draw [loosely dashed] (8,5.5)--(8,6.5);
	\draw [loosely dashed] (8,7)--(8,7.6) ;
	\draw [dotted] (8,8.5)--(8,9) ;
	\draw [loosely dashed] (8,9.5)--(8,10.5) ;
	
	\draw [loosely dashed] (5,3)--(5,4) ;
	\draw [loosely dashed] (5,5.5)--(5,6.5) ;
	\draw [loosely dashed] (5,7)--(5,7.6) ;
	\draw [loosely dashed] (5,9.5)--(5,10.5) ;
	
	\draw [loosely dashed] (11,3)--(11,4) ;
	\draw [loosely dashed] (11,5.5)--(11,6.5);
	\draw [loosely dashed] (11,7)--(11,7.6) ;
	\draw [loosely dashed] (11,9.5)--(11,10.5) ;

	\draw (2,2)--(5,2)--(5,2.5)--(2,2.5)--(2,3);
	\draw (2,4)--(2,4.5)--(5,4.5)--(5,5)--(2,5)--(2,5.5);
	\draw (2,8)--(2,8.5)--(5,8.5)--(5,9)--(2,9)--(2,9.5);
	\draw (2,2) to[out=100, in=260] (2,10.5);
	
	\draw (8,2)--(11,2)--(11,2.5)--(8,2.5)--(8,3);
	\draw (8,4)--(8,4.5)--(11,4.5)--(11,5)--(8,5)--(8,5.5);
	\draw (8,8)--(8,8.5)--(11,8.5)--(11,9)--(8,9)--(8,9.5);
	\draw (8,2) to[out=100, in=260] (8,10.5);
	
	\draw [densely dotted, thick] (5,2.5)--(5,3);
	\draw [densely dotted, thick] (5,4)--(5,4.5);
	\draw [densely dotted, thick] (5,5)--(5,5.5);
	\draw [densely dotted, thick] (5,8)--(5,8.5);
	\draw [densely dotted, thick] (5,9)--(5,9.5);
	\draw [densely dotted, thick] (5,2) to[out=100, in=260] (5,10.5);
	
	\draw [densely dotted, thick] (11,2.5)--(11,3);
	\draw [densely dotted, thick] (11,4)--(11,4.5);
	\draw [densely dotted, thick] (11,5)--(11,5.5);
	\draw [densely dotted, thick] (11,8)--(11,8.5);
	\draw [densely dotted, thick] (11,9)--(11,9.5);
	\draw [densely dotted, thick] (11,2) to[out=100, in=260] (11,10.5);

	\draw [dotted] (2,3)--(5,3);
	\draw [dotted] (2,4)--(5,4);
	\draw [dotted] (2,5.5)--(5,5.5);
	\draw [dotted] (2,6.5)--(5,6.5);
	\draw [dotted] (2,8)--(5,8);
	\draw [dotted] (2,9.5)--(5,9.5);
	\draw [dotted] (2,10.5)--(5,10.5);
	
	\draw [dotted] (8,3)--(11,3);
	\draw [dotted] (8,4)--(11,4);
	\draw [dotted] (8,5.5)--(11,5.5);
	\draw [dotted] (8,6.5)--(11,6.5);
	\draw [dotted] (8,8)--(11,8);
	\draw [dotted] (8,9.5)--(11,9.5);
	\draw [dotted] (8,10.5)--(11,10.5);

	\node [right] at (2,1.9){\small{$ v_{i1}^1$}};
	\node [right] at (2,2.6){\small{$ v_{i2}^1$}};
	\node [right] at (2,4.4){\small{$ v_{i(n-1)}^1$}};
	\node [right] at (2,5.1){\small{$ v_{in}^1$}};
	\node [right] at (2,7.9){\small{$ v_{i(r-1)(n-2)}^1$}};
	\node [right] at (2,10.6){\small{$ v_{ir(n-2)}^1$}};
	
	\node [right] at (5,1.9){\small{$ v_{i1}^2$}};
	\node [right] at (5,2.6){\small{$ v_{i2}^2$}};
	\node [right] at (5,4.4){\small{$ v_{i(n-1)}^2$}};
	\node [right] at (5,5.1){\small{$ v_{in}^2$}};
	\node [right] at (5,7.9){\small{$ v_{i(r-1)(n-2)}^2$}};
	\node [right] at (5,10.6){\small{$ v_{ir(n-2)}^2$}};
	
	\node [right] at (8,1.9){\small{$ v_{i1}^2$}};
	\node [right] at (8,2.6){\small{$ v_{i2}^2$}};
	\node [right] at (8,4.4){\small{$ v_{i(n-1)}^2$}};
	\node [right] at (8,5.1){\small{$ v_{in}^2$}};
	\node [right] at (8,7.9){\small{$ v_{i(r-1)(n-2)}^2$}};
	\node [right] at (8,10.6){\small{$ v_{ir(n-2)}^2$}};
	
	\node [right] at (11,1.9){\small{$ v_{i1}^3$}};
	\node [right] at (11,2.6){\small{$ v_{i2}^3$}};
	\node [right] at (11,4.4){\small{$ v_{i(n-1)}^3$}};
	\node [right] at (11,5.1){\small{$ v_{in}^3$}};
\node [right] at (11,7.9){\small{$ v_{i(r-1)(n-2)}^3$}};
	\node [right] at (11,10.6){\small{$ v_{ir(n-2)}^3$}};
	
	\node [left] at (2.1,2.25){\small{$ e_{i1}^1 \bf \{$}};
	\node [left] at (2.1,4.75){\small{$ e_{i2}^1 \bf \{$}};
	\node [left] at (2.1,8.75){\small{$ e_{ir}^1 \bf \{$}};
	
	\node [left] at (5.1,2.25){\small{$ e_{i1}^2 \bf \{$}};
		\node [left] at (5.1,4.75){\small{$ e_{i2}^2 \bf \{$}};
		\node [left] at (5.1,8.75){\small{$ e_{ir}^2 \bf \{$}};
	
	\node [left] at (8.1,2.25){\small{$ e_{i1}^2 \bf \{$}};
		\node [left] at (8.1,4.75){\small{$ e_{i2}^2 \bf \{$}};
		\node [left] at (8.1,8.75){\small{$ e_{ir}^2 \bf \{$}};
	
\node [left] at (11.1,2.25){\small{$ e_{i1}^3 \bf \{$}};
	\node [left] at (11.1,4.75){\small{$ e_{i2}^3 \bf \{$}};
	\node [left] at (11.1,8.75){\small{$ e_{ir}^3 \bf \{$}};
	
	\node [above] at (2,10.8){$C_i^1$};
	\node [above] at (5,10.8){$C_i^2$};
	\node [above] at (8,10.8){$C_i^2$};
	\node [above] at (11,10.8){$C_i^3$};
	
	\node [below] at (3.5,1.8){$C_{i1}$};
	\node [below] at (9.5,1.8){$C_{i2}$};
	
	\end{tikzpicture}
	\caption{Construction of $rn$-cycles in $Q_n$ using $r(n-2)$-cycles in $Q_{n-2}$ }\label{f3}
\end{figure}
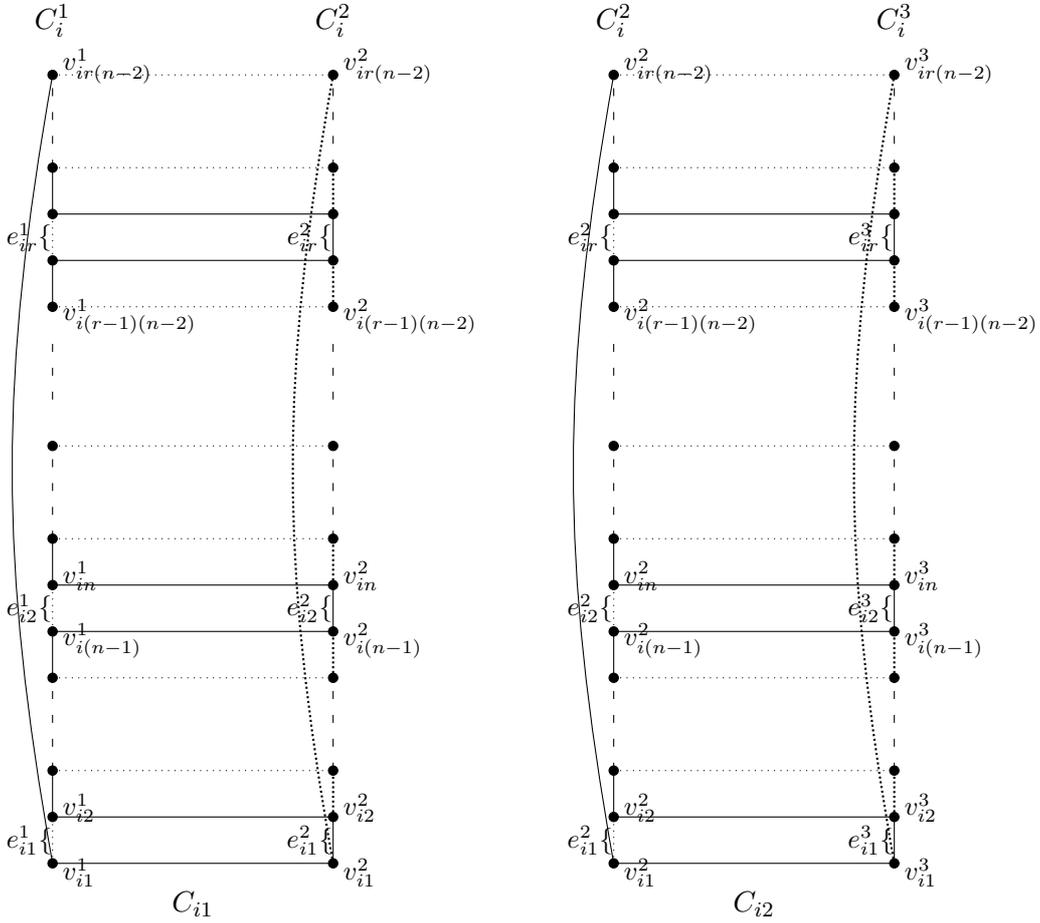     

\end{proof}

\section{Proof for Matching}
The proof of Main Theorem 1.6 follows by applying induction argument on $n.$ In Section 4, we have proved the induction step in Theorem 4.1. However, we have stated the Basis step in Theorem 3.1, and proved its first part. It remains to prove the second part of the theorem, which is crucial. 

In this section, we prove the second part of Theorem 3.1. For the proof we use the properties of the subgroup $H,$ given in several lemmas of Section 2. Also, throughout this section we use the Notation 3.2 introduced in the Step 2 of the proof of Theorem 3.1. 

Recall that in the Step 2 of Theorem 3.1, we had selected the edges from the $2^mn$-cycles of $Q_n,$ satisfying condition (I) of Theorem 3.1. We need to show that those edges also satisfy condition (II) of the theorem. More precisely, we need to prove that the collection $\mathcal{M} = M \cup M' = (M_1\cup M_2 \ldots\cup M_{n/4} ) \cup (M'_1\cup M'_2 \ldots\cup M'_{n/4} ) $ of $ 2^{n-1}$ edges  selected from the cycles of $F_i$ and $ F_{i}'$ in step 2 of Theorem 3.1, is a perfect matching of $Q_n,$ where $n=2^m.$

Recall that $G =\{A \subset \langle 2^{m-1}-1\rangle \colon  |A|~even\}$ is a group w.r.t. symmetric difference $\Delta$ and  $H$ is the subgroup of $G$ that is generated by the collection $\mathsf K =\{~\{j,2^{i-1}+j\}|~j=1,2,\ldots,2^{i-1}-1, ~i=2,\ldots,m-1\}.$  We have proved some properties of $H$ in Lemmas 2.3, 2.5, 2.7 and 2.8 of Section 2. Those properties will be used here.

Initially we prove that the set $M = M_1 \cup M_2 \cup \ldots \cup M_{n/4},$ where $M_i$ is the collection of edges selected from the cycles of $F_i$ in Section 3, is a matching. For that, first we prove $M_1 \cup M_j, j>1$ is a matching. As each $F_i, i=1,2,\ldots,n/4$ is the union of vertex-disjoint cycles, it suffices to prove that the edges selected from the first cycle $\Phi(\emptyset,\mathcal S)$ in $F_1$ are vertex-disjoint with those selected from an arbitrary cycle $\Phi(A\Delta\langle 2j-2\rangle\Delta\theta(\langle 2j-2\rangle\Delta B))$ in $F_j,$ where $A, B \in H$ and $2 \leq j \leq n/4.$

\textbf{Notation 3.2.} \begin{enumerate}
\item [(i)] $\langle 0\rangle=\emptyset.$
\item [(ii)]For $1\leq i \leq n/2,$  $\langle i\rangle =\{1,2,\ldots,i\}=[i]$ (in usual sense)
\item [(iii)]For $1\leq i \leq n/2,$  $\langle n/2+i\rangle =\{i+1,i+2,\ldots,n/2\}=\langle n/2 \rangle\Delta\langle i \rangle.$
\item [(iv)]For $X \subseteq \{1,2,\ldots,n/2\},$ $\overline{X}=\langle n/2 \rangle\Delta X,$ the \emph{complement} of $X$ in $\langle n/2 \rangle.$ 
\item [(v)]For $A=\{a_1,a_2,\ldots,a_r\}\subseteq \langle n/2 \rangle,$ $\theta (A)=\{\theta(a_1),\theta(a_2),\ldots,\theta(a_r)\},$\\ and $\theta( \emptyset)= \theta(\langle 0\rangle )=  \emptyset.$ Also, $\overline{\theta(A)}=\theta(\overline{A}).$
\end{enumerate} 
Therefore $\langle n/2+i\rangle=\overline{\langle i\rangle}$ and $\overline{\langle n/2+i\rangle}=\overline{\overline{\langle i\rangle}}=\langle i\rangle,$ for $1\leq i \leq n/2.$\\ 
 Recall that the $n$ edges selected from the cycle $\Phi^1_\emptyset$ of $F_1$ contributing to $M_1$ are\\
\begin{tabular}{lll}
$e^1_r$&$=$&$ (~\overline{\langle n/2-r+1\rangle }\Delta~\theta(~\langle r-1\rangle ~), \overline{\langle n/2-r+2\rangle }\Delta~\theta(~\langle r-1\rangle ~)~), $\\
$e^1_{n/2+r}$&$=$&$ (~\langle n/2-r+1\rangle \Delta~\theta(~\overline{\langle r-1\rangle }~), \langle n/2-r+2\rangle \Delta~\theta(~\overline{\langle r-1\rangle }~)~),$ where $r=1,2,\ldots,n/2.$
\end{tabular} \\
Also, the $n$ edges selected from the cycle $\Phi^j_{A\Delta~\theta(~B~)}$ of $F_j$ contributing to $M_j$ are\\

\begin{tabular}{lll}
$e^j_{s_{A\Delta~\theta(~B~)}}$&$=$&$ (~A\Delta\langle 2j-2\rangle\Delta\overline{ \langle n/2-s+2j-1\rangle }\Delta~\theta(~\langle s-1\rangle\Delta \langle 2j-2\rangle \Delta B~),$\\
&&$A\Delta\langle 2j-2\rangle\Delta\overline{ \langle n/2-s+2j\rangle }\Delta~\theta(~\langle s-1\rangle\Delta \langle 2j-2\rangle \Delta B~)~), $\\
$e^j_{n/2+s_{A\Delta~\theta(~B~)}}$&$=$&$(~ A\Delta\langle 2j-2\rangle\Delta \langle n/2-s+2j-1\rangle \Delta~\theta(~\overline{\langle s-1\rangle}\Delta \langle 2j-2\rangle \Delta B~),$\\
&&$ A\Delta\langle 2j-2\rangle\Delta \langle n/2-s+2j\rangle \Delta~\theta(~\overline{\langle s-1\rangle}\Delta \langle 2j-2\rangle \Delta B~)~),$ \\ &&  
\end{tabular}\\ where $s=2j-1,2j,\ldots,n/2+2j-2$ and $j=2,3,\ldots,n/4.$

We prove that the edges $e^1_r,e^1_{n/2+r}, e^j_{s_{A\Delta~\theta(~B~)}}, e^j_{n/2+s_{A\Delta~\theta(~B~)}}$ are all vertex-disjoint. 

 To start with, we prove that the first end vertex of $e^1_r$ is different from the first end vertex of $e^j_{s_{A\Delta~\theta(~B~)}}$ in the following lemma:

\begin{lemma}
$\overline{\langle n/2-r+1\rangle }\Delta~\theta(~\langle r-1\rangle ~)\neq A\Delta\langle 2j-2\rangle\Delta\overline{ \langle n/2-s+2j-1\rangle }\Delta~\theta(~\langle s-1\rangle\Delta \langle 2j-2\rangle \Delta B~)$ for any $  A, B \in H , $   $ r \in \{1, 2, \ldots, n/2\},$ $s \in \{2j-1, 2j, \ldots, n/2 + 2j - 2\}$ and $ j \geq 2.$ 
\end{lemma}
\begin{proof} We prove the lemma by contradiction.\\ Suppose $\overline{\langle n/2-r+1\rangle}\Delta~\theta(~\langle r-1\rangle ~) = $ $A\Delta\langle 2j-2\rangle\Delta\overline{ \langle n/2-s+2j-1\rangle }\Delta~\theta(~\langle s-1\rangle\Delta \langle 2j-2\rangle \Delta B~)$ for some  $  A, B \in H , $   $ r \in \{1, 2, \ldots, n/2\},$ $s \in \{2j-1, 2j, \ldots, n/2 + 2j - 2\}$ and $ j \geq 2.$ \\Therefore $A\Delta\langle 2j-2\rangle\Delta\overline{ \langle n/2-s+2j-1\rangle }\Delta\overline{\langle n/2-r+1\rangle}=\theta(~B \Delta\langle 2j-2\rangle \Delta \langle r-1\rangle\Delta \langle s-1\rangle~).$\\ But this is possible only if sets on both the sides of equality are empty, as $LHS \subseteq \{1,2,\ldots,n/2\} ,$ while $RHS \subseteq \theta(~\{1,2,\ldots,n/2\}~)= \{n/2+1,n/2+2,\ldots,n\}.$ Also, by Notation 3.2(v), $\theta(~\emptyset~)=\emptyset.$ Therefore we get \begin{center}
\begin{tabular}{lll}
$A$&$=$&$\langle 2j-2\rangle\Delta\overline{ \langle n/2-r+1\rangle}\Delta\overline{ \langle n/2-s+2j-1\rangle },$\\
$B$&$=$&$\langle 2j-2\rangle\Delta \langle r-1\rangle\Delta \langle s-1\rangle .$
\end{tabular}
\end{center}
Note that by 3.2(iv), $\overline{\langle n/2-r+1\rangle }=\langle n/2\rangle \Delta\langle n/2-r+1\rangle$ and $\overline{\langle n/2-s+2j-1\rangle}=\langle n/2\rangle \Delta\langle n/2-s+2j-1\rangle.$ So we can rewrite $A$ and $B$ as \begin{center}
\begin{tabular}{lll}
$A$&$=$&$\langle 2j-2\rangle\Delta \langle n/2-r+1\rangle\Delta \langle n/2-s+2j-1\rangle ,$\\
$B$&$=$&$\langle 2j-2\rangle\Delta \langle r-1\rangle\Delta \langle s-1\rangle ,$
\end{tabular}
\end{center}
where $  A, B \in H , $   $ r \in \{1, 2, \ldots, n/2\},$ $s \in \{2j-1, 2j, \ldots, n/2 + 2j - 2\}$ and $ j \geq 2.$ \\ \\
We have following two cases.\\
\textbf{(~I~).} $A=\emptyset$ or $B=\emptyset.$\\
Suppose $A=\emptyset.$ Then $\langle 2j-2\rangle = \langle n/2-r+1\rangle\Delta \langle n/2-s+2j-1\rangle. $ Here all the three sets start with $1$ and all are subsets of $\{1,2,\ldots,n/2\}.$ Therefore either $n/2-r+1=2j-2$ and $n/2-s+2j-1=0,$ or $n/2-r+1=0$ and $n/2-s+2j-1=2j-2.$ But $n/2-s+2j-1=0$ is not possible, as $j\geq 2$ and $2j-1\leq s \leq n/2+2j-2.$ So $n/2-r+1=0$ and $n/2-s+2j-1=2j-2,$ which gives $r=n/2+1=s.$ Due to this, $B=\langle 2j-2\rangle\Delta \langle r-1\rangle\Delta \langle s-1\rangle =\langle 2j-2\rangle,$ a contradiction to $B \in H$ by Lemma 2.3(1). Similarly, if $B=\emptyset,$ then we get a contradiction to $A \in H.$\\ \\
\textbf{(~II~).} Both $A$ and $B$ are non-empty.\vskip.2cm\noindent
\textbf{\emph{Claim(a)}.} $2 \leq r \leq n/2$ and $2j \leq s \leq n/2.$\\
Suppose $r=1$ and $s=2j-1.$ Then $A$ becomes a consecutive string giving contradiction to $A\in H$ by Lemma 2.3(1). If $r=1$ and  $s \neq 2j-1,$ then $n/2-r+1=n/2$ and $n/2-s+2j-1 < n/2.$ Also, $2\leq j \leq n/4$ implies $2j-2\leq n/2-2.$ So $A$ will contain $n/2,$ a contradiction by Lemma 2.5(1). Suppose $r\neq 1$ and $s=2j-1.$ Then $n/2-r+1<n/2$ and $n/2-s+2j-1 = n/2.$ Hence $A$ will contain $n/2,$ a contradiction by Lemma 2.5(1). Therefore, $r \neq 1, s \neq 2j-1.$ It remains to show that $s \leq n/2.$ If $s \geq n/2+1,$ then $\langle s-1\rangle$ and hence, $B$ will contain $n/2,$  a contradiction to $B \in H$ by Lemma 2.5(1). Therefore $2 \leq r \leq n/2$ and $2j \leq s \leq n/2$ proving the claim.   \vskip.2cm\noindent
We list some of the values of $r-1,s-1,n/2-r+1$ and $n/2-s+2j-1$ for the ready reference as follows.\vskip.2cm\noindent
\hskip2cm\begin{tabular}{|c|c|c|}
\hline
$r$ & $r-1$ & $n/2-r+1$ \\
\hline
$2$ & $1$ & $n/2-1$ \\
$3$ & $2$ & $n/2-2$ \\
$4$ & $3$ & $n/2-3$ \\
\vdots &\vdots &\vdots \\
$n/2$ & $n/2-1$ & $1$ \\
\hline
\end{tabular}
\hskip2cm\begin{tabular}{|c|c|c|}
\hline
 $s$ & $s-1$ & $n/2-s+2j-1$\\
\hline
 $2j$ & $2j-1$ & $n/2-1$\\
$2j+1$ & $2j$ & $n/2-2$\\
 $2j+2$ & $2j+1$ & $n/2-3$\\
\vdots &\vdots &\vdots\\
 $n/2$ & $n/2-1$ & $2j-1$\\
\hline
\end{tabular}\\
\vskip.2cm\noindent
\textbf{\emph{Claim(b)}.} $s-(~2j-2~)<r<s.$\\
Suppose $s'=s-(~2j-2~).$ So our claim reduces to $s'<r<s.$\\ As $2j \leq s \leq n/2,$ we have $2\leq s' \leq n/2-(~2j-2~).$ Also, $n/2-s+2j-1=n/2-s'+1.$\vskip.2cm\noindent Therefore $A=\langle 2j-2\rangle\Delta \langle n/2-r+1\rangle\Delta \langle n/2-s+2j-1\rangle=\langle 2j-2\rangle\Delta \langle n/2-r+1\rangle\Delta \langle n/2-s'+1\rangle.$\vskip.2cm\noindent If $r=s',$  then $A=\langle 2j-2\rangle\Delta \langle n/2-r+1\rangle\Delta \langle n/2-s'+1\rangle=\langle 2j-2\rangle,$  a consecutive string and hence, a contradiction by Lemma2.3(1).\vskip.2cm\noindent Suppose $r<s'.$ Then $r<s'+(~2j-2~)=s,$ giving $s-r>2j-2.$ So $B=\langle 2j-2\rangle\Delta \langle r-1\rangle\Delta \langle s-1\rangle =\langle 2j-2\rangle\Delta\{r,\ldots,s-1\}=B_1\cup B_2$(~say~), where $B_1,B_2$ are consecutive strings, such that $B_1$ starts with $1$ and $|B_1|<|B_2|.$ This gives a contradiction by Lemma 2.7(3), as $B$ itself can not be a consecutive string being an element of $H.$ \vskip.2cm\noindent Therefore $s'<r.$ 
We get contradiction to $A \in H$ on the similar lines as above, if $r \geq s.$ So $r<s.$ So the claim is proved.
\vskip.2cm\noindent
By \emph{Claim(b)} above, $n/2-s'+1 > n/2-r+1$. So $\langle n/2-r+1\rangle\Delta \langle n/2-s'+1\rangle=\{n/2-r+2,\ldots,n/2-s'+1\}.$ Similarly for $B,$ $\langle r-1\rangle\Delta \langle s-1\rangle=\{r,\ldots,s-1\}.$ Therefore \begin{center}
\begin{tabular}{lll}
$A$&$=$&$\langle 2j-2\rangle\Delta\{n/2-r+2,\ldots,n/2-s'+1\} ,$\\
$B$&$=$&$\langle 2j-2\rangle\Delta\{r,\ldots,s-1\}.$
\end{tabular}
\end{center}
Here onwards, we say $A_1=\langle 2j-2\rangle=B_1,$ $A_2=\{n/2-r+2,\ldots,n/2-s'+1\}$ and $B_2= \{r,\ldots,s-1\}.$ So $A=A_1\Delta A_2$ and $B=B_1\Delta B_2.$
\vskip.2cm\noindent
\textbf{\emph{Claim(c).}} $2j-2 < n/4.$ \\
Suppose if possible, $2j-2\geq n/4.$ Then $n/4 \in A_1=B_1=\langle 2j-2\rangle.$ So by Lemma 2.5(1), one must have $n/4 \in A_2=\{n/2-r+2,\ldots,n/2-s'+1\} $ and $n/4 \in B_2=\{r,\ldots,s-1\}. $ Now $n/4 \in A_2$ implies $n/4 \geq n/2-r+2,$ giving $r \geq n/4+2.$ But this gives a contradiction to $r \leq n/4,$ as $n/4\in B_2.$ So one must have $2j-2<n/4.$ 
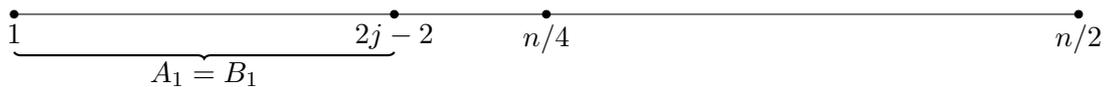
\begin{figure}[h] 
	\begin{tikzpicture} 
	\draw (1,1)--(15,1);
	\draw[fill=black](1,1) circle(.05); 
	\draw[fill=black](15,1) circle(.05);
	\draw[fill=black](8,1) circle(.05);
	\draw[fill=black](6,1) circle(.05);

	\node [below] at (1,1){$1$};
	\node [below] at (15,1){$n/2$};
	\node [below] at (8,1){$n/4$};
	\node [below] at (6,1){$2j-2$};
		\node [below] at (3.5,0.5){$A_1=B_1$};
		\draw [
		    thick,
		    decoration={
		        brace,
		        mirror,
		        raise=0.5cm
		    },
		    decorate
		] (1,1) -- (6,1) ;	
		
	\end{tikzpicture}
\end{figure}     
\vskip.2cm\noindent So far, we got $2 \leq r \leq n/2,$ $2j \leq s \leq n/2,$ $s'=s-(~2j-2~)<r<s$ and $2j-2<n/4,$ where $j \geq 2.$ Now we discuss the cases related to values of $s'$ as follows, which will end the proof.
\vskip.2cm\noindent
\textbf{\emph{Case 1.}} $s'=n/4.$\\
Note that $A=\langle 2j-2\rangle\Delta\{n/2-r+2,\ldots,n/2-s'+1\} \in H,$ where $|A_1|=|\langle 2j-2\rangle|$ is even. So $|A_2|=|\{n/2-r+2,\ldots,n/2-s'+1\}|$ must be even and non-zero. So $n/2-r+2\neq n/2-s'+1.$ Moreover, as $s'=n/4,$ $n/2-s'+1=n/4+1.$ That means $n/4\in A_2.$ Recall that $2j-2<n/4,$  and hence $n/4\notin A_1.$ But this implies that $n/4 \in A,$  a contradiction to $A \in H.$
\vskip.2cm\noindent
\textbf{\emph{Case 2.}} $s'>n/4.$\\
\begin{figure}[h] 
	\begin{tikzpicture} 
	\draw (1,1)--(15,1);
	\draw[fill=black](1,1) circle(.05); 
	\draw[fill=black](15,1) circle(.05);
	\draw[fill=black](8,1) circle(.05);
	\draw[fill=black](6,1) circle(.05);
	\draw[fill=black](9,1) circle(.05);
	\draw[fill=black](11,1) circle(.05);
	\draw[fill=black](14,1) circle(.05);
	\draw[fill=black](13,1) circle(.05);

	\node [below] at (1,1){$1$};
	\node [below] at (15,1){$n/2$};
	\node [below] at (8,1){$n/4$};
	\node [below] at (6,1){$2j-2$};
	\node [below] at (9,1){$s'$};
	\node [below] at (11,1){$r$};
	\node [below] at (14,1){$s$};
	\node [below] at (13,1){$s-1$};
	\node [below] at (3.5,0.5){$A_1=B_1$};
	\draw [
	    thick,
	    decoration={
	        brace,
	        mirror,
	        raise=0.5cm
	    },
	    decorate
	] (1,1) -- (6,1) ;	
	\node [below] at (12,0.5){$B_2$};
			\draw [
			    thick,
			    decoration={
			        brace,
			        mirror,
			        raise=0.5cm
			    },
			    decorate
			] (11,1) -- (13,1) ;

	\end{tikzpicture}
\end{figure}     
As $s'<r<s$ and $2j-2<n/4,$ $B=B_1\cup B_2.$\\
\emph{Claim.} $A=A_1\cup A_2.$\\
As $s'>n/4,$ $n/2-s'+1<n/4+1$ and hence $n/2-s'+1\leq n/4.$ In fact, $n/2-s'+1 < n/4.$ For, if $n/2-s'+1=n/4,$ then $A\in H$ is such that $A\subseteq \langle n/4 \rangle$ and $A$ contains $n/4.$ This is a contradiction by Lemma 2.5(1).  Let $k$ be the largest integer such that the largest element of $A$ is less than $n/2^k.$ So both $n/2^k$ and $n/2^{k+1}$ are not in $A.$ To prove $A=A_1\cup A_2$ is equivalent to prove that $A_1\cap A_2=\emptyset.$ Suppose $A_1\cap A_2\neq\emptyset.$ Then we must have $n/2-r+2\leq 2j-2.$ Depending on $n/2-s'+1$ and nature of $A$ being an element of $H,$ we have following two subcases.\\
\emph{Subcase (i).} $n/2-s'+1<2j-2.$ In this case, $A=\{1,\ldots,n/2-r+1\}\cup\{n/2-s'+2,\ldots,2j-2\}.$ As both $n/2^k$ and $n/2^{k+1}$ are not in $A$ and $A_1$ starts with $1,$ we should have $n/2-r+1<n/2^{k+1}<n/2-s'+2<2j-2<n/2^k.$ Here $n/2-r+1<n/2^{k+1}$ gives $r>n/2-n/2^{k+1}+1=n/4+n/8+\ldots+n/2^{k+1}+1.$ Moreover, $n/2^{k+1}<2j-2$ implies that $B_1=\{1,\ldots,2j-2\}$ will contain $n/2^{k+1}.$ So for $B\in H,$ we must have $n/2^{k+1}+n/2^k+\ldots+n/8+n/4\in B_2$ by Lemma 2.5(2). This gives $r \leq n/4+n/8+\ldots+n/2^{k+1},$ a contradiction.\\
\emph{Subcase (ii).} $n/2-r+2<2j-2<n/2-s'+1.$ In this case, $A=\{1,\ldots,n/2-r+1\}\cup\{2j-1,\ldots,n/2-s'+1\}$ and it satisfies  the inequality $n/2-r+1<n/2^{k+1}\leq 2j-2.$ So again by the same argument as in \emph{Subcase(i)}, we get contradiction.\\
Therefore $A=A_1\cup A_2,$ proving the claim.
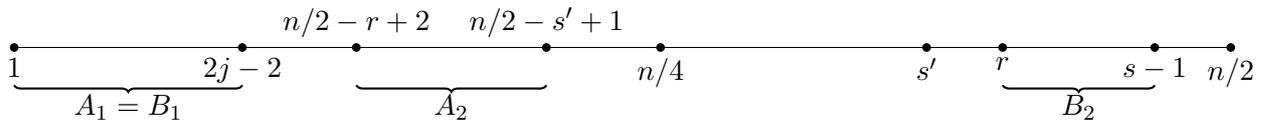
\begin{figure}[h] 
	\begin{tikzpicture} 
	\draw (0,1)--(16,1);
	\draw[fill=black](0,1) circle(.05); 
	\draw[fill=black](16,1) circle(.05);
	\draw[fill=black](8.5,1) circle(.05);
	\draw[fill=black](3,1) circle(.05);
	\draw[fill=black](12,1) circle(.05);
	\draw[fill=black](13,1) circle(.05);
	\draw[fill=black](15,1) circle(.05);
	\draw[fill=black](4.5,1) circle(.05);
	\draw[fill=black](7,1) circle(.05);	
	 
	\node [below] at (0,1){$1$};
	\node [below] at (16,1){$n/2$};
	\node [below] at (8.5,1){$n/4$};
	\node [below] at (3,1){$2j-2$};
	\node [below] at (12,1){$s'$};
	\node [below] at (13,1){$r$};
	\node [below] at (15,1){$s-1$};
	\node [above] at (4.5,1){$n/2-r+2$};
	\node [above] at (7,1){$n/2-s'+1$};
	
	\node [below] at (1.5,0.5){$A_1=B_1$};
\draw [
    thick,
    decoration={
        brace,
        mirror,
        raise=0.5cm
    },
    decorate
] (0,1) -- (3,1) ;
\node [below] at (14,0.5){$B_2$};
			\draw [
			    thick,
			    decoration={
			        brace,
			        mirror,
			        raise=0.5cm
			    },
			    decorate
			] (13,1) -- (15,1) ;	
	\node [below] at (5.75,0.5){$A_2$};
				\draw [
				    thick,
				    decoration={
				        brace,
				        mirror,
				        raise=0.5cm
				    },
				    decorate
				] (4.5,1) -- (7,1) ;
	
	\end{tikzpicture}
\end{figure} 
\vskip.2cm\noindent
Recall that both $A$ and $B$ in $H$ are non-empty, $2\leq r \leq n/2,$ $2j \leq s \leq n/2,$ $s-(~2j-2~)=s'<r<s.$ Also, $A=A_1\cup A_2=\{1,\ldots,2j-2\}\cup\{n/2-r+2,\ldots,n/2-s'+1\}$ and $B=B_1\cup B_2=\{1,\ldots,2j-2\}\cup\{r,\ldots,s-1\}.$\\
Let $k$ be the largest such that $A \subseteq \{1,2,\ldots,n/2^k\}.$ So by Lemma 2.5(1), $n/2^{k}$ and $n/2^{k+1}$ both do not belong to $A.$ As $A_1$ starts with $1$ and $A_1\cap A_2=\emptyset,$ we have $2j-2<n/2^{k+1}<n/2-r+2.$ Note that $2j-2=|A_1|>|A_2|=r-s'.$ Therefore we can not have $A_1+n/2^{k+1}=A_2.$ (~Here by $A_1+n/2^{k+1},$ we mean $n/2^{k+1}$ is added to every element of set $A_1.$~) But we wish to get an element in $H,$ whose symmetric difference with $A$ will lead to union of two consecutive strings that differ by a power of $2,$ as in Lemma 2.7(1). 
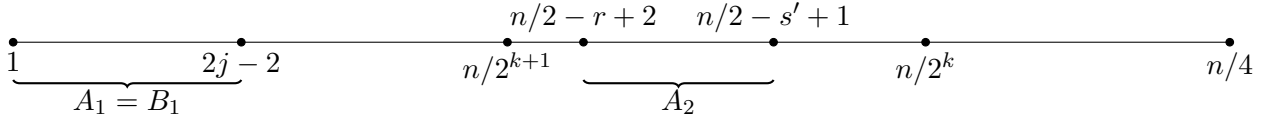
\begin{figure}[h] 
	\begin{tikzpicture} 
	\draw (0,1)--(16,1);
	\draw[fill=black](0,1) circle(.05); 
	\draw[fill=black](16,1) circle(.05);
	\draw[fill=black](10,1) circle(.05);
	\draw[fill=black](3,1) circle(.05);
	\draw[fill=black](12,1) circle(.05);
	\draw[fill=black](6.5,1) circle(.05);
	\draw[fill=black](7.5,1) circle(.05);	
	 
	\node [below] at (0,1){$1$};
	\node [below] at (16,1){$n/4$};
	\node [above] at (10,1){$n/2-s'+1$};
	\node [below] at (3,1){$2j-2$};
	\node [below] at (12,1){$n/2^k$};
	\node [below] at (6.5,1){$n/2^{k+1}$};
	\node [above] at (7.5,1){$n/2-r+2$};
	
	\node [below] at (1.5,0.5){$A_1=B_1$};
\draw [
    thick,
    decoration={
        brace,
        mirror,
        raise=0.5cm
    },
    decorate
] (0,1) -- (3,1) ;

	\node [below] at (8.75,0.5){$A_2$};
				\draw [
				    thick,
				    decoration={
				        brace,
				        mirror,
				        raise=0.5cm
				    },
				    decorate
				] (7.5,1) -- (10,1) ;
	
	\end{tikzpicture}
\end{figure} 
\vskip.2cm\noindent
Let $S_t=n/2^{k+1}+n/2^{k+2}+\ldots+n/2^{k+t},$ where $t$ be the smallest, such that $A'=A\Delta (~A_2\cup(~A_2-S_t~)~)$ satisfies $n/2-r+2-S_t<n/2^{k+t}<2j-1<n/2-s'+1-S_t.$ So $A'=\{1,\ldots,n/2-r+1-S_t\}\cup\{2j-1,\ldots,n/2-s'+1-S_t\} = A_1'\cup A_2'$(~say~). Again, $|A_1'|=n/2-r'+1-S_t$ and $|A_2'|=n/2-s'+1-S_t-(~2j-1~)+1=n/2-s+1-S_t,$ as $s'=s-(~2j-2~).$ As $r<s,$ $n/2-r+1-S_t>n/2-s+1-S_t,$ giving $|A_1'|>|A_2'|.$ So again $A_1'+n/2^{k+t} \neq A_2'$ and we will have to repeat the above procedure. Let $S_u=n/2^{k+t+1}+n/2^{k+t+2}+\ldots+n/2^{k+t+u},$ where $u$ be the smallest integer such that $A''=A'\Delta(~A_2'\cup (~A_2'-S_u~)~)$ satisfies $2j-1-S_u<n/2-r+1-S_t<n/2-s'+1-S_t-S_u.$ Therefore $A''=\{1,\ldots,2j-2-S_u\}\cup\{n/2-r+1-S_t,\ldots,n/2-s'+1-S_t-S_u\}=A_1''\cup A_2''.$ Again, as $|A_1''|>|A_2''|,$ we can't get $A_1''+n/2^{k+t+u}=A_2''.$ In fact, this continues, and ultimately we will be left with a single consecutive string, giving contradiction to Lemma 2.7(1). So for $s'>n/4,$ we get contradiction.
\vskip.2cm\noindent
\textbf{\emph{Case 3.}} $s'<n/4.$

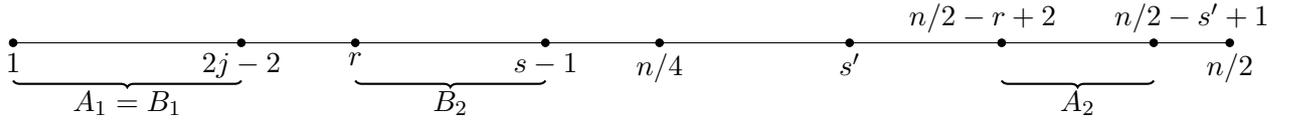
\begin{figure}[h] 
	\begin{tikzpicture} 
	\draw (0,1)--(16,1);
	\draw[fill=black](0,1) circle(.05); 
	\draw[fill=black](16,1) circle(.05);
	\draw[fill=black](8.5,1) circle(.05);
	\draw[fill=black](3,1) circle(.05);
	\draw[fill=black](11,1) circle(.05);
	\draw[fill=black](13,1) circle(.05);
	\draw[fill=black](15,1) circle(.05);
	\draw[fill=black](4.5,1) circle(.05);
	\draw[fill=black](7,1) circle(.05);	
	 
	\node [below] at (0,1){$1$};
	\node [below] at (16,1){$n/2$};
	\node [below] at (8.5,1){$n/4$};
	\node [below] at (3,1){$2j-2$};
	\node [below] at (11,1){$s'$};
	\node [above] at (12.75,1){$n/2-r+2$};
	\node [above] at (15.5,1){$n/2-s'+1$};
	\node [below] at (4.5,1){$r$};
	\node [below] at (7,1){$s-1$};
	
	\node [below] at (1.5,0.5){$A_1=B_1$};
\draw [
    thick,
    decoration={
        brace,
        mirror,
        raise=0.5cm
    },
    decorate
] (0,1) -- (3,1) ;
\node [below] at (14,0.5){$A_2$};
			\draw [
			    thick,
			    decoration={
			        brace,
			        mirror,
			        raise=0.5cm
			    },
			    decorate
			] (13,1) -- (15,1) ;	
	\node [below] at (5.75,0.5){$B_2$};
				\draw [
				    thick,
				    decoration={
				        brace,
				        mirror,
				        raise=0.5cm
				    },
				    decorate
				] (4.5,1) -- (7,1) ;
	
	\end{tikzpicture}
\end{figure} 
By the similar arguments as given in Case 2 above, and interchanging the roles of sets $A$   and $B,$ we get contradiction.
This completes the proof. (See Appendix for Illustration.)
\end{proof}

Lemma 5.1 proves that the first vertex of the edge $e^1_r$ is different from the first vertex of the edge $e^j_{s_{A\Delta~\theta(~B~)}}.$ In the following lemma, we prove that the end vertices of the edges $e^1_r,e^1_{n/2+r}, e^j_{s_{A\Delta~\theta(~B~)}}, e^j_{n/2+s_{A\Delta~\theta(~B~)}}$ are all distinct.

\begin{lemma}
The edges $e^1_r,e^1_{n/2+r}, e^j_{s_{A\Delta~\theta(~B~)}}, e^j_{n/2+s_{A\Delta~\theta(~B~)}}$ are vertex-disjoint for any $  A, B \in H , $   $ r \in \{1, 2, \ldots, n/2\},$ $s \in \{2j-1, 2j, n/2 + 2j - 2\}$ and $ j \geq 2.$ 
\end{lemma}
\begin{proof} 
Being edges in the matching $M_1,$ $e^1_r,e^1_{n/2+r}$ are vertex-disjoint. Similarly, the edges $e^j_{s_{A\Delta~\theta(~B~)}}, e^j_{n/2+s_{A\Delta~\theta(~B~)}}$ are vertex-disjoint. We need to prove that the end vertices of the edges $e^1_r,e^1_{n/2+r}, e^j_{s_{A\Delta~\theta(~B~)}}, e^j_{n/2+s_{A\Delta~\theta(~B~)}}$ are all distinct. For convenience, we call $e^1_r=(~X_1,X_2~),$ $ e^1_{n/2+r} = (~X_3,X_4~), e^j_{s_{A\Delta~\theta(~B~)}}= (~Y_1,Y_2~)$ and $e^j_{n/2+s_{A\Delta~\theta(~B~)}}=(~Y_3,Y_4~).$ Therefore we have to prove that $X_p \neq Y_q,$ for all $p,q=1,2,3,4.$ 

We first prove that $Y_1 \neq X_p$ for any $p \in \{1,2,3,4\}.$

 By Lemma 5.1, $Y_1 \neq X_1.$ 

Suppose $Y_1=X_2.$ Then $\overline{\langle n/2-r+2\rangle }\Delta~\theta(~\langle r-1\rangle ~) = A\Delta\langle 2j-2\rangle\Delta\overline{ \langle n/2-s+2j-1\rangle}\Delta$ $ \theta(~\langle s-1\rangle\Delta \langle 2j-2\rangle \Delta B~)$ for some $  A, B \in H , $   $ r \in \{1, 2, \ldots, n/2\}$ $s \in \{2j-1, 2j,\ldots, n/2 + 2j - 2\}$ and $ j \geq 2.$ Therefore $A=\langle 2j-2\rangle\Delta\overline{ \langle n/2-r+2\rangle}\Delta\overline{ \langle n/2-s'+1\rangle },  B=\langle 2j-2\rangle\Delta \langle r-1\rangle \Delta\langle s-1\rangle .$ Let $A_1=\langle 2j-1\rangle ,$ $A_2=\overline{\langle n/2-r+2\rangle }\overline{\langle n/2-s'+1\rangle}$ and $B_2=\langle r-1\rangle\Delta \langle s-1\rangle .$ Then $A$ is the symmetric difference of $A_1,A_2,$ and $B$ is the symmetric difference of $A_1,B_2.$ Since $|A|, |B|$ and$|A_1|$ are even, $|A_2|=|(~r-1~)-s'|$ and $|B_2|=|s-r|$ are also even. Therefore $s'$ and $r-1$ have same the parity and also, $r$ and $s$ have the same parity. As $s=s'+2j-2,$ $s$ and $s'$ have the same parity. Hence $r$ and $r-1$ have the same parity, a contradiction. Therefore $Y_1 \neq X_2.$

Suppose $Y_1=X_3.$ Then $\langle n/2-r+1\rangle \Delta~\theta(~\overline{\langle r-1\rangle }~) =  A\Delta\langle 2j-2\rangle\Delta\overline{ \langle n/2-s+2j-1\rangle }\Delta$ $\theta(~\langle s-1\rangle\Delta \langle 2j-2\rangle \Delta B~).$ Consequently, $A=\langle 2j-2\rangle\Delta \langle n/2-r+1\rangle\Delta\overline{ \langle n/2-s+2j-1\rangle }$  and $ B=\langle 2j-2\rangle\Delta\overline{ \langle r-1\rangle}\Delta \langle s-1\rangle .$ If $2j-1 \leq s \leq n/2,$ then $B$ contains $n/2,$ and if $n/2+1 \leq s \leq n/2+2j-2,$ then $A$ contains $n/2,$ a contradiction by Lemma 2.5(1). Therefore $Y_1 \neq X_3.$

Suppose $Y_1=X_4.$ Then $\langle n/2-r+2\rangle \Delta~\theta(~\overline{\langle r-1\rangle }~)= A\Delta\langle 2j-2\rangle\Delta\overline{ \langle n/2-s+2j-1\rangle }\Delta$ $\theta(~\langle s-1\rangle\Delta \langle 2j-2\rangle \Delta B~).$ Hence $A=\langle 2j-2\rangle\Delta \langle n/2-r+2\rangle\Delta\overline{ \langle n/2-s+2j-1\rangle }, $ and $ B=\langle 2j-2\rangle\Delta\overline{ \langle r-1\rangle}\Delta \langle s-1\rangle .$ If $2j-1 \leq s \leq n/2+1,$ then $B$ is a consecutive string or contains $n/2,$ and if $n/2+2 \leq s \leq n/2+2j-2,$ then $A$ is a consecutive string or contains $n/2,$ a contradiction. Therefore $Y_1 \neq X_4.$

Thus we have proved that $Y_1 \neq X_p$ for any $p \in \{1,2,3,4\}.$

 Note that $X_3=Y_3$ gives $A=\langle 2j-2\rangle\Delta \langle n/2-r+1\rangle\Delta \langle n/2-s+2j-1\rangle  = \langle 2j-2\rangle\Delta\overline{ \langle n/2-r+1\rangle }$ $\Delta\overline{\langle n/2-s+2j-1\rangle }$ and $B=\langle 2j-2\rangle\Delta\overline{ \langle r-1\rangle }\Delta\overline{\langle s-1\rangle }=\langle 2j-2\rangle\Delta \langle r-1\rangle\Delta \langle s-1\rangle $ by definition of $\overline{\langle i\rangle }.$ But this implies $X_1=Y_1,$ a contradiction by Lemma 5.1. Similarly, $X_4=Y_4$ implies $X_2=Y_2.$ But $X_2=Y_2$ implies $\overline{\langle n/2-r+2\rangle }\Delta~\theta(~\langle r-1\rangle ~) = A\Delta\langle 2j-2\rangle \Delta\overline{\langle n/2-s+2j\rangle }\Delta~\theta(~\langle s-1\rangle\Delta$ $  \langle 2j-2\rangle \Delta B~).$ So $A=\langle 2j-2\rangle\Delta\overline{ \langle n/2-r+2\rangle }\Delta\overline{\langle n/2-s+2j\rangle }$ and $B=\langle 2j-2\rangle\Delta  \langle r-1\rangle\Delta \langle s-1\rangle .$ This gives a contradiction by replacing the arguments for $r,s$ and $s'$ by $r-1,s-1$ and $s'-1,$ respectively for $A$ and making corresponding changes in Lemma 5.1. Therefore $X_p \neq Y_p, p \in \{2,3,4\}.$

Proof of the case $X_1 \neq Y_2$ follows from that of the case $X_2 \neq Y_1,$ replacing the arguments for $r,s$ and $s'$ by $r-1,s-1$ and $s'-1,$ respectively. Moreover, $X_3=Y_4$ implies $X_1=Y_2$ and $X_4=Y_3$ implies $X_2=Y_1,$ a contradiction.

On the similar lines as above,  $X_3 \neq Y_1$ implies $X_1 \neq Y_3,$ $X_4 \neq Y_2$ and $X_2 \neq Y_4.$  While, $X_4 \neq Y_1$ implies $X_1 \neq Y_4,$ $X_3 \neq Y_2$ and $X_2 \neq Y_3.$

Hence, the edges $e^1_r,e^1_{n/2+r}, e^i_s, e^j_{n/2+s}$ are all vertex-disjoint.

\end{proof}

\begin{lemma}
$ M=M_1\cup M_2\cup\ldots \cup M_{n/4}$ is a matching in $Q_n.$
\end{lemma}
\begin{proof}
We have already proved that $M_1\cup M_j$ is a matching in Lemma 5.2. Now for $i,j\in\{2,3,\ldots,n/4\}$ and $i<j$ we need prove that $M_i \cup M_j$ is a matching. From Lemma 2.3(3), recall that  $H_i=\langle 2i-2\rangle H, i=1,2,\ldots,n/4$ are cosets of $H,$ which partition group $(~G;\Delta~)$ of even subsets of $\langle n/2-1\rangle .$ Moreover, by Remark 2.11, $H_i \times H_i$ is a coset of $G\times G$ and is given by $H_i\times H_i=\langle 2i-2\rangle \Delta~\theta(~\langle 2i-2\rangle ~)(~H_1\times H_1~).$ By Step 1 in the proof of Theorem 3.1, $W_1=\displaystyle\biguplus_{A\in H_1}C(~A,S~)$ is a 2-regular spanning subgraph of $Q_{n/2}$ formed by the union of vertex-disjoint $n$-cycles having edge-direction sequence $S=(~1,2,\ldots,n/2,1,2,\ldots,n/2~)$ and their initial vertices belong to the subgroup $H=H_1$ of $G.$ As $H_i=\langle 2i-2\rangle H_1,$ $W_i=\sigma_{\langle 2i-2\rangle }(~W_1~).$ Therefore $W_i\Box W_i=\sigma_{\langle 2j-2\rangle \Delta~\theta(~\langle 2j-2\rangle ~)}(~W_1 \Box W_1~).$ But $W_i \Box W_i = F_i \sqcup F_i'$ for all $i \in \langle n/4\rangle .$ Therefore $F_i\sqcup F_i' = \sigma_{\langle 2i-2\rangle \Delta~\theta(~\langle 2i-2\rangle ~)}(~F_1\sqcup F_1'~).$ Similarly $F_j \sqcup F_j' = \sigma_{\langle 2j-2\rangle \Delta~\theta(~\langle 2j-2\rangle ~)}(~F_1 \sqcup F_1'~).$ In fact, $F_i = \sigma_{\langle 2i-2\rangle \Delta~\theta(~\langle 2i-2\rangle ~)}(~F_1~),$ and $F_i' = \sigma_{\langle 2i-2\rangle \Delta~\theta(~\langle 2i-2\rangle ~)}(~F_1'~),$ for all $i \leq n/4.$ Now $i<j$ implies $ \langle 2j-2\rangle =\langle 2i-2\rangle\Delta\{2i-1,2i,\ldots,2j-2\}.$ Therefore $F_j=\sigma_{\langle 2j-2\rangle \Delta~\theta(~\langle 2j-2\rangle ~)}(~F_1~)$ $=\sigma_{\{2i-1,\ldots,2j-2\}\Delta~\theta(~\{2i-1,\ldots,2j-2\}~)}~o~\sigma_{\langle 2i-2\rangle \Delta~\theta(~\langle 2i-2\rangle ~)}(~F_1~)$ $=\sigma_{\{2i-1,\ldots,2j-2\}\Delta~\theta(~\{2i-1,\ldots,2j-2\}~)}(~F_i~).$ Now $\{2i-1,2i,\ldots,2j-2\}\in H_k,$ for some $k\in \{1,2,\ldots,n/4\}.$ Therefore $\{2i-1,2i,\ldots,2j-2\}\Delta~\theta(~\{2i-1,2i,\ldots,2j-2\}~) \in (~F_k\sqcup F_k'~) = \sigma_{\langle 2k-2\rangle \Delta~\theta(~\langle 2k-2\rangle ~)}(~F_1\sqcup F_1'~).$ So one can say that $F_i, F_j$ differ by $F_k$ or $F_k'.$ Now $M_i$ and $M_j$ are the matchings selected from the cycles in $F_i$ and $F_j$ respectively. It suffices to prove that $M_1 \cup M_i$ and $M_1 \cup M_j$ are matchings for showing that $M_i \cup M_j$ is a matching. For, if $M_i \cup M_j$ is not a matching, then $M_1 \cup M_k \cup M_k'$ will not be a matching, which contradicts Lemma 5.2 and Lemma 5.5.
\end{proof}

We now prove that $M' = M_1' \cup M_2' \cup \ldots \cup M_{n/4}',$ where $M_i'$ is the collection of the edges selected from the cycles of $F_i'$ in Section 3, is a matching. 

\begin{lemma}
 $M'=M_1'\cup M_2'\cup\ldots \cup M_{n/4}'$ is a matching in $Q_n.$
\end{lemma}

\begin{proof}
The cycles of $F_i'$ are mutually vertex-disjoint and $M_i'$ contains the edges selected from each cycle of $F_i'$ satisfying condition (~I~) of Theorem 3.1. Therefore $M_i'$ is a matching for each $i=1,2,\ldots,n/4.$ As in the proof of the above lemma, for $M'$ to be a matching, it suffices to prove that $M_1'\cup M_j'$ is a matching. Let $2 \leq j \leq n/4.$ To prove $M_1'\cup M_j'$ is a matching, it is enough to prove that any edge in $M_1'$ selected from the first cycle $\Gamma^1_\emptyset=\Gamma(~\emptyset,\Delta~\theta(~\mathcal S~)~)$ in $F_1'$ is vertex disjoint with any edge of $M_j'$ selected from an arbitrary cycle in $F_j'.$ 

Recall that $M_1'$ contains the following $n$ edges of the cycle $\Gamma^1_\emptyset=\Gamma(~\emptyset,\Delta~\theta(~\mathcal S~)~):$ \\
\begin{tabular}{lll}
$f^1_r$&$=$&$ (~\langle r-1\rangle \Delta~\theta(~\langle n/2-r+1\rangle ~), \langle r-1\rangle \Delta~\theta(~\langle n/2-r+2\rangle ~)~), $\\
$f^1_{n/2+r}$&$=$&$ (~\overline{\langle r-1\rangle }\Delta~\theta(~\overline{\langle n/2-r+1\rangle }~),\overline{\langle r-1\rangle }\Delta~\theta(~\overline{\langle n/2-r+2\rangle }~)~),$\\&& where $r=1,2,\ldots,n/2$
\end{tabular}\\

Any cycle in $F_j'$ is $\Gamma^j_{A\Delta~\theta(~B~)}=\Gamma(~A\Delta\langle 2j-2\rangle \Delta~\theta(~\langle 2j-2\rangle \Delta B~)~)$ for some  $A, B \in H.$ 
The $n$ edges of this cycle that belong to $M_j'$ are as follows.\\
\begin{tabular}{lll}
$f^j_{s_{A\Delta~\theta(~B~)}}$&$=$&$ (~A\Delta\langle 2j-2\rangle\Delta \langle s-1\rangle \Delta~\theta(~\langle 2j-2\rangle\Delta \langle n/2-s+2j-1\rangle \Delta B~),$\\&&$A\Delta\langle 2j-2\rangle\Delta \langle s-1\rangle \Delta~\theta(~\langle 2j-2\rangle\Delta \langle n/2-r+2j\rangle \Delta B~)~), $\\
$f^j_{n/2+s_{A\Delta~\theta(~B~)}}$&$=$&$ (~A\Delta\langle 2j-2\rangle\Delta\overline{ \langle s-1\rangle }\Delta~\theta(~\langle 2j-2\rangle\Delta\overline{ \langle n/2-s+2j-1\rangle }\Delta B~),$\\&&$A\Delta\langle 2j-2\rangle\Delta\overline{ \langle s-1\rangle }\Delta~\theta(~\langle 2j-2\rangle\Delta\overline{ \langle n/2-s+2j\rangle }\Delta B~)~), $\\&&where $s=2j-1,2j,\ldots,n/2+2j-2.$
\end{tabular}\\

 Thus we need to prove only that the edges $f^1_r,f^1_{n/2+r}, f^j_{s_{A\Delta~\theta(~B~)}}, f^j_{n/2+s_{A\Delta~\theta(~B~)}}$ are all vertex-disjoint for any $1 \leq r \leq n/2,$ $2j-1 \leq s \leq n/2+2j-2$ and $A,B \in H$.\\

Suppose the first vertex $\langle r-1\rangle \Delta~\theta(~\langle n/2-r+1\rangle ~)$ of $f^1_r$ and the first vertex $A\Delta\langle 2j-2\rangle $ $\langle s-1\rangle \Delta~\theta(~\langle n/2-s+2j-1\rangle\Delta \langle 2j-2\rangle \Delta B~)$ of $f^j_{s_{A\Delta~\theta(~B~)}}$ are same. Then, by the definition of symmetric difference and by noting that $C \cap \Delta~\theta(~D~)=\emptyset$ for any $C, D \subset \langle n/2\rangle ,$ $$A=\langle 2j-2\rangle\Delta \langle r-1\rangle\Delta \langle s-1\rangle ,$$ $$B=\langle 2j-2\rangle\Delta \langle n/2-r+1\rangle\Delta \langle n/2-s+2j-1\rangle .$$ Then it follows from the proof of Lemma 5.1 that $A \notin H$ or $B \notin H,$ a contradiction.

Similarly, as in the proof of Lemma 5.2, we see that any vertex of $f^1_r$ or $f^1_{n/2+r}$ is different from any vertex of both $f^j_{s_{A\Delta~\theta(~B~)}}$ and $f^j_{n/2+s_{A\Delta~\theta(~B~)}}.$ 

Therefore $M'$ is a matching.

\end{proof}

\begin{lemma}
$\mathcal M=M \cup  M'$ forms a perfect matching of $Q_n.$
\end{lemma}
\begin{proof}
By Lemma 5.3 and Lemma 5.4, $M$ and $M'$ are the matchings in $Q_n.$ The edges of $M$ and $M'$ are selected from the cycles in $F_i$'s and $F_i'$'s respectively. Both $M$ and $M'$ satisfy condition (~I~). Note that $F_i,F_i'$ are both $2$-regular, spanning subgraphs of $Q_n.$ Edge direction sequence of cycles in $F_i$ is $\mathcal S,$ while the edge direction sequence of cycles in $F_i'$ is $\Delta~\theta(~\mathcal S~).$  Also, vertex set of the cycle $\Phi^i_{A\Delta~\theta(~B~)}=\Phi(~A\Delta\langle 2i-2\rangle \Delta~\theta(~B\Delta\langle 2i-2\rangle ~),\mathcal S~)$ in $F_i$ is same as the vertex set of the cycle $\Gamma^i_{A\Delta~\theta(~B~)}=\Gamma(~A\Delta\langle 2i-2\rangle \Delta~\theta(~B\Delta\langle 2i-2\rangle ~),\Delta~\theta(~\mathcal S~)~)$ in $F_i',$ for any $A,B \in H.$ So while selecting the edges from a cycle $\Phi^i_{A\Delta~\theta(~B~)}$ in $F_i,$ we start with $2i-1^{st}$ edge and that from a cycle $\Gamma^i_{A\Delta~\theta(~B~)}$ in $F_i'$ we start with $n/2+2i-1^{st}$ edge. This selection assures the vertex-disjointness of the edges selected from $\Phi^i_{A\Delta~\theta(~B~)}$ with those selected from $\Gamma^i_{A\Delta~\theta(~B~)}.$ Moreover, for fixed $i$ the cycles in $F_i$ are all vertex-disjoint. Same is the situation for the cycles in $F_i'.$ Therefore it suffices to prove that the edges selected from a cycle $\Phi^1_{\emptyset\emptyset}=\Phi(~\emptyset,\mathcal S~)$ in $F_1$ are vertex-disjoint with the edges selected from an arbitrary cycle $\Gamma^j_{A\Delta~\theta(~B~)}=\Gamma(~A\Delta\langle 2j-2\rangle \Delta~\theta(~B\Delta\langle 2j-2\rangle ~),\Delta~\theta(~\mathcal S~)~)$ in $F_j',$ $j \in \{2,\ldots,n/4\}.$

Now the edges selected from $\Phi^1_\emptyset$ and  $\Gamma^j_{A\Delta~\theta(~B~)}$ are as follows.\\
\begin{tabular}{lll}
$e^1_r$&$=$&$(~\overline{\langle n/2-r+1\rangle }\Delta~\theta(~\langle r-1\rangle ~),\overline{\langle n/2-r+2\rangle }\Delta~\theta(~\langle r-1\rangle ~)~),$\\
$e^1_{n/2+r}$&$=$&$(~\langle n/2-r+1\rangle \Delta~\theta(~\overline{\langle r-1\rangle }~),\langle n/2-r+2\rangle \Delta~\theta(~\overline{\langle r-1\rangle }~)~),$\\&& where $r=1,2,\ldots,n/2.$\\
$f^j_{s_{A\Delta~\theta(~B~)}}$&$=$&$ (~A\Delta\langle 2j-2\rangle\Delta \langle s-1\rangle \Delta~\theta(~\langle 2j-2\rangle\Delta \langle n/2-s+2j-1\rangle \Delta B~),$\\&&$A\Delta\langle 2j-2\rangle\Delta \langle s-1\rangle \Delta~\theta(~\langle 2j-2\rangle\Delta \langle n/2-s+2j\rangle \Delta B~)~), $\\
$f^j_{n/2+s_{A\Delta~\theta(~B~)}}$&$=$&$ (~A\Delta\langle 2j-2\rangle\Delta\overline{ \langle s-1\rangle }\Delta~\theta(~\langle 2j-2\rangle\Delta\overline{ \langle n/2-s+2j-1\rangle }B~),$\\&&$A\Delta\langle 2j-2\rangle\Delta\overline{ \langle s-1\rangle }\Delta~\theta(~\langle 2j-2\rangle\Delta\overline{ \langle n/2-s+2j\rangle }B~)~), $\\&&where $s=2j-1,2j,\ldots,n/2+2j-2.$
\end{tabular}\\
\emph{Claim:} $e^1_r,e^1_{n/2+r}, f^j_{s_{A\Delta~\theta(~B~)}}, f^j_{n/2+s_{A\Delta~\theta(~B~)}}$ are all vertex-disjoint.\\
As in the proof of Lemma 5.2, we call $e^1_r=(~X_1,X_2~), e^1_{n/2+r}=(~X_3,X_4~), f^j_{s_{A\Delta~\theta(~B~)}}=(~Y_1,Y_2~)$ and $f^j_{n/2+s_{A\Delta~\theta(~B~)}}=(~Y_3,Y_4~)$ for convenience. We prove that $X_p \neq Y_1$ for all $p \in \{1,2,3,4\}.$ The remaining cases $X_p \neq Y_q, p=1,2,3,4, q=2,3,4$ follow along lines similar to those in the proof of Lemma 5.2. 

(~1~). Suppose if possible $X_1 = Y_1.$ Then $\overline{\langle n/2-r+1\rangle }\Delta~\theta(~\langle r-1\rangle ~)=A\Delta\langle 2j-2\rangle $ $\langle s-1\rangle \Delta~\theta(~\langle 2j-2\rangle\Delta \langle n/2-s+2j-1\rangle \Delta B~)$ for some $A,B \in H,$ $r=1,2,\ldots,n/2, s=2j-1,\ldots,n/2+2j-2$ and $j \geq 2.$  Therefore $A = \langle 2j-2\rangle\Delta \langle s-1\rangle\Delta\overline{ \langle n/2-r+1\rangle }$ and $B=\langle 2j-2\rangle\Delta \langle r-1\rangle\Delta \langle n/2-s+2j-1\rangle .$ It is easy to observe that if $A=\emptyset,$ then $B \notin H,$ a contradiction. Similarly, if $B=\emptyset$ then $A \notin H,$ a contradiction. Therefore both $A$ and $B$ are non-empty. Observe that $\overline{\langle n/2-r+1\rangle }$ contains $n/2$ if $r \geq 2$ and is $\emptyset$ if $r=1.$ By notations given in the proof of Theorem 3.1, if $s=n/2+k\geq n/2+2$ then $\langle s-1\rangle =\langle n/2\rangle\Delta \langle k-1\rangle =\{k,k+1,\ldots,n/2\}.$ 

If $r=1,$ $B$ is a consecutive string, a contradiction by Lemma 2.3(2).

Therefore $r \geq 2.$

If $2j-2 \leq s \leq n/2,$ then $A$ contains $n/2,$ a contradiction.

If $s=n/2+1,$ then $A$ is consecutive, a contradiction.

Suppose $n/2+2 \leq s \leq n/2+2j-2.$ Then the symmetric difference of $A$ and $B$ is $A\Delta B =$ $\langle r-1\rangle\Delta \langle n/2-s+2j-1\rangle\Delta \langle s-1\rangle\Delta\overline{ \langle n/2-r+1\rangle }=\langle r-1\rangle\Delta \langle s-1\rangle\Delta \langle n/2-s'+1\rangle \Delta \overline{\langle n/2-r+1\rangle }=\{r, \ldots, s-1\} \Delta \{n/2-r+2,\ldots,n/2-s'+1\},$ where $s'=s-(~2j-2~).$ By the similar arguments used in the Case(~i~) of the proof of Lemma 5.1, we get a contradiction. 

Therefore $X_1 \neq Y_1.$

(~2~). Suppose $X_2=Y_1.$ Then $\overline{\langle n/2-r+2\rangle }\Delta~\theta(~\langle r-1\rangle ~)=A\Delta\langle 2j-2\rangle\Delta \langle s-1\rangle \Delta~\theta(~\langle 2j-2\rangle\Delta \langle n/2-s+2j-1\rangle \Delta B~),$ where $A,B \in H$ are non-empty, $r=1,2,\ldots,n/2, s=2j-1,\ldots,n/2+2j-2$ and $j \geq 2.$  Therefore $A = \langle 2j-2\rangle\Delta \langle s-1\rangle\Delta\overline{ \langle n/2-r+2\rangle }$ and $B=\langle 2j-2\rangle\Delta \langle r-1\rangle\Delta \langle n/2-s+2j-1\rangle .$ 

If $r=1,$ then $B$ is a consecutive string, a contradiction. 

If $r=2,$ then $A$ is either a consecutive string or contains $n/2,$ a contradiction.

Therefore $r \geq 3.$ 

Suppose $2j-1 \leq s \leq n/2+1$ then either $A$ is a consecutive string or contains $n/2,$ a contradiction. 

Suppose $n/2+2 \leq s \leq n/2+2j-2.$ Put $s-1 = n/2+k$ in $A.$ Therefore $A = \langle 2j-2\rangle $ $\langle s-1\rangle\Delta\overline{ \langle n/2-r+2\rangle }=\langle 2j-2\rangle\Delta \langle n/2+k\rangle\Delta\overline{ \langle n/2-r+2\rangle }=\langle 2j-2\rangle\Delta\overline{ \langle k\rangle }\Delta\overline{\langle n/2-r+2\rangle }=\langle 2j-2\rangle\Delta\{k+1,\ldots,n/2\}\Delta\{n/2-r+3,\ldots,n/2\}.$ If $k+1 <  n/2-r+3,$  then $A=\langle 2j-2\rangle\Delta\{k+1,\ldots,n/2-r+2\},$ and if $k+1> n/2-r+3,$  then $A=\langle 2j-2\rangle\Delta\{n/2-r+3,\ldots,k\}.$ In either of the cases, $A=A_1 A_2,$  where $|A_1|=2j-2$ and $|A_2|=n/2-k-r+2.$ Similarly $B=\langle 2j-2\rangle\Delta \langle r-1\rangle\Delta \langle n/2-s+2j-1\rangle =\langle 2j-2\rangle\Delta\{r,\ldots,n/2-s+2j-1\}=B_1 B_2,$ where $|B_1|=2j-2$ and $|B_2|=n/2-s-r+2j.$ Note that both $|A_2|$ and $|B_2|$ must be even. Therefore $r,k,s$ must be of the same parity. However, $k=s-1-n/2$ imply that $k$ and $s$ are of different parities as $n/2$ is even, a contradiction. 

Therefore $X_2 \neq Y_1.$

(~3~). Suppose $X_3=Y_1.$ Then $A=\langle 2j-2\rangle\Delta \langle s-1\rangle\Delta \langle n/2-r+1\rangle $ and $B=\langle 2j-2\rangle $ $\overline{\langle r-1\rangle}\Delta \langle n/2-s+2j-1\rangle .$ If $s=2j-1,$ then $A$ is a consecutive string, a contradiction. For any other values of $s,$ $B$ contains $n/2,$ a contradiction. Therefore $X_3 \neq Y_1.$

(~4~). Suppose $X_4=Y_1.$ Then $A=\langle 2j-2\rangle\Delta \langle s-1\rangle\Delta \langle n/2-r+2\rangle $ and $B=\langle 2j-2\rangle $ $\overline{\langle r-1\rangle}\Delta \langle n/2-s+2j-1\rangle .$ If $s=2j-1,$ then $A$ is a consecutive string. Otherwise $B$ contains $n/2,$ a contradiction. Therefore $X_4 \neq Y_1.$

Thus $\mathcal M$ is matching of $Q_n.$

Now $$|\mathcal M|=|M|+|M'|=|M_1|+|M_2|+\ldots+|M_{n/4}|+|M_1'|+|M_2'|+\ldots+|M_{n/4}'|$$
$$=n/4\times|M_1|+n/4\times |M_1'|= n/2 \times |M_1|.$$

But $|M_1|=n \times$ number of $2^mn$-cycles in $F_1.$ Recall from Lemma 2.3(2) that $|H|=2^{n/2-m}.$ Therefore number of $2^mn$-cycles in $Q_n$ is $2^{n/2-m}\times 2^{n/2-m}.$ So $$|M_1|=n \times 2^{n/2-m}\times 2^{n/2-m}=2^m \times 2^{n-2m}=2^{n-m}.$$

Therefore, $$|\mathcal M|=n/2 \times 2^{n-m}=2^{m-1}\times 2^{n-m}=2^{n-1}.$$

Hence, $\mathcal M $ forms a perfect matching of $Q_n.$
\end{proof}

\textbf {The proof of Theorem 3.1 gets completed from Lemma 5.5.}\\

\textbf {This completes the proof of Main Theorem 1.6.}\\ 

\centerline {\textbf{Concluding Remark}}

In this paper, we considered the problem of decomposing the hypercube $Q_n$ into $2^mn$-cycles.
We constructed such decompositions for  $n \geq 2^m.$ However, the problem remains open for $n < 2^m.$

\centerline {\textbf {Acknowledgments}}

 The first and third authors would like to thank  DST-SERB, Government of India for the financial support under the project SR/S4/MS:750/12.


 \newpage
\centerline{\textbf{\large{Appendix}}}
Here we give an illustration of Theorem 3.1 by constructing  $64$-cycles in $Q_8$ and selecting edges from the cycles to form a perfect matching of $Q_8.$
\vskip.2cm\noindent
\textbf{(~I~). Construction of $64$-cycles in the decomposition of $Q_8.$}\\ We know that $Q_8=Q_4\Box Q_4.$ Here $n=2^3,$ $n/2=2^2.$ Consider $Q_4.$ Its vertex set is $\mathcal P(\{1,2,3,4\}).$ As per notations of Section 2, we have $G=\{A\colon A\subset \{1,2,3\}, ~|A|~\textrm{even}\}=\{\emptyset,\{1,2\},\{1,3\},\{2,3\}\}.$ 
Here $K=\{\{1,3\}\}.$ The subgroup $H$ is generated by the symmetric difference of $2$-element subsets in $K.$ Thus $H=<\{1,3\}>=\{\emptyset,\{1,3\}\}$
The subgroup $H=\{\emptyset,\{1,3\}\}$ is the subgroup of $G$ and $H_1=H,$ $H_2=\{1,2\}\Delta H=\{\{1,2\},\{2,3\}\}$ are the cosets of $H$ in $G.$

The collection of edge-disjoint $8$-cycles which decompose $Q_4$ is given by $\{C(A,S)\colon A\in G\},$ where $S=(1,2,3,4,1,2,3,4)$ is the edge-direction sequence of the cycles. So $Q_4=\displaystyle\bigsqcup_{A\in G}C(A,S),$ $W_1=\displaystyle\biguplus_{A\in H_1}C(A,S)=C(\emptyset,S)\uplus C(\{1,3\},S)$ and $W_2=\displaystyle\biguplus_{A\in H_2}C(A,S)=C(\{1,2\},S)\uplus C(\{2,3\},S) .$ 
Now $Q_8 = Q_{4} \Box Q_{4} = (W_1 \sqcup W_2 ) \Box (W_1 \sqcup W_2)=(W_1 \Box W_1) \sqcup (W_2 \Box W_2)\textrm{(by Lemma 2.8(1))} .$ Hence
$ W_1 \Box W_1 =(~C(\emptyset,S)\Box C(\emptyset,S)~)\uplus(~C(\emptyset,S)\Box C(\{1,3\},S)~)\uplus(~C(\{1,3\},S)\Box C(\emptyset,S)~)\uplus(~C(\{1,3\},S)\Box C(\{1,3\},S)~)$ \\
$=(\Phi_{\emptyset\emptyset}\sqcup\Gamma_{\emptyset\emptyset})\uplus(\Phi_{\emptyset\{1,3\}}\sqcup\Gamma_{\emptyset\{1,3\}})\uplus(\Phi_{\{1,3\}\emptyset}\sqcup\Gamma_{\{1,3\}\emptyset})\uplus(\Phi_{\{1,3\}\{1,3\}}\sqcup\Gamma_{\{1,3\}\{1,3\}})$\\
$=(\Phi_{\emptyset\emptyset}\uplus\Phi_{\emptyset\{1,3\}}\uplus\Phi_{\{1,3\}\emptyset}\uplus\Phi_{\{1,3\}\{1,3\}})\sqcup(\Gamma_{\emptyset\emptyset}\uplus\Gamma_{\emptyset\{1,3\}}\uplus\Gamma_{\{1,3\}\emptyset}\uplus\Gamma_{\{1,3\}\{1,3\}})=F_1\sqcup F_1'.$\vskip.2cm\noindent
Similarly,
 $W_2\Box W_2=(\Phi_{\{1,2\}\{1,2\}}\uplus\Phi_{\{1,2\}\{2,3\}}\uplus\Phi_{\{2,3\}\{1,2\}}\uplus\Phi_{\{2,3\}\{2,3\}}) \sqcup$\\ $(\Gamma_{\{1,2\}\{1,2\}}\uplus\Gamma_{\{1,2\}\{2,3\}}\uplus\Gamma_{\{2,3\}\{1,2\}}\uplus \Gamma_{\{2,3\}\{2,3\}})=F_2\sqcup F_2'.$ \vskip.2cm\noindent
 Thus, $Q_8=F_1\sqcup F_2\sqcup F_1'\sqcup F_2'\\=\{\Phi_{AB}\colon A,~B ~\in~ H_1\}~\sqcup~\{\Phi_{AB}\colon A,~B ~\in~ H_2\}~\sqcup~\{\Gamma_{AB}\colon A,~B ~\in~ H_1\}~\sqcup~\{\Gamma_{AB}\colon A,~B ~\in~ H_2\}$
 is a decomposition of $Q_8$ into $64$-cycles. These cycles are constructed explicitly in the figures at the end. The initial vertex and the edge direction sequence of each of these cycles is shown in the figures. 
 
  Note that the above cycles $\Phi_{AB}$ and $\Gamma_{AB}$ are each of length $64$ and the subscripts $AB$ stand for their initial vertices $A\Delta\theta(B).$ Further, all the cycles in the $2$-regular, spanning subgraphs $F_1$ and $F_2$ of $Q_8$ have edge-direction sequence $\mathcal S$ while all the cycles in $F_1'$ and $F_2'$ have edge-direction sequence  $\theta(\mathcal S),$ where $\mathcal S$ and $\theta(\mathcal S)$ are as follows.\\ \\
\small
\begin{tabular}{rcl}
$\mathcal S$&$=$&$(~~1, 2, 3, 4, 1,2,3, \theta(1),$\\
&&$4, 1, 2,3, 4, 1, 2, \theta(2),$\\
&&$3,4,1,2,3,4,1,\theta(3),$\\
&&$2,3,4,1,2,3,4,\theta(4))$\\
&&$(1, 2, 3, 4, 1,2,3, \theta(1),$\\
&&$4, 1, 2,3, 4, 1, 2, \theta(2),$\\
&&$3,4,1,2,3,4,1,\theta(3),$\\
&&$2,3,4,1,2,3,4,\theta(4)~~).$
\end{tabular}
\begin{tabular}{rcl}
$\theta(\mathcal S)$&$=$&$(~~\theta(1), \theta(2), \theta(3), \theta(4), \theta(1),\theta(2),\theta(3),1,$\\
&&$\theta(4), \theta(1), \theta(2),\theta(3), \theta(4), \theta(1), \theta(2), 2,$\\
&&$\theta(3),\theta(4),\theta(1),\theta(2),\theta(3),\theta(4),\theta(1),3,$\\
&&$\theta(2),\theta(3),\theta(4),\theta(1),\theta(2),\theta(3),\theta(4),4,$\\
&&$(\theta(1), \theta(2), \theta(3), \theta(4), \theta(1),\theta(2),\theta(3),1,$\\
&&$\theta(4), \theta(1), \theta(2),\theta(3), \theta(4), \theta(1), \theta(2), 2,$\\
&&$\theta(3),\theta(4),\theta(1),\theta(2),\theta(3),\theta(4),\theta(1),3,$\\
&&$\theta(2),\theta(3),\theta(4),\theta(1),\theta(2),\theta(3),\theta(4),4~~).$
\end{tabular}\\
\normalsize
\vskip.2cm\noindent
\noindent
\textbf{(~II~). Selection of edges to form a perfect matching of $Q_8.$} \\We select the edges as per the scheme given in the proof of Theorem 3.1 to form a perfect matching of $Q_8,$ satisfying conditions (I) and (II) of this theorem.
\vskip.2cm\noindent
\textbf{(1). Edges from cycles in $F_1.$}  \\
\begin{tabular}{lll}
$e^1_{r_{AB}}$&$=$&$ (~A\Delta\overline{\langle 4-r+1\rangle} \Delta~\theta(~B\Delta\langle r-1\rangle ~), A\Delta\overline{\langle 4-r+2\rangle }\Delta~\theta(~B\Delta\langle r-1\rangle ~)~), $\\
$e^1_{{4+r}_{AB}}$&$=$&$ (~A\Delta\langle 4-r+1\rangle\Delta~\theta(~B\Delta\overline{\langle r-1\rangle }~), A\Delta\langle 4-r+2\rangle \Delta~\theta(~B\Delta\overline{\langle r-1\rangle }~)~),$\\&& where $r=1,2,3,4$ and $A,B\in H_1=\{\emptyset,\{1,3\}\}.$
\end{tabular} \\
We give the following table to show some calculations.\\
\small
\begin{tabular}{|c|c|c|c|c|c|c|c|c|}
\hline&&&&&&&&\\ $r$&$\langle r-1\rangle$&$\overline{\langle r-1\rangle}$&$\langle 4-r+1\rangle$&$\overline{\langle 4-r+1\rangle}$&$\langle 4-r+2\rangle$&$\overline{\langle 4-r+2\rangle}$&$\theta(\langle r-1\rangle)$&$\theta(\overline{\langle r-1\rangle})$\\
\hline
$1$&$\emptyset$&$\{1,2,3,4\}$&$\{1,2,3,4\}$&$\emptyset$&$\{2,3,4\}$&$\{1\}$&$\emptyset$&$\{5,6,7,8\}$\\
$2$&$\{1\}$&$\{2,3,4\}$&$\{1,2,3\}$&$\{4\}$&$\{1,2,3,4\}$&$\emptyset$&$\{5\}$&$\{6,7,8\}$\\
$3$&$\{1,2\}$&$\{3,4\}$&$\{1,2\}$&$\{3,4\}$&$\{1,2,3\}$&$\{4\}$&$\{5,6\}$&$\{7,8\}$\\
$4$&$\{1,2,3\}$&$\{4\}$&$\{1\}$&$\{2,3,4\}$&$\{1,2\}$&$\{3,4\}$&$\{5,6,7\}$&$\{8\}$\\
\hline
\end{tabular}
\vskip.2cm \noindent
\normalsize So the edges selected from the cycles in $F_1$ are explicitly given in the following tables.\\
\small
\begin{tabular}{|c|c|c|}
\hline &$\Phi_{\emptyset\emptyset}$&$\Phi_{\emptyset\{1,3\}}$\\
\hline $e^1_1$&$(~\emptyset,~\{1\}~)$&$(~\{5,7\},~\{1,5,7\}~)$\\
 $e^1_2$&$(~\{4,5\},~\{5\}~)$&$(~\{4,7\},~\{7\}~)$\\
 $e^1_3$&$(~\{3,4,5,6\},~\{4,5,6\}~)$&$(~\{3,4,6,7\},~\{4,6,7\}~)$\\
 $e^1_4$&$(~\{2,3,4,5,6,7\},~\{3,4,5,6,7\}~)$&$(~\{2,3,4,6\},~\{3,4,6\}~)$\\
   $e^1_{4+1}$&$(~\{1,2,3,4,5,6,7,8\},~\{2,3,4,5,6,7,8\}~)$&$(~\{1,2,3,4,6,8\},~\{2,3,4,6,8\}~)$\\
 $e^1_{4+2}$&$(~\{1,2,3,6,7,8\},~\{1,2,3,4,6,7,8\}~)$&$(~\{1,2,3,5,6,8\},~\{1,2,3,4,5,6,8\}~)$\\
 $e^1_{4+3}$&$(~\{1,2,7,8\},~\{1,2,3,7,8\}~)$&$(~\{1,2,5,8\},~\{1,2,3,5,8\}~)$\\
 $e^1_{4+4}$&$(~\{1,8\},~\{1,2,8\}~)$&$(~\{1,5,7,8\},~\{1,2,5,7,8\}~)$\\
 \hline
\end{tabular} \vskip.2cm
\begin{tabular}{|c|c|c|}
\hline &$\Phi_{\{1,3\}\emptyset}$&$\Phi_{\{1,3\}\{1,3\}}$\\
\hline $e^1_1$&$(~\{1,3\},~\{3\}~)$&$(~\{1,3,5,7\},~\{3,5,7\}~)$\\
 $e^1_2$&$(~\{1,3,4,5\},~\{1,3,5\}~)$&$(~\{1,3,4,7\},~\{1,3,7\}~)$\\
 $e^1_3$&$(~\{1,4,5,6\},~\{1,3,4,5,6\}~)$&$(~\{1,4,6,7\},~\{1,3,4,6,7\}~)$\\
 $e^1_4$&$(~\{1,2,4,5,6,7\},~\{1,4,5,6,7\}~)$&$(~\{1,2,4,6\},~\{1,4,6\}~)$\\
   $e^1_{4+1}$&$(~\{2,4,5,6,7,8\},~\{1,2,4,5,6,7,8\}~)$&$(~\{2,4,6,8\},~\{1,2,4,6,8\}~)$\\
 $e^1_{4+2}$&$(~\{2,6,7,8\},~\{2,4,6,7,8\}~)$&$(~\{2,5,6,8\},~\{2,4,5,6,8\}~)$\\
 $e^1_{4+3}$&$(~\{2,3,7,8\},~\{2,7,8\}~)$&$(~\{2,3,5,8\},~\{2,5,8\}~)$\\
 $e^1_{4+4}$&$(~\{3,8\},~\{2,3,8\}~)$&$(~\{3,5,7,8\},~\{2,3,5,7,8\}~)$\\
 \hline
\end{tabular} 
\vskip.2cm \noindent
\normalsize
\textbf{(2). Edges from cycles in $F_2.$} We now select the edges from the cycles in $F_2,$ that contribute in the perfect matching of $Q_8.$ \\
\begin{tabular}{lll}
$e^i_{r_{A\theta(~B~)}}$&$=$&$ (~A\Delta\langle 2\rangle\Delta \overline{\langle 4-r+3\rangle}\Delta~\theta(~\langle r-1\rangle\Delta \langle 2\rangle\Delta B~),~A\Delta\langle 2\rangle\Delta \overline{\langle 4-r+4\rangle}\Delta~\theta(~\langle r-1\rangle\Delta \langle 2\rangle\Delta B~)~), $\\
$e^i_{4+r_{A\theta(~B~)}}$&$=$&$(~ A\Delta\langle 2\rangle \Delta\langle 4-r+3\rangle \Delta~\theta(~\overline{\langle r-1\rangle }\Delta\langle 2\rangle\Delta B~),~ A\Delta\langle 2\rangle \Delta\langle 4-r+4\rangle \Delta~\theta(~\overline{\langle r-1\rangle }\Delta\langle 2\rangle \Delta B~)~),$\\&& where $r=3,4,5,6,$ and $ A,B \in H.$
\end{tabular} \\
Again we provide a table of calculations.\\
\small
\begin{tabular}{|c|c|c|c|c|c|c|c|c|}
\hline&&&&&&&&\\ $r$&$\langle r-1\rangle$&$\overline{\langle r-1\rangle}$&$\langle 4-r+3\rangle$&$\overline{\langle 4-r+3\rangle}$&$\langle 4-r+4\rangle$&$\overline{\langle 4-r+4\rangle}$&$\theta(\langle r-1\rangle)$&$\theta(\overline{\langle r-1\rangle})$\\
\hline
$3$&$\{1,2\}$&$\{3,4\}$&$\{1,2,3,4\}$&$\emptyset$&$\{2,3,4\}$&$\{1\}$&$\{5,6\}$&$\{7,8\}$\\
$4$&$\{1,2,3\}$&$\{4\}$&$\{1,2,3\}$&$\{4\}$&$\{1,2,3,4\}$&$\emptyset$&$\{5,6,7\}$&$\{8\}$\\
$5$&$\{1,2,3,4\}$&$\emptyset$&$\{1,2\}$&$\{3,4\}$&$\{1,2,3\}$&$\{4\}$&$\{5,6,7,8\}$&$\emptyset$\\
$6$&$\{2,3,4\}$&$\{1\}$&$\{1\}$&$\{2,3,4\}$&$\{1,2\}$&$\{3,4\}$&$\{6,7,8\}$&$\{5\}$\\
\hline
\end{tabular}
\vskip.2cm \noindent
\normalsize Hence the edges selected from the cycles in $F_2$ are as follows.\\
\small
\begin{tabular}{|c|c|c|}
\hline &$\Phi_{\{1,2\}\{1,2\}}$&$\Phi_{\{1,2\}\{2,3\}}$\\
\hline $e^2_1$&$(~\{1,2\},~\{2\}~)$&$(~\{1,2,5,7\},~\{2,5,7\}~)$\\
 $e^2_2$&$(~\{1,2,4,7\},~\{1,2,7\}~)$&$(~\{1,2,4,5\},~\{1,2,5\}~)$\\
 $e^2_3$&$(~\{1,2,3,4,7,8\},~\{1,2,4,7,8\}~)$&$(~\{1,2,3,4,5,8\},~\{1,2,4,5,8\}~)$\\
 $e^2_4$&$(~\{1,3,4,5,7,8\},~\{1,2,3,4,5,7,8\}~)$&$(~\{1,3,4,8\},~\{1,2,3,4,8\}~)$\\
   $e^2_{4+1}$&$(~\{3,4,5,6,7,8\},~\{1,3,4,5,6,7,8\}~)$&$(~\{3,4,6,8\},~\{1,3,4,6,8\}~)$\\
 $e^2_{4+2}$&$(~\{3,5,6,8\},~\{3,4,5,6,8\}~)$&$(~\{3,6,7,8\},~\{3,4,6,7,8\}~)$\\
 $e^2_{4+3}$&$(~\{5,6\},~\{3,5,6\}~)$&$(~\{6,7\},~\{3,6,7\}~)$\\
 $e^2_{4+4}$&$(~\{2,6\},~\{6\}~)$&$(~\{2,5,6,7\},~\{5,6,7\}~)$\\
 \hline
\end{tabular} \vskip.2cm
\begin{tabular}{|c|c|c|}
\hline &$\Phi_{\{2,3\}\{1,2\}}$&$\Phi_{\{2,3\}\{2,3\}}$\\
\hline $e^2_1$&$(~\{2,3\},~\{1,2,3\}~)$&$(~\{2,3,5,7\},~\{1,2,3,5,7\}~)$\\
 $e^2_2$&$(~\{2,3,4,7\},~\{2,3,7\}~)$&$(~\{2,3,4,5\},~\{2,3,5\}~)$\\
 $e^2_3$&$(~\{2,4,7,8\},~\{2,3,4,7,8\}~)$&$(~\{2,4,5,8\},~\{2,3,4,5,8\}~)$\\
 $e^2_4$&$(~\{4,5,7,8\},~\{2,4,5,7,8\}~)$&$(~\{4,8\},~\{2,4,8\}~)$\\
   $e^2_{4+1}$&$(~\{1,4,5,6,7,8\},~\{4,5,6,7,8\}~)$&$(~\{1,4,6,8\},~\{4,6,8\}~)$\\
 $e^2_{4+2}$&$(~\{1,5,6,8\},~\{1,4,5,6,8\}~)$&$(~\{1,6,7,8\},~\{1,4,6,7,8\}~)$\\
 $e^2_{4+3}$&$(~\{1,3,5,6\},~\{1,5,6\}~)$&$(~\{1,3,6,7\},~\{1,6,7\}~)$\\
 $e^2_{4+4}$&$(~\{1,2,3,6\},~\{1,3,6\}~)$&$(~\{1,2,3,5,6,7\},~\{1,3,5,6,7\}~)$\\
 \hline
\end{tabular} 
\vskip.2cm \noindent
\normalsize
\textbf{(3). Edges from cycles in $F_1'.$}\\
\begin{tabular}{lll}
$f^i_{r_{A\theta(~B~)}}$&$=$&$ (~A\Delta\langle r-1\rangle \Delta~\theta(~ \langle 4-r+1\rangle\Delta B~),~A\Delta\langle r-1\rangle\Delta~\theta(~\langle 4-r+2\rangle \Delta B~)~), $\\
$f^i_{4+r_{A\theta(~B~)}}$&$=$&$ (~A\Delta\overline{\langle r-1\rangle }\Delta~\theta(~\overline{\langle 4-r+1\rangle }\Delta B~),~A\Delta\overline{\langle r-1\rangle }\Delta~\theta(~\overline{\langle 4-r+2\rangle }\Delta B~)~),$\\&& where $r=1,2,3,4$ and $A,B\in H.$
\end{tabular}\\
\vskip.2cm \noindent
So the edges selected from the cycles in $F_1'$ are as follows.\\
\small
\begin{tabular}{|c|c|c|}
\hline &$\Gamma_{\emptyset\emptyset}$&$\Gamma_{\emptyset\{1,3\}}$\\
\hline $f^1_1$&$(~\{5,6,7,8\},~\{6,7,8\}~)$&$(~\{6,8\},~\{5,6,8\}~)$\\
 $f^1_2$&$(~\{1,5,6,7\},~\{1,5,6,7,8\}~)$&$(~\{1,6,\},~\{1,6,8\}~)$\\
 $f^1_3$&$(~\{1,2,5,6\},~\{,1,2,5,6,7\}~)$&$(~\{1,2,6,7\},~\{1,2,6\}~)$\\
 $f^1_4$&$(~\{1,2,3,5\},~\{1,2,3,5,6\}~)$&$(~\{1,2,3,7\},~\{1,2,3,6,7\}~)$\\
   $f^1_{4+1}$&$(~\{1,2,3,4\},~\{1,2,3,4,5\}~)$&$(~\{1,2,3,4,5,7\},~\{1,2,3,4,7\}~)$\\
 $f^1_{4+2}$&$(~\{2,3,4,8\},~\{2,3,4\}~)$&$(~\{2,3,4,5,7,8\},~\{2,3,4,5,7\}~)$\\
 $f^1_{4+3}$&$(~\{3,4,7,8\},~\{3,4,8\}~)$&$(~\{3,4,5,8\},~\{3,4,5,7,8\}~)$\\
 $f^1_{4+4}$&$(~\{4,6,7,8\},~\{4,7,8\}~)$&$(~\{4,5,6,8\},~\{4,5,8\}~)$\\
 \hline
\end{tabular} \vskip.2cm
\begin{tabular}{|c|c|c|}
\hline &$\Gamma_{\{1,3\}\emptyset}$&$\Gamma_{\{1,3\}\{1,3\}}$\\
\hline $f^2_1$&$(~\{1,3,5,6,7,8\},~\{1,3,6,7,8\}~)$&$(~\{1,3,6,8\},~\{1,3,5,6,8\}~)$\\
 $f^2_2$&$(~\{3,5,6,7\},~\{3,5,6,7,8\}~)$&$(~\{3,6\},~\{3,6,8\}~)$\\
 $f^2_3$&$(~\{2,3,5,6\},~\{2,3,5,6,7\}~)$&$(~\{2,3,6,7\},~\{2,3,6\}~)$\\
 $f^2_4$&$(~\{2,5\},~\{2,5,6\}~)$&$(~\{2,7\},~\{2,6,7\}~)$\\
   $f^2_{4+1}$&$(~\{2,4\},~\{2,4,5\}~)$&$(~\{2,4,5,7\},~\{2,4,7\}~)$\\
 $f^2_{4+2}$&$(~\{1,2,4,8\},~\{1,2,4\}~)$&$(~\{1,2,4,5,7,8\},~\{1,2,4,5,7\}~)$\\
 $f^2_{4+3}$&$(~\{1,4,7,8\},~\{1,4,8\}~)$&$(~\{1,4,5,8\},~\{1,4,5,7,8\}~)$\\
 $f^2_{4+4}$&$(~\{1,3,4,6,7,8\},~\{1,3,4,7,8\}~)$&$(~\{1,3,4,5,6,8\},~\{1,3,4,5,8\}~)$\\
 \hline
\end{tabular}
\vskip.2cm \noindent
\normalsize
\textbf{(4). Edges from cycles in $F_2'.$}\\
\begin{tabular}{lll}
$f^i_{r_{A\theta(~B~)}}$&$=$&$ (~A\Delta\langle 2\rangle \Delta\langle r-1\rangle \Delta~\theta(~\langle 2\rangle\Delta \langle 4-r+3\rangle\Delta B~),~A\Delta\langle 2\rangle \Delta\langle r-1\rangle\Delta~\theta(~\langle 2\rangle\Delta \langle 4-r+4\rangle \Delta B~)~), $\\
$f^i_{4+r_{A\theta(~B~)}}$&$=$&$ (~A\Delta\langle 2\rangle \Delta\overline{\langle r-1\rangle }\Delta~\theta(~\langle 2\rangle \Delta\overline{\langle4-r+3\rangle }\Delta B~),~A\Delta\langle 2\rangle \Delta\overline{\langle r-1\rangle }\Delta~\theta(~\langle 2\rangle\Delta \overline{\langle 4-r+4\rangle }\Delta B~)~),$\\&& where $r=3,4,5,6$ and $A,B \in H.$
\end{tabular}\\
\small
\begin{tabular}{|c|c|c|}
\hline &$\Gamma_{\{1,2\}\{1,2\}}$&$\Gamma_{\{1,2\}\{2,3\}}$\\
\hline $f^2_1$&$(~\{7,8\},~\{5,7,8\}~)$&$(~\{5,8\},~\{8\}~)$\\
 $f^2_2$&$(~\{3,7\},~\{3,7,8\}~)$&$(~\{3,5\},~\{3,5,8\}~)$\\
 $f^2_3$&$(~\{3,4\},~\{3,4,7\}~)$&$(~\{3,4,5,7\},~\{3,4,5\}~)$\\
 $f^2_4$&$(~\{1,3,4,6\},~\{1,3,4\}~)$&$(~\{1,3,4,5,6,7\},~\{1,3,4,5,7\}~)$\\
   $f^2_{4+1}$&$(~\{1,2,3,4,5,6\},~\{1,2,3,4,6\}~)$&$(~\{1,2,3,4,6,7\},~\{1,2,3,4,5,6,7\}~)$\\
 $f^2_{4+2}$&$(~\{1,2,4,5,6,8\},~\{1,2,4,5,6\}~)$&$(~\{1,2,4,6,7,8\},~\{1,2,4,6,7\}~)$\\
 $f^2_{4+3}$&$(~\{1,2,5,6,7,8\},~\{1,2,5,6,8\}~)$&$(~\{1,2,6,8\},~\{1,2,6,7,8\}~)$\\
 $f^2_{4+4}$&$(~\{2,5,7,8\},~\{2,5,6,7,8\}~)$&$(~\{2,8\},~\{2,6,8\}~)$\\
 \hline
\end{tabular} \vskip.2cm
\begin{tabular}{|c|c|c|}
\hline &$\Gamma_{\{2,3\}\{1,2\}}$&$\Gamma_{\{2,3\}\{2,3\}}$\\
\hline $f^2_1$&$(~\{1,3,7,8\},~\{1,3,5,7,8\}~)$&$(~\{1,3,5,8\},~\{1,3,8\}~)$\\
 $f^2_2$&$(~\{1,7\},~\{1,7,8\}~)$&$(~\{1,5\},~\{1,5,8\}~)$\\
 $f^2_3$&$(~\{1,4\},~\{1,4,7\}~)$&$(~\{1,4,5,7\},~\{1,4,5\}~)$\\
 $f^2_4$&$(~\{4,6\},~\{4\}~)$&$(~\{4,5,6,7\},~\{4,5,7\}~)$\\
   $f^2_{4+1}$&$(~\{2,4,5,6\},~\{2,4,6\}~)$&$(~\{2,4,6,7\},~\{2,4,5,6,7\}~)$\\
 $f^2_{4+2}$&$(~\{2,3,4,5,6,8\},~\{2,3,4,5,6\}~)$&$(~\{2,3,4,6,7,8\},~\{2,3,4,6,7\}~)$\\
 $f^2_{4+3}$&$(~\{2,3,5,6,7,8\},~\{2,3,5,6,8\}~)$&$(~\{2,3,6,8\},~\{2,3,6,7,8\}~)$\\
 $f^2_{4+4}$&$(~\{1,2,3,5,7,8\},~\{1,2,3,5,6,7,8\}~)$&$(~\{1,2,3,8\},~\{1,2,3,6,8\}~)$\\
 \hline
\end{tabular}
 \vskip.2cm \noindent
\normalsize
One can check that the edges listed in all the tables above, together form a perfect matching of $Q_8.$ These edges are shown by solid lines/curves in the following figures.
\begin{figure}[h]
    %

\end{figure}

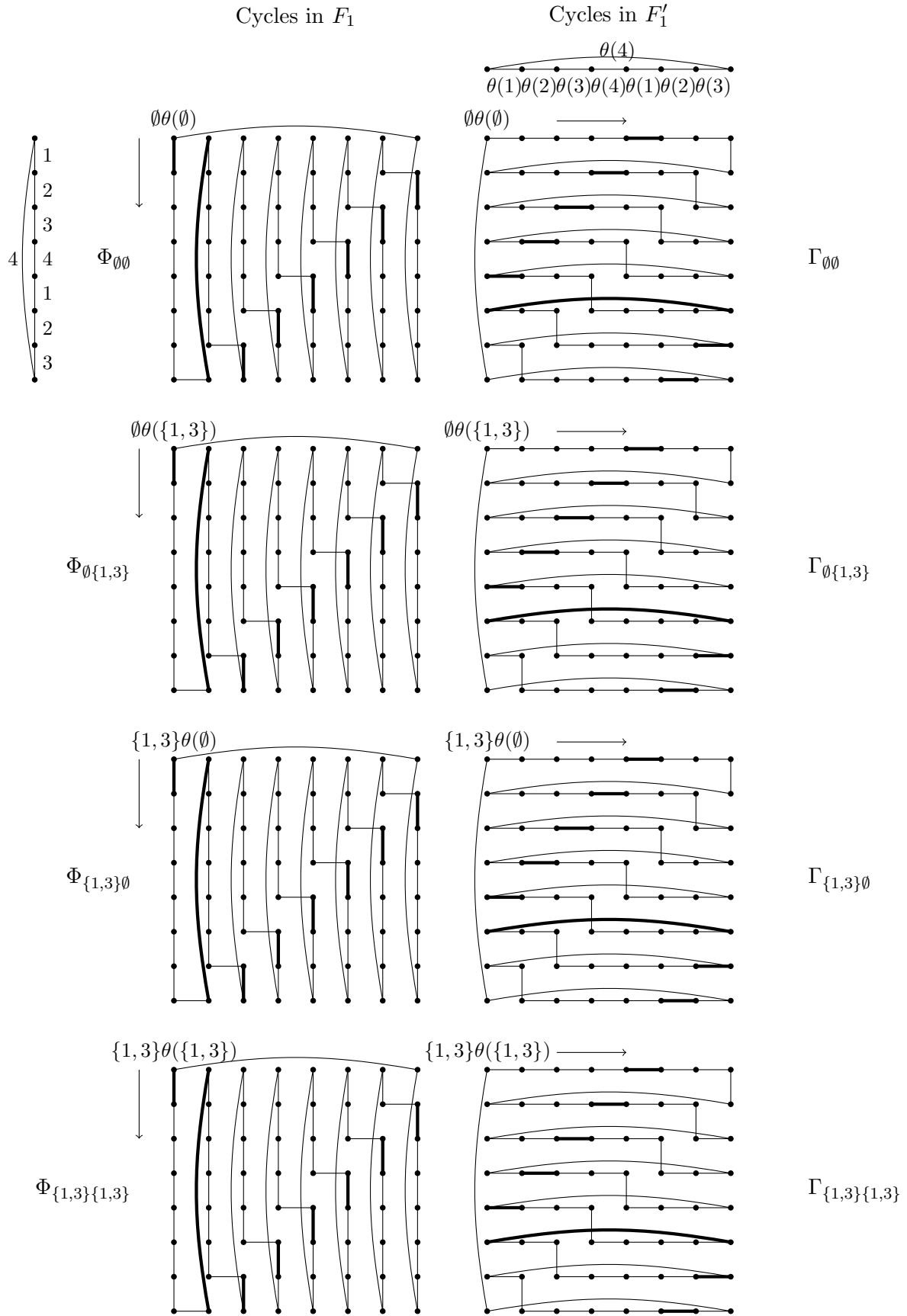
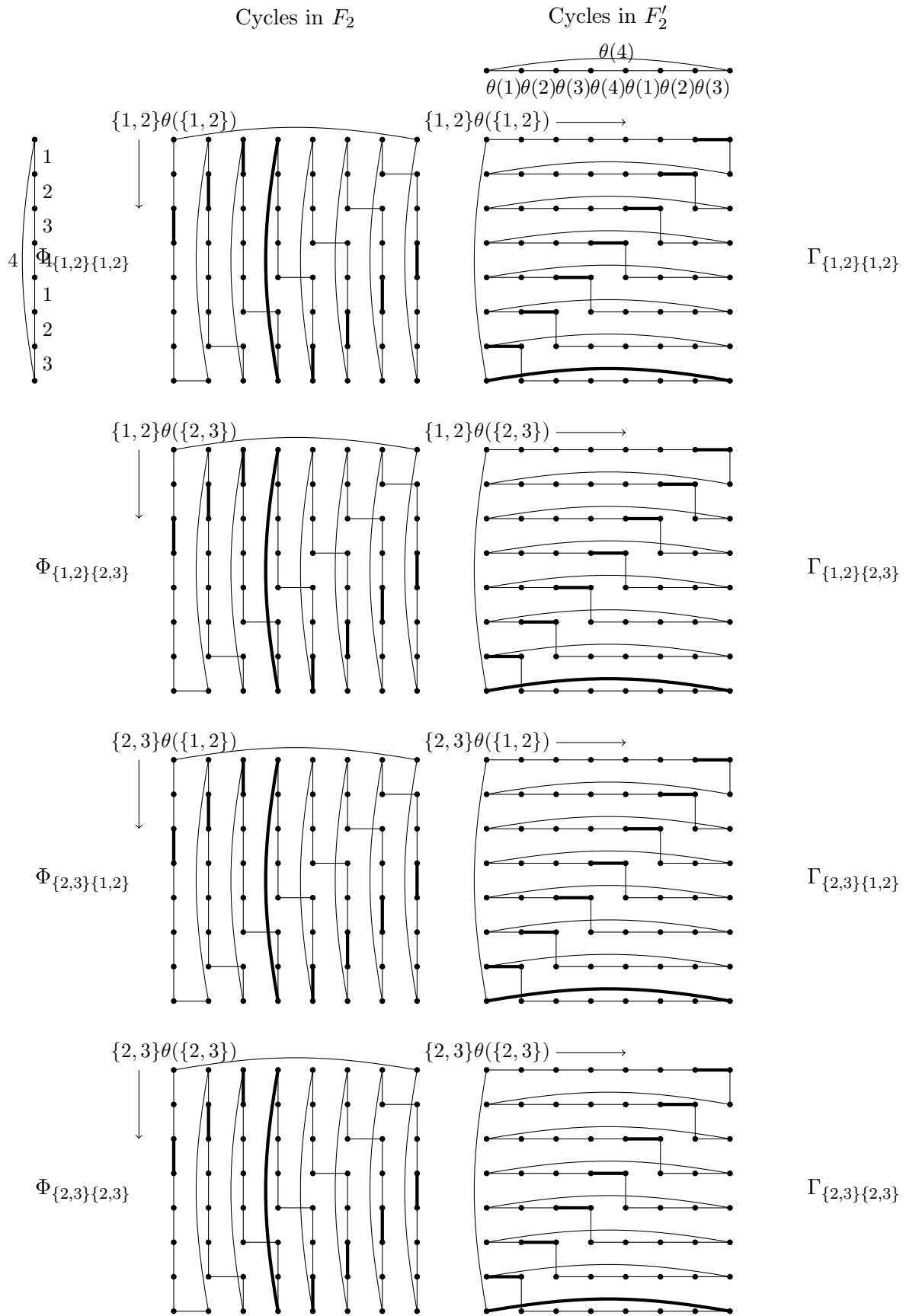
\end{document}